\documentclass{amsart}
\usepackage{amsmath,amssymb,amsthm,enumerate,graphicx,mathtools}

\usepackage[usenames, dvipsnames]{color}

\usepackage{xypic} 
\usepackage{mathrsfs} 
\usepackage{xspace}

\usepackage{hyperref}

\newtheorem{thm}{Theorem}[section] 
\newtheorem{fact}[thm]{Fact}
\newtheorem{obs}[thm]{Observation}
\newtheorem{prop}[thm]{Proposition}
\newtheorem{lem}[thm]{Lemma} 
\newtheorem{cor}[thm]{Corollary}

\theoremstyle{definition}
\newtheorem{defn}[thm]{Definition}
\newtheorem{quest}[thm]{Question}
\newtheorem{rmk}[thm]{Remark}

\newtheorem*{claim*}{Claim}

\newtheorem{example}[thm]{Example}

\newcommand{\mc}[1]{\mathcal{#1}}

\newcommand{\mf}[1]{\mathfrak{#1}}
\newcommand{\ms}[1]{\mathscr{#1}} 

\newcommand{\Zb}{\mathbb{Z}}
\newcommand{\Nb}{\mathbb{N}}
\newcommand{\Rb}{\mathbb{R}}
\newcommand{\Qb}{\mathbb{Q}}
\newcommand{\Tb}{\mathbb{T}}

\newcommand{\U}{\mathcal{U}}

\newcommand{\Qp}{\mathbb{Q}_p}

\newcommand{\Es}{\ms{E}}

\newcommand{\tdlc}{t.d.l.c.\@\xspace}
\newcommand{\tdlcsc}{t.d.l.c.s.c.\@\xspace}
\newcommand{\lcsc}{l.c.s.c.\@\xspace}
\newcommand{\defbold}{\textbf}

\newcommand{\Td}{\vec{T}}
\newcommand{\acts}{\curvearrowright}

\newcommand{\normal}{\trianglelefteq}

\newcommand{\inv}{^{-1}}
\newcommand{\triv}{\{1\}}

\newcommand{\QZ}{\mathrm{QZ}}

\newcommand{\CC}{\mathrm{C}}
\newcommand{\N}{\mathrm{N}}
\newcommand{\Z}{\mathrm{Z}}

\newcommand{\rest}{\upharpoonright}

\newcommand{\injects}{\hookrightarrow}

\newcommand{\img}{\mathrm{img}}
\newcommand{\sym}{\mathrm{Sym}}

\newcommand{\Stab}{\mathrm{Stab}}

\newcommand{\inn}{\mathrm{inn}}

\newcommand{\afr}{\mf{a}}
\newcommand{\bfr}{\mf{b}}
\newcommand{\cfr}{\mf{c}}

\newcommand{\cgrp}[1]{\overline{\langle #1 \rangle}}
\newcommand{\ngrp}[1]{\overline{\langle\langle #1\rangle \rangle}}

\newcommand{\grp}[1]{\langle #1 \rangle}
\newcommand{\ol}[1]{\overline{#1}}

\newcommand{\Aut}{\mathop{\rm Aut}\nolimits}
\newcommand{\Inn}{\mathop{\rm Inn}\nolimits}
\newcommand{\Rad}[1]{\mathop{\rm Rad}_{#1}\nolimits}
\newcommand{\Res}[1]{\mathop{\rm Res}_{#1}\nolimits}

\newcommand{\wt}[1]{\widetilde{#1}}

\newcommand{\RadLE}{\Rad{\mc{LE}}}

\newcommand{\Sol}{\mathrm{Sol}}

\makeindex
\begin{document}

\title[Dense normal subgroups and chief factors]{Dense normal subgroups and chief factors in locally compact groups}
\author{Colin D. Reid}
\address{  University of Newcastle,
	School of Mathematical and Physical Sciences,
	University Drive,
	Callaghan NSW 2308, Australia}
\email{colin@reidit.net}
\thanks{The first author is an ARC DECRA fellow.  Research supported in part by ARC Discovery Project DP120100996.}

\author{Phillip R. Wesolek}
\address{ Department of Mathematical Sciences,
	Binghamton University,
	PO Box 6000,
	Binghamton, New York 13902}
\email{pwesolek@binghamton.edu}
\thanks{The second author is supported by ERC grant \#278469.}
\address{ Previous address: Universit\'{e} catholique de Louvain,
	Institut de Recherche en Math\'{e}matiques et Physique (IRMP),
	Chemin du Cyclotron 2, box L7.01.02,
	1348 Louvain-la-Neuve, Belgium}

\begin{abstract}
	In \textit{The essentially chief series of a compactly generated locally compact group}, an analogue of chief series for finite groups is discovered for compactly generated locally compact groups.  In the present article, we show that chief factors necessarily exist in all locally compact groups with sufficiently rich topological structure.  We also show that chief factors have one of seven types, and for all but one of these types, there is a decomposition into discrete groups, compact groups, and topologically simple groups.

	Our results for chief factors require exploring the theory developed in \textit{Chief factors in Polish groups} in the setting of locally compact groups. In this context, we obtain tighter restrictions on the factorization of normal compressions and the structure of quasi-products. Consequently, both (non-)amenability and elementary decomposition rank are preserved by normal compressions.
\end{abstract}

\maketitle

\tableofcontents

\addtocontents{toc}{\protect\setcounter{tocdepth}{1}}

\section{Introduction}
It is clear that a finite group $G$ admits a normal series in which each of the factors is a chief factor, that is, a normal factor $K/L$ of $G$ such that there is no normal subgroup of $G$ lying between $K$ and $L$.  It is then straightforward to show that chief factors of finite groups are products of simple groups. One thereby arrives at composition series, which give a complete decomposition of a finite group into simple groups.

The previous work \cite{RW_EC_15} establishes the existence of the essentially chief series for compactly generated locally compact groups, which is an analogue for compactly generated locally compact groups of the chief series for finite groups. For the essentially chief series, we must allow for compact and discrete factors in addition to chief factors.  A primary focus of the work at hand is an attempt to complete the analogy with finite groups by decomposing chief factors into topologically simple groups, discrete groups, and compact groups. Theorem~\ref{thm:intro_block_structure} is our high mark in this direction, but it does not quite complete the analogy; Question~\ref{qu:well-foundedness} remains an outstanding mystery.

The proofs require exploring the theory developed in \cite{RW_P_15} in the setting of \textit{locally compact} Polish (equivalently, locally compact second countable) groups. As expected, the results we obtain for locally compact Polish groups are much stronger than those for general Polish groups.

We shall use the abbreviations \lcsc for ``locally compact second countable" and \tdlcsc for ``totally disconnected locally compact second countable."

\subsection{Normal compressions}

We begin by studying normal compressions. A topological group $H$ is a \textbf{normal compression}\index{normal compression} of a topological group $G$ if there is a continuous, injective homomorphism $\psi: G \rightarrow H$ with dense normal image. As shown in \cite{RW_P_15}, when $G$ and $H$ are Polish groups, there is a Polish group $\tilde{H}$ such that $\psi$ can be decomposed as $\beta \circ \alpha$ with $\alpha: G \rightarrow \tilde{H}$ a closed embedding with normal image and $\beta: \tilde{H} \rightarrow H$ a quotient map. 

In the locally compact setting, this factorization can be refined to obtain better control over the group $\tilde{H}$.

\begin{thm}[See Theorem~\ref{thm:compression_factoring:tdlcsc}]
Suppose that $G$ and $H$ are \tdlcsc groups with $\psi: G \rightarrow H$ a normal compression map.  Then there is a \tdlcsc group $\tilde{H}$, a closed embedding $\alpha: G \rightarrow \tilde{H}$ with normal image, and a quotient map $\beta: \tilde{H} \rightarrow H$ such that
\begin{enumerate}[(1)]
\item $\psi=\beta \circ \alpha$;
\item $\tilde{H}/\alpha(G)$ is compact;
\item $\ker(\beta)$ is discrete, lies in the FC-center of $\tilde{H}$, and centralizes $\alpha(G)$;
\item $\tilde{H} = \ol{\alpha(G)\ker(\beta)}$.
\end{enumerate}
\end{thm}

In the case of compactly presented \tdlcsc groups, more can be said; see Corollary~\ref{cor:compression_factoring:compactly_presented}. For a normal compression of \lcsc groups, the compression map factorizes as a sequence of closed normal embeddings and quotient maps such that all kernels and cokernels appearing in the sequence are `small'; see Theorem \ref{thm:compression_factoring}.

Using these factorizations, we study the extent to which properties pass between $G$ and $H$ when $H$ is a normal compression of $G$.

\begin{thm}
	Suppose that $G$ and $H$ are \lcsc groups with $H$ a normal compression of $G$. 
\begin{enumerate}[(1)]
\item $G$ is amenable if and only if $H$ is amenable.  \textup{(See Proposition~\ref{prop:amenability:invariance}.)}
\item If $G$ and $H$ are totally disconnected, then $G$ is elementary with decomposition rank $\alpha$ if and only if $H$ is elementary with decomposition rank $\alpha$.  \textup{(See Proposition~\ref{prop:compression:elementary_rank}.)}
\end{enumerate}
\end{thm}
The class of \textbf{elementary groups}\index{elementary groups} is the smallest class of \tdlcsc groups containing the profinite groups and discrete groups and closed under the elementary operations. This class admits a canonical rank function called the decomposition rank, taking values in the set $\omega_1$ of at most countable ordinals.  It is convenient to define non-elementary \tdlcsc groups to have decomposition rank $\omega_1$.

A Polish group $G$ is \textbf{quasi-discrete}\index{quasi-discrete} if the quasi-center $\QZ(G):=\{g\in G\mid \CC_G(g)\text{ is open}\}$ is dense.  The property of being discrete is not preserved under normal compressions, but quasi-discreteness is preserved up to metabelian factors. 

\begin{thm}[See Theorem~\ref{thm:compression:quasidiscrete}]
Suppose that $G$ and $H$ are \lcsc groups with $H$ a normal compression of $G$. 
	\begin{enumerate}[(1)]
		\item If $H$ is quasi-discrete, then $G/\ol{\QZ(G)}$ is metabelian.
		\item If $G$ is quasi-discrete, then $H/\ol{\QZ(H)}$ is metabelian.
	\end{enumerate}
\end{thm}

\subsection{Generalized direct products}

We next consider generalized direct products in the locally compact setting.  

\begin{defn}
	For $G$ a topological group and $J$ a collection of closed normal subgroups, we set $G_J :=  \cgrp{N \mid N \in J}$.  A collection $\mc{S}$ of closed normal subgroups is a \defbold{quasi-direct factorization} of $G$ if $G_{\mc{S}}=G$, and $\mc{S}$ has the following topological independence property:
	\[
	\forall X\subseteq \mathcal{P}(\mc{S}):\; \bigcap X=\emptyset \Rightarrow \bigcap_{A\in X}G_A=\{1\}. 
	\]
	In such a case, $(G,\mc{S})$ is said to be a \textbf{quasi-product}\index{quasi-product} (or we say $G$ is a quasi-product, when the factorization is implicit).\end{defn}

Basic examples of quasi-products are local direct products.

\begin{defn}
	Suppose that $(G_i)_{i\in I}$ is a set of topological groups and that there is a distinguished open subgroup $O_i\leq G_i$ for each $i\in I$.  The \textbf{local direct product}\index{local direct product} of $(G_i)_{i\in I}$ over $(O_i)_{i\in I}$ is denoted by $\bigoplus_{i\in I}\left(G_i,O_i\right)$ and defined to be
	\[
	\left\{ f:I\rightarrow \bigsqcup_{i\in I} G_i\mid f(i)\in G_i \text{, and }f(i)\in O_i\text{ for all but finitely many }i\in I\right\}
	\]
	with component-wise multiplication and the group topology such that $\prod_{i\in I}O_i$ with the product topology continuously embeds as an open subgroup of $\bigoplus_{i\in I}\left(G_i,O_i\right)$.
\end{defn}

We identify certain quasi-products that are closely related to local direct products.

\begin{defn}
	Let $G$ be a locally compact group with $\mc{S}$ a quasi-direct factorization of $G$.  A \defbold{local direct model}\index{local direct model} for $(G,\mc{S})$ is a continuous homomorphism $\phi: \bigoplus_{\mc{S}}(G) \rightarrow G$ where $\bigoplus_{\mc{S}}(G) := \bigoplus_{N\in \mc S}(N, N\cap O)$ for some open subgroup $O$ of $G$ is such that $\bigoplus_{\mc{S}}(G)$ is locally compact and $\phi\rest_N=\mathrm{id}_N$ for all $N \in \mc{S}$.
\end{defn}

Local direct models for quasi-products, provided they exist, enjoy a universal property and are in particular unique; see Theorem~\ref{thm:LDM_universal_prop} and Corollary~\ref{cor:LDM_unique}. 

\begin{thm}[See Corollary~\ref{cor:tdlc_LDM}] 
	If $(G,\mc{S})$ is a \tdlc quasi-product, then $(G,\mc{S})$ admits a local direct model.
\end{thm}

A natural source of quasi-products are groups of strict semisimple type: A Polish group $G$ is of \defbold{strict semisimple type}\index{strict semisimple type} if it is a quasi-product of non-abelian topologically simple groups.

\begin{thm}[See Theorem~\ref{thm:strict_semisimple}]
For $G$ an \lcsc group of strict semisimple type, 
\[
G\simeq D\times \prod_{i\in I}C_i\times\prod_{j=0}^nS_j
\]
with $D$ a \tdlcsc group of strict semisimple type, each $C_i$ a topologically simple compact connected Lie group, and each $S_j$ a non-compact topologically simple connected Lie group.
\end{thm}

\begin{cor}[See Corollary~\ref{cor:strict_semisimple:local_model}]
Suppose that $G$ is an \lcsc group of strict semisimple type and let $\mc{S}$ be the set of non-abelian topologically simple closed normal subgroups of $G$.  Then $(G,\mc{S})$ is a quasi-product that admits a local direct model.\end{cor}

\subsection{Chief factors and blocks}
We finally bring our results together to study chief factors in locally compact groups. The following results draw on the general theory of chief factors as developed in \cite{RW_P_15} as well as the existence of essentially chief series for compactly generated locally compact groups as shown in \cite{RW_EC_15}.  In fact, we will show (Proposition~\ref{prop:min_covered_exist}) that even without the assumption of compact generation, chief factors of arbitrary \lcsc groups $G$ account for a large part of their complexity as topological groups.

A normal factor $K/L$ of closed normal subgroups of a topological group $G$ is a \textbf{chief factor}\index{chief factor} if there is no closed $M\normal G$ such that $L<M<K$.
\begin{defn}
Given a topological group $G$, we say that the closed normal factors $K_1/L_1$ and $K_2/L_2$ are \defbold{associated}\index{association} if the following equations hold:
	\[
	\ol{K_1L_2} = \ol{K_2L_1}; \; K_1 \cap \ol{L_1L_2} = L_1; \; K_2 \cap \ol{L_1L_2} = L_2.
	\]
\end{defn}
The association relation is an equivalence relation when restricted to non-abelian chief factors via \cite[Proposition 6.8]{RW_P_15}. A \textbf{chief block}\index{chief block} is an association class of non-abelian chief factors.  The set of chief blocks of a group $G$ is denoted by $\mf{B}_G$\index{$\mf{B}_G$}.

The association relation is not in general comparable with isomorphism. To study group-theoretic properties via chief blocks, we identify properties which hold on the level of blocks. Blocks with compact or discrete representatives have exceptional properties, so it is natural to isolate a class of blocks which excludes these.

\begin{defn} 
Suppose that $\mf{a}\in \mf{B}_G$. We say that the block $\mf{a} \in \mf{B}_G$ is \defbold{robust}\index{chief block!robust} if none of the representatives of $\mf{a}$ are compact or quasi-discrete. The collection of robust chief blocks is denoted by $\mf{B}_G^r$.
\end{defn}

\begin{defn}
	A property $P$ of groups is a \textbf{(robust) block property}\index{block property} if for every \lcsc group $G$ and $\mf{a}$ a (robust) chief block of $G$, there exists $K/L\in \mf{a}$ with $P$ if and only if every $K/L\in \mf{a}$ has $P$. For a (robust) block property $P$, we say that a (robust) chief block $\mf{a}$ of $G$ has property $P$ if some, equivalently all, $K/L\in \mf{a}$ has property $P$.
\end{defn}

\begin{thm}[See Propositions~\ref{prop:cg_block_property} and \ref{prop:block_properties}]
The following are block properties: Being elementary with rank $\alpha$, being amenable, being quasi-discrete, and being isomorphic to a given connected group $H$.  Compact generation is a robust block property.
\end{thm}

Representatives of non-robust blocks therefore have a restricted structure.

\begin{cor}
Suppose that $G$ is an \lcsc group and suppose $\mf{a}$ is a chief block of $G$ that is not robust.  Then either every representative of $\mf{a}$ is isomorphic to a fixed connected compact group, or every representative of $\mf{a}$ is totally disconnected and quasi-discrete.
\end{cor}

To generalize the existence of chief factors from compactly generated groups to general \lcsc groups, and also to analyze the internal structure of these chief factors, we employ the concepts of minimally covered and extendable chief blocks as introduced in \cite{RW_P_15}. For a Polish group $G$, a closed normal subgroup $N\normal G$ \textbf{covers}\index{chief block!covering} a block $\afr\in \mf{B}_G$ if there is $K/L\in \afr$ such that $K\leq N$.  A block $\afr$ is \textbf{minimally covered}\index{chief block!minimally covered} if there is a least normal subgroup that covers the block.

\begin{defn}
	Let $G$ be a Polish group and $H$ be a closed subgroup of $G$.  Given $\mf{a} \in \mf{B}_H$ and $\mf{b} \in \mf{B}_G$, we say that $\mf{b}$ is an \defbold{extension}\index{chief block!extension} of $\mf{a}$ in $G$ if for every closed normal subgroup $K$ of $G$, the group $K$ covers $\mf{b}$ if and only if $K \cap H$ covers $\mf{a}$.
\end{defn}

\begin{thm}[See Theorem~\ref{thm:limit_extension}] 
	If $O$ is a compactly generated open subgroup of an \lcsc group $G$, then all robust blocks of $O$ extend to $G$.
\end{thm}

Blocks of an \lcsc group arising from extending a robust block of a compactly generated open subgroup are called \textbf{regionally robust blocks}\index{chief block!regionally robust}. The set of regionally robust blocks of an \lcsc group $G$ is denoted by $\mf{B}_G^{rr}$\index{$\mf{B}_G^{rr}$}. These blocks are minimally covered and robust; see Lemma~\ref{lem:extension:min_covered} and Propositions~\ref{prop:min_covered} and \ref{prop:robust_extension}. The set of regionally robust blocks is setwise invariant under all continuous automorphisms. A \lcsc group also has at most countably many regionally robust blocks (Lemma~\ref{lem:limit_blocks_number}). Moreover, an \lcsc group $G$ has at least one regionally robust block provided that its structure as a topological group is sufficiently complex:

\begin{prop}[See Proposition~\ref{prop:min_covered_exist}] For $G$ an \lcsc group, either
	\begin{enumerate}
		\item $\mf{B}_{G}^{rr}\neq \emptyset$, or
		\item $G^{\circ}$ is compact-by-solvable-by-compact, and $G/G^{\circ}$ is elementary with rank at most $\omega+1$. If $G$ is compactly generated, then $G/G^{\circ}$ is elementary with finite rank.
	\end{enumerate}
\end{prop}

Significant restrictions on the structure of topologically characteristically simple \lcsc groups then follow; chief factors in \lcsc groups are particular examples of such groups.

\begin{thm}[See Theorem~\ref{thm:chief:block_structure}]\label{thm:intro_block_structure}
	Suppose that $G$ is a topologically characteristically simple \lcsc group. Then $G$ is exactly one of the following:
\begin{enumerate}[(1)]
		\item an abelian group;
		\item a non-abelian quasi-discrete \tdlcsc group;
		\item an elementary \tdlcsc group of decomposition rank $2$ with trivial quasi-center;
		\item an elementary \tdlcsc group of decomposition rank $\omega+1$;
		\item a direct product of copies of a simple Lie group;
		\item a quasi-product of copies of a topologically simple \tdlcsc group of decomposition rank greater than $\omega+1$;
		\item a \tdlcsc group of stacking type and decomposition rank greater than $\omega+1$.
\end{enumerate}
	In particular, chief factors of \lcsc groups have one of the above forms. 
\end{thm}

All of the cases (1)-(7) can occur as chief factors of an \lcsc group. Example~\ref{ex:stacking} shows case (7) occurs naturally, even for a chief factor of a compactly generated \lcsc group.

\section{Preliminaries}

\subsection{Notations and definitions}
For $\mc{P}$ a poset, we say that $\mc{D}\subseteq \mc{P}$ is a \textbf{directed family} if for all $M,N\in \mc{D}$, there is $L\in \mc{D}$ with $M \le L$ and $N \le L$. We say $\mc{F}\subseteq \mc{P}$ is a \textbf{filtering family} if for all $M,N\in \mc{F}$ there is $L\in \mc{F}$ such that $L\leq M\cap N$.

Given a group $G$ and a set $I$, we denote the direct sum by $G^{<I}$ and the direct product by $G^I$. For a subset $K\subseteq G$, $\CC_G(K)$ is the collection of elements of $G$ that centralize every element of $K$. For the normalizer, we write $\N_G(K)$.  For $A,B\subseteq G$, we put 
\[
\left[A,B\right]:=\grp{aba\inv b\inv \;|\;a\in A\text{ and }b\in B}
\]
For $a,b\in G$, $[a,b]:=aba\inv b\inv $ and $a^b:=bab^{-1}$. For $a,b,c\in G$, we set $[a,b,c]:=[[a,b],c]$.

All topological groups are Hausdorff topological groups and are written multiplicatively. A topological group is \textbf{Polish}\index{Polish group} if the topology is separable and admits a complete, compatible metric. A locally compact group is Polish if and only if it is second countable. Topological group isomorphism is denoted by $\simeq$. For a topological group $G$, the connected component of the identity is denoted by $G^\circ$. The topological closure of a set $K$ in $G$ is denoted by $\ol{K}$. We shall say that $G$ is \defbold{topologically perfect} if $[G,G]$ is dense in $G$.

 We use ``t.d.", ``l.c.", and ``s.c." for ``totally disconnected", ``locally compact", and ``second countable", respectively. The symbols $\Zb$, $\Qb$, $\Zb_p$, $\Qb_p$, and $\Rb$ denote the additive groups of the relevant rings with the discrete topology for $\Qb$ and $\Zb$ and the usual locally compact topology for the others.  We write $\Tb$ for the circle group $\Rb/\Zb$ and write $C_p$ for the cyclic group of order $p$.

\subsection{Generalities on locally compact groups}\label{sec:prelim_connected}
A locally compact group is \defbold{almost connected}\index{almost connected} if the identity component is cocompact. For $G$ a locally compact group, the set of almost connected subgroups is denoted by $\U(G)$. Note that $G^\circ = U^\circ$ for all open subgroups $U$ of $G$, so $\U(G)$ consists precisely of those open subgroups $U$ of $G$ such that $U/G^{\circ}$ is compact.  If $G$ is a \tdlc group, then $\U(G)$ is a base of identity neighborhoods, by van Dantzig's theorem.

A locally compact group is \defbold{locally elliptic}\index{locally elliptic} if every compact subset (equivalently, every finite subset) generates a relatively compact subgroup.  The \defbold{locally elliptic radical}\index{locally elliptic radical} of a locally compact group $G$, denoted by $\RadLE(G)$\index{$\RadLE(G)$}, is the union of all closed normal locally elliptic subgroups of $G$.

\begin{thm}[Platonov, \cite{Plat66}]\label{thm:platonov_radical}
	For $G$ a locally compact group, $\RadLE(G)$ is the unique largest locally elliptic closed normal subgroup of $G$.  Additionally,
	\[
	\RadLE(G/\RadLE(G)) = \triv.
	\]
\end{thm}

The \textbf{quasi-center} of a locally compact group $G$ is 
\[
\QZ(G):=\{g\in G\mid \CC_G(g)\text{ is open}\}.
\]
The quasi-center is a characteristic subgroup, but it is usually not closed. A locally compact group is \textbf{quasi-discrete} if it has a dense quasi-center.

For a compactly generated \tdlc group $G$, there exists a locally finite graph $\Gamma$ such that $G$ acts vertex-transitively on $\Gamma$ with compact open stabilizers; such a graph is called a \defbold{Cayley--Abels graph} for $G$. The existence of this graph is ensured by the next proposition.

\begin{prop}[See {\cite[Beispiel 5.2]{A72}} or {\cite[Section 11]{Mon01}}]\label{prop:factor}
Let $G$ be a compactly generated \tdlc group, let $U \in \U(G)$, let $A$ be a compact symmetric subset of $G$ such that $G = \langle U, A \rangle$ and let $D$ be a dense symmetric subset of $G$.
\begin{enumerate}[(1)]
\item There exists a finite symmetric subset $B$ of $G$ such that $B \subseteq D$ and 
\[
BU = UB = UBU = UAU.
\]
\item For any subset $B$ satisfying part (1), it is the case that $G = \langle B \rangle U$ and the coset space $G/U$ carries the structure of a locally finite connected graph, invariant under the natural $G$-action, where $gU$ is adjacent to $hU$ if and only if $gU \neq hU$ and $Ug\inv hU \subseteq UBU$.  Thus, $G/U$ is the vertex set of a Cayley--Abels graph for $G$.
\end{enumerate}
\end{prop}

A (real) \defbold{Lie group}\index{Lie group} is a topological group that is a finite-dimensional analytic manifold over $\Rb$ such that the group operations are analytic maps. Closed subgroups of Lie groups are again Lie groups. A Lie group $G$ can have any number of connected components in general, but $G^\circ$ is always an \emph{open} subgroup of $G$. 

A Lie group $L$ has a largest connected solvable normal subgroup; it is called the \defbold{solvable radical}\index{Lie group!solvable radical} of $L$ and denoted by $\Sol(L)$\index{$\Sol(L)$}.  A connected Lie group with trivial solvable radical is called \defbold{semisimple}\index{Lie group!semisimple}. A semisimple Lie group has discrete center, and modulo its center, it is a finite direct product of abstractly simple groups.

The manifold structure of a Lie group gives a well-defined dimension $\dim_{\Rb}(G)$, which is additive with respect to extensions. We stress that dimension zero Lie groups are discrete. The dimension places a restriction on directed families of normal subgroups.

\begin{lem}\label{lem:Lie:directed_family}
	Let $G$ be a connected Lie group and let $\mc{D}$ be a directed family of closed normal subgroups of $G$.  Then there exists $N\in \mc{D}$ such that $MN/N$ is discrete in $G/N$ for all $M\in \mc{D}$.
\end{lem}

\begin{proof}
	Let $d$ be the maximum of the set $\{\dim_{\Rb}(N) \mid N\in \mc{D}\}$ and take $N\in \mc{D}$ realizing this maximum. For all $M\in \mc{D}$, there is $K\in \mc{D}$ extending $N$ and $M$, since $\mc{D}$ is directed. The group $K/N$ is a dimension zero Lie subgroup of $G/N$, hence $K/N$ is discrete. We thus conclude that $MN/N$ is also discrete, verifying the lemma.
\end{proof}

Lie groups play a fundamental role in the structure theory of connected groups.

\begin{thm}[Gleason--Yamabe; see {\cite[Theorem 4.6]{MZ}}]\label{thm:yamabe_radical}
If $G$ is an almost connected locally compact group, then $\RadLE(G)$ is compact, and $G/\RadLE(G)$ is a Lie group.  Moreover, there exist arbitrarily small compact normal subgroups $K$ of $G$ such that $G/K$ is a Lie group.
\end{thm}

\subsection{Normal compressions of Polish groups} We conclude our preliminary discussion by recalling the basic theory of normal compressions developed in \cite{RW_P_15}.

\begin{defn}\label{def:compression}
	For $G$ and $H$ Polish groups, a \defbold{normal compression map}\index{normal compression map} from $G$ to $H$ is an injective, continuous homomorphism $\psi: G \rightarrow H$ with dense, normal image. When such a map exists, we say that $H$ is a \defbold{normal compression}\index{normal compression} of $G$.
\end{defn}

Given a normal compression map $\psi:G\rightarrow  H$, there is a canonical action $H\acts G$ via
\[
h.g:=\psi^{-1}(h\psi(g)h^{-1}).
\]
This action is called the \textbf{$\psi$-equivariant}\index{$\psi$-equivariant action} action, as it is the action with respect to which $\psi$ is an $H$-equivariant map for $H$ acting on itself by conjugation.

\begin{prop}[{\cite[Proposition 3.5]{RW_P_15}}]\label{prop:Polish:semidirect}
Suppose that $G$ and $H$ are Polish groups with $\psi: G \rightarrow H$ a normal compression map. Then the semidirect product $G\rtimes_{\psi} H$ via the $\psi$-equivariant action is a Polish group under the product topology.
\end{prop}
Normal compression maps factor through this semidirect product.

\begin{thm}[{\cite[Theorem~3.6]{RW_P_15}}]\label{thm:psi-compression_factor_rel}
Suppose that $G$ and $H$ are Polish groups with $\psi: G \rightarrow H$ a normal compression map and let $O\leq H$ be an open subgroup. Then the following hold:
\begin{enumerate}[(1)]
\item $\pi:G\rtimes_{\psi}O\rightarrow H$ via $(g,o)\mapsto \psi(g)o$ is a continuous, surjective homomorphism with $\ker(\pi)=\{(g\inv,\psi(g))\mid g\in \psi^{-1}(O)\}$, and if $O = H$, then $\ker(\pi)\simeq G$ as topological groups.
\item $\psi = \pi \circ \iota$ where $\iota: G \rightarrow G\rtimes_{\psi}O$ is the usual inclusion.
\item $G\rtimes_{\psi}O =\ol{\iota(G)\ker (\pi)}$, and the subgroups $\iota(G)$ and $\ker(\pi)$ are closed normal subgroups of $G\rtimes_{\psi}O$ with trivial intersection.
\end{enumerate}
\end{thm}

This factorization allows some properties to pass back and forth under normal compressions, possibly up to an abelian factor.

\begin{prop}[{\cite[Proposition 3.8]{RW_P_15}}]\label{prop:normal_compression}
Suppose that $G$ and $H$ are Polish groups with $\psi: G \rightarrow H$ a normal compression map and let $K$ be a closed normal subgroup of $G$. Then the following hold:
\begin{enumerate}[(1)]
\item The image $\psi(K)$ is a normal subgroup of $H$.
\item If $\psi(K)$ is also dense in $H$, then $\ol{[G,G]} \le K$, and every closed normal subgroup of $K$ is normal in $G$.
\end{enumerate}
\end{prop}

We shall also make use of two technical results concerning connected components.

\begin{lem}[{\cite[Lemma 3.11]{RW_P_15}}]\label{lem:disconnected_to_connected}
Let $G$ and $H$ be Polish groups with $\psi: G \rightarrow H$ a normal compression map.  Let $R := \psi\inv(H^\circ)$ and let $N := \ol{\psi(R)}$.  Then the following inclusions hold:
\[
[H^\circ,\psi(G)] \le \psi(R); \;[H^\circ,H] \le N; [H^{\circ},\psi(G)]\leq \psi(G^{\circ}); \; [H^\circ, N] \le \ol{\psi(G^\circ)}.
\]
\end{lem}

\begin{thm}[{\cite[Theorem 3.12]{RW_P_15}}]\label{thmintro:connected}
	Suppose that $G$ and $H$ are Polish groups with $\psi: G \rightarrow H$ a normal compression map. Then the following hold:
	\begin{enumerate}[(1)]
		\item The group $H^\circ/\ol{\psi(G^\circ)}$ is nilpotent of class at most two.
		\item If $H$ is connected, then both $G/G^\circ$ and $H/\ol{\psi(G^\circ)}$ are abelian.
		\item If $G$ is totally disconnected, then $H^\circ$ is central in $H$.
	\end{enumerate}
\end{thm}

\begin{cor}[{\cite[Corollary 3.13]{RW_P_15}}]\label{cor:con_association}
	Suppose that $G$ and $H$ are centerless topologically perfect Polish groups and that $H$ is a normal compression of $G$.  Then $G$ is connected if and only if $H$ is connected, and $G$ is totally disconnected if and only if $H$ is totally disconnected.
\end{cor}

\section{Elementary groups and the decomposition rank}\label{sec:elementary}
Elementary groups will appear several times in the present work. We here recall the relevant definitions and note several new properties.
\subsection{Introduction}

\begin{defn}The class $\Es$ of \defbold{elementary groups}\index{elementary group} is the smallest class of \tdlcsc groups such that
	\begin{enumerate}[(i)]
		\item $\Es$ contains all second countable profinite groups and countable discrete groups.
		\item $\Es$ is closed under taking closed subgroups.
		\item $\Es$ is closed under taking Hausdorff quotients.
		\item $\Es$ is closed under forming group extensions.
		\item If $G$ is a \tdlcsc group and $G=\bigcup_{i\in \Nb}O_i$ where $(O_i)_{i\in \Nb}$ is an $\subseteq$-increasing sequence of open subgroups of $G$ with $O_i\in\Es$ for each $i$, then $G\in\Es$. We say that $\Es$ is \textbf{closed under countable increasing unions}.
	\end{enumerate}
\end{defn} 
The operations $\mathrm{(ii)}-\mathrm{(v)}$ are often called the \textbf{elementary operations}\index{elementary operations}. It turns out operations $\mathrm{(ii)}$ and $\mathrm{(iii)}$ follow from the others, and $\mathrm{(iv)}$ can be weakened to $\mathrm{(iv)}'$: \textit{$\Es$ is closed under extensions of profinite groups and discrete groups}. These results are given by \cite[Theorem 1.3]{W_1_14}. The class of elementary groups enjoys additional closure properties:

\begin{thm}[See {\cite[Theorem 1.3]{W_1_14}}]\label{thm:elementary_closure}
	The class $\Es$ enjoys the following properties:
\begin{enumerate}[(1)]
\item If $G\in \Es$, $H$ is a \tdlcsc group, and $\psi:H\rightarrow G$ is a continuous, injective homomorphism, then $H\in \Es$. 
\item If $G$ is a \tdlcsc group that is residually in $\Es$, then $G\in \Es$. In particular, $\Es$ is closed under inverse limits that result in a \tdlcsc group.
\end{enumerate}
\end{thm}

A primary tool to study elementary groups is a canonical rank function.

\begin{defn}
	Let $\ms{T}$ be the class of \tdlcsc groups.  The \textbf{decomposition rank}\index{decomposition rank} $\xi$ is a partial ordinal-valued function from $\ms{T}$ to $[1,\omega_1)$ satisfying the following properties:
	\begin{enumerate}[(a)]
		\item $\xi(G)=1$ if and only if $G \simeq \triv$;
		
		\item If $G \in \ms{T}$ is non-trivial and $G=\bigcup_{i\in \Nb}O_i$ with $(O_i)_{i\in \Nb}$ some increasing sequence of compactly generated open subgroups of $G$, then $\xi(G)$ is defined if and only if $\xi(R_i)$ is defined for $R_i := \Res{}(O_i)$ for all $i \in \Nb$.  If $\xi(G)$ is defined, then
		\[
		\xi(G)=\sup_{i\in \Nb}\xi(R_i)+1.
		\]
	\end{enumerate}
\end{defn}

By \cite[Theorem 4.7]{W_1_14} and \cite[Lemma 4.12]{W_1_14}, such a function $\xi$ exists, is unique, has domain of definition exactly $\Es$, and is equivalent to the decomposition rank as defined in \cite{W_1_14}.  Note that if $G$ is a non-trivial \tdlcsc group that is compactly generated and has no non-trivial discrete quotients, then $G$ is \emph{not} in $\Es$.  The class $\Es$ is thus strictly smaller than the class of all \tdlcsc groups. For convenience, we extend the definition of $\xi$ to all \tdlcsc groups by saying that a non-elementary group satisfies $\xi(G) = \omega_1$.

By definition, the rank of an elementary \tdlcsc group is always a countable \emph{successor} ordinal.  We observe a slightly stronger restriction on the rank of a compactly generated group.

\begin{lem}\label{lem:betaplus2}
Let $G$ be a compactly generated elementary \tdlcsc group.  Then
\[
\xi(G) = \xi(\Res{}(G)) + 1 = \beta+2
\]
for some at most countable ordinal $\beta$.
\end{lem}

\begin{proof}
From the rank formula, $\xi(G) = \xi(\Res{}(G)) + 1$; take $O_i = G$ for all $i \in \Nb$. On the other hand, $\xi(\Res{}(G))$ is a successor ordinal, say $\beta+1$, so $\xi(G) = \beta+2$.
\end{proof}

\subsection{Stability of the decomposition rank}

The decomposition rank enjoys several useful stability properties.

\begin{prop}\label{prop:xi_monotone}
For $G$ a non-trivial elementary group, then the following hold:
\begin{enumerate}[(1)]
	\item If $H$ is a \tdlcsc group, and $\psi:H\rightarrow G$ is a continuous, injective homomorphism, then $\xi(H)\leq \xi(G)$. \textup{(\cite[Corollary 4.10]{W_1_14})}
	\item If $L\normal G$ is closed, then $\xi(G/L)\leq \xi(G)$. \textup{(\cite[Theorem 4.19]{W_1_14})}
\end{enumerate}
\end{prop}

We give here a slightly sharper version of \cite[Lemma~7.4]{RW_Hom_15}.  Given a successor ordinal $\alpha$, we write $(\alpha - 1)$ to mean the ordinal $\beta$ such that $\alpha = \beta+1$.

\begin{lem}\label{lem:d-rank_extensions}
Suppose that $G$ is a \tdlcsc group and $G$ lies in a short exact sequence of \tdlcsc groups
\[
\triv\rightarrow N\rightarrow G\rightarrow Q\rightarrow \triv.
\]
Then $\xi(G)\leq (\xi(N)-1)+\xi(Q)$.
\end{lem}

\begin{proof}
We argue by induction on $\xi(Q)$ for all such $G$. In the base case, $\xi(Q)= 1$ and $Q=\triv$, so $\xi(G) = \xi(N) = (\xi(N) -1) + \xi(Q)$.  Suppose that the lemma holds up to $\beta$ and that $\xi(Q)=\beta+1$. Let $\pi:G\rightarrow Q$ be the usual projection and fix $(O_i)_{i\in \Nb}$ an increasing exhaustion of $G$ by compactly generated open subgroups. Put $W_i:=\pi(O_i)$ and note that the $W_i$ form an increasing exhaustion of $Q$ by compactly generated open subgroups.

In view of the decomposition rank formula, it suffices to consider $\xi(\Res{}(O_i))$ for an arbitrary $i$. Fix $i\in \Nb$, form $R:=\Res{}(O_i)$, and put $M:=\ol{RN}$. It follows that $M/N\leq \Res{}(W_i)$, hence $\xi(M/N)\leq \beta$. The induction hypothesis implies $\xi(M)\leq (\xi(N)-1)+\beta$, and since $R\injects M$, we conclude that $\xi(R)\leq (\xi(N)-1)+\beta$. Appealing to Proposition~\ref{prop:xi_monotone}, $\xi(G)\leq (\xi(N)-1)+\beta+1$, verifying the lemma.  
\end{proof}

\begin{lem}\label{lem:product_elementary} 
Suppose that $H_1,\dots,H_n$ are \tdlcsc groups. If each $H_i$ is elementary, then $\prod_{i=1}^nH_i$ is elementary with \[
\xi\left(\prod_{i=1}^nH_i\right)=\max_{1\leq i\leq n}\xi(H_i).
\]
\end{lem}
\begin{proof}
Set $G:=\prod_{i=1}^nH_i$. In view of Proposition~\ref{prop:xi_monotone}, $\max_{1\leq i \leq n}\xi(H_i)\leq \xi(G)$. We argue by induction on $\xi(G)$ for the converse inequality. If $\xi(G)=1$, then $G$ is the trivial group, and the inductive claim is immediate. Suppose that $\xi(G)=\beta+1$. For each $1\leq i\leq n$, let $(O_{j,i})_{j\in \Nb}$ be an increasing exhaustion of $H_i$ by compactly generated open subgroups. The groups $L_j:=\prod_{i=1}^nO_{j,i}$ then form an increasing exhaustion of $G$ by compactly generated open subgroups. Hence, $\xi(G)=\sup_{j\in \Nb}\xi(\Res{}(L_i))+1$.

On the other hand, $\Res{}(L_j)=\prod_{i=1}^n\Res{}(O_{j,i})$, so the inductive hypothesis ensures that 
\[
\xi(\Res{}(L_j))=\max_{1\leq i\leq n}\xi(\Res{}(O_{j,i})). 
\]
By passing to a subsequence if necessary, we may assume there is $1\leq k\leq n$ such that $\xi(\Res{}(L_j))=\xi(\Res{}(O_{j,k}))$ for all $j\in \Nb$. Therefore, 
\[
\xi(G)=\sup_{j\in \Nb}\xi(\Res{}(O_{j,k}))+1=\xi(H_k).
\]
 Since it must be the case that $\xi(H_k)=\max_{1\leq i\leq n}\xi(H_i)$, the induction is complete.
\end{proof}

\begin{lem}\label{lem:compact_ext}
	Suppose that $G$ is a \tdlcsc group and $N$ is a non-trivial closed cocompact normal subgroup of $G$. If $N\in \Es$, then $G\in \Es$ with $\xi(G)=\xi(N)$.
\end{lem}

\begin{proof}
	Plainly $G\in \Es$. To compute the decomposition rank, fix $(O_i)_{i\in \Nb}$ a countable $\subseteq$-exhaustion of $G$ by compactly generated open subgroups of $G$ and put $N_i:=N\cap O_i$. The group $N_i$ is open in $N$, and since $N_i$ is cocompact in $O_i$, it is also compactly generated. We therefore infer that
	\[
	\xi(N)=\sup_{i\in \Nb}\xi(\Res{}(N_i))+1.
	\]
	
	On the other hand, \cite[Lemma 4.16]{W_1_14} shows that $O_i/\Res{}(N_i)$ is residually discrete, and it follows that $\Res{}(O_i)=\Res{}(N_i)$. We conclude that
	\[
	\xi(G)=\sup_{i\in \Nb}\xi(\Res{}(O_i))+1=\sup_{i\in \Nb}\xi(\Res{}(N_i))+1=\xi(N),
	\]
	verifying the lemma.
\end{proof}

At this point it will be convenient to introduce notation for a kind of limit that often occurs with elementary decomposition rank.

\begin{defn}Given a set $I$ of ordinals, define $\sup^+(I) = \sup(I)$ if $\sup(I)$ is either $\omega_1$ or a successor ordinal and define $\sup^+(I) = \sup(I)+1$ if $\sup(I)$ is a countable limit ordinal. If $I$ is a singleton $\{\alpha\}$, we write $\sup^+(\alpha)$. The ordinal $\sup^+(I)$ is the least ordinal that is an upper bound for $I$ and that is a potential decomposition rank for a \tdlcsc group. \end{defn}

It is easily verified, using the definition of the decomposition rank, that if $G$ is an ascending union of open subgroups $O_i$, then $\xi(G) = {\sup_{i \in \Nb}}^+\xi(O_i)$.  In fact, the same formula applies more generally if $G$ contains a dense ascending union of closed subgroups, as long as these subgroups each have open normalizer.
 
\begin{prop}\label{prop:d-rank:ascending}
	Let $G$ be a \tdlcsc group and let $(C_i)_{i \in I}$ be a directed family of closed subgroups of $G$ each with open normalizer such that $G = \ol{\bigcup_{i \in I}C_i}$. If each $C_i$ is elementary, then $G$ is elementary with
	\[
	\xi(G) = {\sup_{i \in I}}^+\xi(C_i).
	\]
\end{prop}

\begin{proof}
By Proposition~\ref{prop:xi_monotone}, we have $\xi(C_i) \le \xi(G)$ for each $i$. Thus $\sup_{i \in I}^+\xi(C_i)\leq \xi(G)$. 
	
For the reverse inequality, fix $U\in \U(G)$ and let $(O_j)_{j\in \Nb}$ be an increasing exhaustion of $G$ by compactly generated open subgroups each containing $U$.  For each $j$, the union $\bigcup_{i\in I}O_j\cap C_i$ is dense in $O_j$, so in view of Proposition~\ref{prop:factor}, we may find $i$ and $A\subseteq O_j\cap C_i$ such that $UAU=AU$ and $UAU$ contains a generating set for $O_j$. Therefore, $O_j=\grp{A}U=(O_j\cap C_i)U$.  
	
The subgroup $C_i$ has an open normalizer, so we may take an open $V\normal U$ such that $V\leq \N_G(C_i)$. The group $\grp{A^V}V$ is then a finite index subgroup of $O_j$. Letting $L$ be the normal core of $\grp{A^V}V$ in $O_j$, we see that $L$ is a finite index normal subgroup of $O_j$ and of $\grp{A^V}V$. Lemma~\ref{lem:compact_ext} thus ensures that 
	\[
	\xi(\grp{A^V}V)=\xi(L)=\xi(O_j).
	\]
	Furthermore, $\cgrp{A^V}$ is a closed cocompact normal subgroup of $\grp{A^V}V$, so a second application of Lemma~\ref{lem:compact_ext} implies that $\xi(\cgrp{A^V})=\xi(O_j)$. Proposition~\ref{prop:xi_monotone}  now ensures that $\xi(O_j)\leq \xi(C_i)$.
	
	Letting $\alpha := \sup_{i \in I}\xi(C_i)$, there are two possibilities: either $\alpha = \alpha'+1$ is achieved by some $\xi(C_i)$, or $\alpha$ is a limit ordinal.  If $\alpha$ is a limit ordinal, then
	\[
	\xi(G) = \sup_{j \in \Nb}\xi(\Res{}(O_j)) + 1 \le \sup_{i \in I}\xi(C_i) + 1 = {\sup_{i \in I}}^+\xi(C_i).
	\]
	On the other hand, suppose that $\alpha$ is achieved by some $\xi(C_i)$; without loss of generality, $\xi(C_i) = \alpha$ for all $i$. For every compactly generated open subgroup $O_j$, we have that $\xi(O_j) \le \alpha$, so $\xi(\Res{}(O_j)) \le \alpha'$. Hence, $\xi(G) \le \alpha'+1= \alpha$.  In either case, we have obtained the required upper bound on $\xi(G)$.  In particular, it follows that $G$ is elementary.
\end{proof}

There is also a dual result for residually elementary groups. 
\begin{lem}\label{lem:ordinal_FF}
	Let $G$ be a \tdlcsc group and $\alpha$ an ordinal. Then $\{N\normal G\mid N\text{ is closed and }\xi(G/N)\leq \alpha\}$ is a filtering family.
\end{lem}
\begin{proof}
	It suffices to show $\Sigma:=\{N\normal G\mid N\text{ is closed and }\xi(G/N)\leq \alpha\}$ is closed under pairwise intersections. Taking $M, N \in \Sigma$, the diagonal map $G\rightarrow G/N\times G/M$ induces an injective, continuous map $G/(N\cap M)\rightarrow G/N\times G/M$. Lemma~\ref{lem:product_elementary} and Proposition~\ref{prop:xi_monotone} now ensure that
	\[
	\xi(G/(N\cap M))\leq \xi(G/N\times G/M)\leq \alpha.
	\]
	Hence, $N\cap M\in \Sigma$.
\end{proof}
\begin{prop}\label{prop:residual_rank}
	Let $G$ be a \tdlcsc group and let $\mc{K}$ be a family of closed normal subgroups of $G$ with trivial intersection.  
	Then
	\[
	\xi(G) = {\sup_{K \in \mc{K}}}^+\xi(G/K).
	\]
	If $G$ is compactly generated and elementary, then there is $K \in \mc{K}$ such that $\xi(G) = \xi(G/K)$.
\end{prop}

\begin{proof}
	Every $K\in \mc{K}$ has rank less than or equal $\sup_{K \in \mc{K}}^+\xi(G/K)$. Applying Lemma~\ref{lem:ordinal_FF}, we may expand $\mc{K}$ to a filtering family $\tilde{\mc{K}}$ such that $\sup_{K\in \tilde{\mc{K}}}^+\xi(G/K)=\sup_{K \in \mc{K}}^+\xi(G/K)$. Replacing $\mc{K}$ with $\tilde{\mc{K}}$, we may assume, without loss of generality, that $\mc{K}$ is a filtering family.
	
	If $G/K$ is non-elementary for some $K \in \mc{K}$, then $G$ is also non-elementary, and there is nothing to prove. We thus assume that $G/K$ is elementary for all $K \in \mc{K}$, so by Theorem~\ref{thm:elementary_closure}(2), $G$ is elementary. Proposition~\ref{prop:xi_monotone} then ensures that $\xi(G/K) \le \xi(G)$ for all $K\in \mc{K}$, and it follows that $\sup_{K \in \mc{K}}^+\xi(G/K)\leq \xi(G)$. 
	
%
	Say that $\alpha+1=\sup^+_{K\in\mc{K}}G/K$. We argue that $\xi(G)\leq \alpha+1$ by induction on $\alpha$. The base case, $\alpha=0$, is immediate, since in this case $G=\{1\}$.	Suppose that $\alpha$ is a successor ordinal; say $\alpha = \beta+1$.  For any compactly generated open subgroup $O\leq G$ and $K\in \mc{K}$, $\xi(O/(O \cap K)) \le \beta+2$, so $\xi(\Res{}(O/(O \cap K))) \le \beta+1$.  The quotient of $O/(O \cap K)$ by $\ol{\Res{}(O)(O\cap K)}/O\cap K$ is a quotient of $O/\Res{}(O)$, so it is residually discrete. Hence, $\Res{}(O)(O \cap K)/(O \cap K) \le \Res{}(O/(O \cap K))$. There is thus an injective, continuous map from $\Res{}(O)/(\Res{}(O) \cap K)$ into $\Res{}(O/(O \cap K))$.  Appealing to Proposition~\ref{prop:xi_monotone}, we conclude that 
	\[
	\xi(\Res{}(O)/(\Res{}(O) \cap K)) \le \beta+1.
	\]
	 Since $\beta < \alpha$ and $\beta$ does not depend on $K$, the inductive hypothesis ensures that $\xi(\Res{}(O)) \le \beta+1 = \alpha$. It now follows that $\xi(G)\leq \beta+2=\alpha+1$, since $O$ is an arbitrary compactly generated open subgroup.
	
	Let us now consider the case when $\alpha$ is a limit ordinal, and since $\alpha>0$, the ordinal $\alpha$ is transfinite. Take $O\leq G$ a compactly generated open subgroup. If $\xi(O)$ is finite, then the inductive hypothesis implies that $\xi(O)=\sup_{K\in\mc{K}}^+(\xi(O/O\cap K))$. Since $\{\xi(O/O\cap K)\}_{K\in \mc{K}}$ is a bounded collection of finite ordinals, there is some $K\in \mc{K}$ such that $\xi(O)=\xi(O/O\cap K)$. Let us now suppose that $\xi(O)$ is transfinite. Since $\mc{K}$ is a filtering family, we may apply \cite[Theorem~3.3]{RW_EC_15} to find $K \in \mc{K}$ such that there is a compact $L \le K \cap O$ that is normal in $O$ and is open in $K \cap O$. Applying Proposition~\ref{prop:xi_monotone} and Lemma~\ref{lem:d-rank_extensions}, we deduce that
	\[
	\xi(O/K\cap O) \le \xi(O) \le \xi(L) + \xi(K \cap O/L) + \xi(O/K\cap O) \le 4 + \xi(O/K\cap O).
	\]
	Since $\xi(O)$ is transfinite, the only way to achieve such inequalities is if $\xi(O/K\cap O) = \xi(O)$.  We conclude that $\xi(O) \le \alpha+1$ for any compactly generated open subgroup $O$ and hence that $\xi(\Res{}(O)) \le \alpha$. It now follows that $\xi(G)\leq \alpha+1$, completing the induction.
	
	The second claim follows from the previous paragraph.	
\end{proof}

\subsection{Regionally SIN groups}

Elementary groups of decomposition rank $2$ have a natural characterization. A topological group is a \textbf{small invariant neighborhood}\index{SIN group} group, abbreviated SIN, if there is a basis at $1$ of conjugation invariant neighborhoods. A \tdlc group is a SIN group if and only if there is a basis at $1$ of compact open \textit{normal} subgroups.

We say that a \tdlc group is \textbf{regionally SIN}\index{regionally SIN group} if every open compactly generated subgroup is a SIN group.
\begin{lem}\label{lem:regionally_SIN}
For a non-trivial \tdlcsc group $G$, the following are equivalent:
\begin{enumerate}[(1)]
\item $G$ is regionally SIN;
\item $G$ is elementary with $\xi(G)=2$.
\end{enumerate}
\end{lem}

\begin{proof}
From the definition of decomposition rank, we see that $G$ is elementary with $\xi(G)=2$ if and only if every compactly generated open subgroup of $G$ is residually discrete. By \cite[Corollary 4.1]{CM11}, a compactly generated \tdlc group is residually discrete if and only if it is a SIN group, and the conclusion follows.
\end{proof}

\begin{lem}\label{lem:quasidiscrete:rank2}
Let $G$ be a quasi-discrete \lcsc group.  Then $G^\circ$ is central in $G$ and $G/G^\circ$ is elementary with $\xi(G/G^\circ)\le 2$.
\end{lem}

\begin{proof}
Every open subgroup contains $G^\circ$, so $G^\circ$ centralizes $\QZ(G)$.  Since $\QZ(G)$ is dense, $G^\circ$ is in fact central.

We see that $G/G^\circ$ is quasi-discrete, so we may assume $G^\circ = \triv$. The group $G$ is then a \tdlcsc group, and via \cite[Proposition 4.3]{CM11}, $G$ is regionally SIN. Lemma~\ref{lem:regionally_SIN} now ensures that $\xi(G)\le 2$.
\end{proof}

\subsection{The bounded rank residuals}\label{sec:bounded_rank}

Proposition~\ref{prop:residual_rank} suggests a natural source of closed characteristic subgroups.

\begin{defn} For a \tdlcsc group $G$ and an ordinal $\alpha \ge 1$, define the \defbold{rank-$\alpha$ residual} to be
	\[
	 \Res{\alpha}(G):=\bigcap\{K\normal G\mid K\text{ is closed and }\xi(G/K)\leq \alpha\}
	 \] 
	The rank-$\omega$ residual will be called the \defbold{finite rank residual} of $G$.
\end{defn} 

These residuals have many nice properties. First, the rank of the quotient $G/\Res{\alpha}(G)$ is restricted.

\begin{prop}\label{prop:residual_rank_bound}
For $G$ a \tdlcsc group and $\alpha$ an at most countable ordinal, the following hold:
\begin{enumerate}[(1)]
\item $\xi(G/\Res{\alpha}(G))\leq \sup^+(\alpha)$;
\item If $G$ is compactly generated and $\alpha$ is a limit ordinal, then $\xi(G/\Res{\alpha}(G)) < \alpha$.
\item If $G$ is compactly generated, then $\Res{2}(G)=\Res{}(G)$.
\end{enumerate}
\end{prop}

\begin{proof}
Claim (1) is immediate from Proposition~\ref{prop:residual_rank} and Lemma~\ref{lem:ordinal_FF}. Claim (3) is an easy exercise.

For claim (2), Proposition~\ref{prop:residual_rank} ensures there is a closed normal subgroup $K$ of $G$ such that $\xi(G/R) = \xi(G/K) \le \alpha$. Since the rank of $G/R$ is a successor ordinal, we conclude that $\xi(G/R)<\alpha$.
\end{proof}

The rank-$\alpha$ residual of $G$ can be expressed in terms of the rank-$\alpha$ residuals of compactly generated open subgroups.

\begin{prop}\label{prop:rank_residual:exhaustion}
Let $\alpha \ge 2$ be an at most countable ordinal, $G$ an \tdlcsc group, and $(O_i)_{i \in \Nb}$ be an increasing exhaustion of $G$ by compactly generated open subgroups. If $\alpha$ is a successor ordinal, then
\[
\Res{\alpha}(G) = \ol{\bigcup_{i \in \Nb} \Res{\alpha}(O_i)}.
\]
If $\alpha$ is a limit ordinal, then
\[
\Res{\alpha}(G) = \bigcap_{\beta < \alpha}\left(\ol{\bigcup_{i \in \Nb} \Res{\beta}(O_i)}\right).
\]
\end{prop}

\begin{proof}
Take $\beta$ an ordinal and $(O_i)_{i\in \Nb}$ an $\subseteq$-increasing exhaustion of $G$ by compactly generated open subgroups. The union $\bigcup_{i\in \Nb}\Res{\beta}(O_i)$ is an increasing union of subgroups, so $\bigcup_{i\in \Nb}\Res{\beta}(O_i)$ is a subgroup of $G$.  Given a compactly generated open subgroup $L$ of $G$, sufficiently large $i$ are such that $L \le O_i$, hence $\Res{\beta}(L) \le \Res{\beta}(O_i)$. We deduce that $\bigcup_{i \in \Nb} \Res{\beta}(O_i)$ is in fact the union of $\Res{\beta}(L)$ as $L$ ranges over \textit{all} compactly generated open subgroups of $G$. Hence, $\bigcup_{i\in\Nb}\Res{\beta}(O_i)$ is indeed a characteristic subgroup of $G$.

We are now ready to prove the claims. Suppose first that $\alpha = \beta+1$ is a successor ordinal and set  $R: = \ol{\bigcup_{i\in \Nb}\Res{\alpha}(O_i)}$. Suppose that $K\normal G$ is closed with $\xi(G/K)\leq \alpha$. By Proposition~\ref{prop:xi_monotone}, $O_i/(O_i \cap K)$ is a compactly generated open subgroup of $G/K$ of rank at most $\alpha$ for every $i$. We deduce that $\Res{\alpha}(O_i) \le K$, so $R \le K$. Therefore, $R \le \Res{\alpha}(G)$.

For the converse inclusion, fix a compactly generated open subgroup $L/R$ of $G/R$. For $i$ sufficiently large, $L \subseteq O_iR$, and there is a quotient map from $O_i/\Res{\alpha}(O_i)$ to $O_iR/R$. Applying Proposition~\ref{prop:xi_monotone}, we deduce that
\[
\xi(L/R) \le \xi(O_iR/R) \le \xi(O_i/\Res{\alpha}(O_i)) \le \alpha.
\]
Thus, $\xi(\Res{}(L/R)) \le \beta$, and since $L$ is arbitrary, it follows that $\xi(G/R) \le \beta+1 = \alpha$. We conclude that $R \ge \Res{\alpha}(G)$.  

Let us now consider the case when $\alpha$ is a limit ordinal. Set $R_{\beta}:=\Res{\beta}(G)$.  Since $G$ has no quotients of rank exactly $\alpha$, we see that
\[
R_{\alpha} = \bigcap_{\beta < \alpha}R_{\beta}.
\]
We can now use induction on $\alpha$ and the first part of the proposition applied to $R_{\beta}$ to write
\[
R_{\alpha} = \bigcap_{\beta <\alpha}\left(\ol{\bigcup_{i \in \Nb} \Res{\beta}(O_i)}\right),
\]
as required.
\end{proof}

\begin{cor}\label{cor:rank_residual:exhaustion:2}
Let $G$ be a \tdlcsc group and $(O_i)_{i \in \Nb}$ be an increasing exhaustion of $G$ by compactly generated open subgroups.  Then
\[
\Res{2}(G) = \ol{\bigcup_{i \in \Nb} \Res{}(O_i)}.
\]
\end{cor}

\begin{rmk}
	If $\alpha$ is a limit ordinal, it is not necessarily the case that $\Res{\alpha}(G) = \ol{\bigcup_{i \in \Nb} \Res{\alpha}(O_i)}$. For instance, if $G$ is a topologically characteristically simple \tdlcsc group with $\xi(G)= \omega+1$ (for an example, see \S\ref{ex:weak} below), then $\Res{\omega}(G) = G$ by Theorem~\ref{thm:lower_finite_rank} below, but the finite rank radical of every compactly generated open subgroup of $G$ is trivial.
\end{rmk}

In the case of the rank-$2$ residual, we also have an easy calculation for the rank of $\Res{2}(G)$ in terms of the rank of $G$.  This will put a restriction on the possibilities for the ranks of topologically characteristically simple \tdlcsc groups.

\begin{prop}\label{prop:rank_char_subgroup}
For $G$ a non-compact elementary \tdlcsc group, exactly one of the following occurs:
\begin{enumerate}[(1)]
\item $\xi(G) = \xi(\Res{2}(G)) = \alpha+1$ where $\alpha$ is a countable transfinite limit ordinal;
\item $\Res{2}(G)$ is a non-cocompact subgroup of $G$, and $\xi(G) = \xi(\Res{2}(G))+1$.
\end{enumerate}
\end{prop}

\begin{proof}
Let $(O_i)_{i \in \Nb}$ be an increasing exhaustion of $G$ by compactly generated open subgroups and set $R:=\Res{2}(G)$. By Corollary~\ref{cor:rank_residual:exhaustion:2}, $\ol{\bigcup_{i\in \Nb}\Res{}(O_i)} = R$ . We have 
\[
\xi(G)=\sup_{i\in \Nb}\xi(\Res{}(O_i))+1,
\]
and by Proposition~\ref{prop:d-rank:ascending}, we also have
\[
\xi(R) = {\sup_{i \in \Nb}}^+\xi(\Res{}(O_i)).
\]

Set $\beta := \sup_{i\in \Nb}\xi(\Res{}(O_i))$. If $\beta$ is a successor ordinal, then $\xi(R) = \beta$, and $\xi(G) = \beta+1$, so $\xi(G) = \xi(R)+1$.  In particular, $R$ is a proper subgroup of $G$ with $\xi(G/R) = 2$. In view of Lemma~\ref{lem:compact_ext}, the only way $R$ can be cocompact is if it is trivial; however, this would imply that $G$ is compact, contradicting the hypotheses.  We conclude that (2) holds and (1) does not hold.

If $\beta$ is a limit ordinal, then $\beta$ is transfinite. We conclude that $\xi(R) = \xi(G) = \beta+1$.  Condition (1) thus holds and (2) does not hold, completing the proof of the dichotomy.
\end{proof}

\begin{cor}\label{cor:rank_char_simple}
	If $G$ is a non-trivial topologically characteristically simple \tdlcsc group, then exactly one of the following holds:
	\begin{enumerate}[(1)]
	\item $\xi(G)=2$;
	\item $\xi(G)=\alpha+1$ for $\alpha$ some countable transfinite limit ordinal;
	\item $G$ is non-elementary.
	\end{enumerate}
\end{cor}

\begin{proof}
The three cases are clearly mutually exclusive. 

Suppose that (3) does not hold, so $G$ is an elementary group. Since $G$ is topologically characteristically simple, $\Res{2}(G) = \triv$, or $\Res{2}(G) = G$.  By Proposition~\ref{prop:rank_char_subgroup}, $\Res{2}(G) = \triv$ if and only if $G$ is elementary of rank $2$; the case $\xi(G)=1$ being eliminated by the hypothesis that $G \neq \triv$.  We thus deduce that (1) holds if $\Res{2}(G)=\triv$. If $\Res{2}(G)=G$, then $G$ is non-compact, since $\xi(G) > 2$. The dichotomy of Proposition~\ref{prop:rank_char_subgroup} thus ensures that $\xi(G)=\alpha+1$ for $\alpha$ some transfinite limit ordinal, hence (2) holds.
\end{proof}

\subsection{The rank-$2$ series}
We conclude this section by isolating a canonical series of closed characteristic subgroups in a \tdlcsc group.

\begin{defn} For a \tdlcsc group $G$, the \defbold{(lower) rank-$2$ series}, denoted by $(G_i)_{i\in \Nb}$, is defined as follows: $G_0:=G$ and $G_{i+1}:=\Res{2}(G_i)$. We call a \tdlcsc group \textbf{rank-2 resolvable} if the rank-$2$ series stabilizes at $\{1\}$ in finitely many steps.
\end{defn}

We can think of the rank-$2$ series as an analogue of the derived series, with rank-$2$ groups playing the role of abelian groups.  The next results show this analogy can be pushed further.


\begin{thm}\label{thm:lower_finite_rank}
	For $G$ a \tdlcsc group and $(G_i)_{i\in \Nb}$ the rank-$2$ series for $G$, the following hold:
	\begin{enumerate}[(1)]
		\item For each $i\in \Nb$, $G_i=\Res{i+1}(G)$.
		\item $\Res{\omega}(G)=\bigcap_{i\in \Nb}G_i$.
		\item If $G$ is compactly generated, then $\Res{\omega}(G) = G_j$ for some $j \in \Nb$; in particular, the series $(G_i)_{i\in \Nb}$ stabilizes in finitely many steps.
	\end{enumerate}
\end{thm}

\begin{proof}

For claim (1), we argue by induction on $i\geq 1$. The base case, $i=1$, holds by definition. Suppose that the claim holds up to $i$ and consider $G_{i+1}$. Let $K\normal G$ be such that $\xi(G/K)\leq i+1$ and put $R:=\pi^{-1}(\Res{2}(G/K))$ where $\pi:G\rightarrow G/K$ is the usual projection. We see that $G_1\leq R$ and that $\xi(R/K)\leq i$ by Proposition~\ref{prop:rank_char_subgroup}. The inductive hypothesis now implies that the $i$-th term of the rank-$2$ series for $R$ is contained in $K$. Since $G_1\leq R$, we conclude that the $i$-th term of the rank-$2$ series of $G_1$ is contained in $K$. That is to say, $G_{i+1}\leq K$, hence $G_{i+1}\leq \Res{i+1}(G)$. 

For the converse inclusion, the induction hypothesis implies that $\xi(G_1/G_{i+1})\leq i+1$, since $G_{i+1}$ is the $i$-th term of the rank-$2$ series for $G_1$. It is easy to see that $\Res{2}(G/G_{i+1})=G_1/G_{i+1}$, hence applying Proposition~\ref{prop:rank_char_subgroup}, we deduce that $\xi(G/G_{i+1})\leq i+2$. Therefore, $G_{i+1}=\Res{i+1}(G)$, completing the induction.

For claim (2), the quotient $G/G_j$ is the largest quotient of $G$ of rank at most $j+1$ by claim (1).  In particular, if $G/K$ is a quotient of $G$ of finite rank, then $K \ge G_j$ for some $j$. Hence, $\bigcap_{j\in \Nb}G_j\leq \Res{\omega}(G)$. The other inclusion is obvious, so we conclude that $\Res{\omega}(G) = \bigcap_{j\in \Nb }G_j$.

For the final claim, Proposition~\ref{prop:residual_rank_bound} ensures that $G/\Res{\omega}(G)$ has finite rank, hence $\Res{\omega}(G) = G_j$ for some $j$.
\end{proof}

The next corollary is now immediate from Theorem~\ref{thm:lower_finite_rank} and Proposition~\ref{prop:residual_rank_bound}.
\begin{cor} 
	Let $G$ be a \tdlcsc group. Then $G$ is an elementary group with finite rank if and only if $G$ is rank-$2$ resolvable.
\end{cor} 

If $G$ does have finite rank, we see that the bound in Lemma~\ref{lem:d-rank_extensions} is achieved by the extensions coming from the rank-$2$ series.  Thus at least in the case of groups of finite rank, the bound in Lemma~\ref{lem:d-rank_extensions} is the best possible.

\begin{cor}
Let $G$ be an elementary \tdlcsc group such that $\xi(G) < \omega$, let $(G_i)_{i\in \Nb}$ be the rank-$2$ series for $G$ and let $i \in \Nb$.  Then
\[
\xi(G) = (\xi(G_i) - 1) + \xi(G/G_i).
\]
\end{cor}

\begin{proof}
By Theorem~\ref{thm:lower_finite_rank} we have $G_i=\Res{i+1}(G)$.  It follows by Proposition~\ref{prop:residual_rank_bound} that $\xi(G/G_i) = j +1$ where $0 \le j \le i$.  Then we see that $G_i = \Res{i+1}(G) = \Res{j+1}(G) = G_j$ and that we have a strictly descending chain of subgroups $G_0 > G_1 > \dots > G_j$.  In particular, $2 \le \xi(G_k) < \omega$ for all $0 \le k < j$, so for these $k$, then $(\xi(G_k) - 1)$ is not a limit ordinal.  Repeated application of Proposition~\ref{prop:rank_char_subgroup} yields $\xi(G) = \xi(G_j) + j$; in turn, $\xi(G_j) + j = (\xi(G_j) - 1) + \xi(G/G_j)$.
\end{proof}

\section{Factorization of normal compressions in \lcsc groups}\label{sec:normal_compressions}

Suppose that $G$ and $H$ are \lcsc groups with $\psi: G \rightarrow H$ a normal compression map.  Since $G$ and $H$ are Polish groups, Proposition~\ref{prop:Polish:semidirect} and Theorem \ref{thm:psi-compression_factor_rel} supply a factorization of $\psi$ as
\[
\xymatrix{
 G\ar_{\psi}[rd] \ar^{\alpha}[r]& G\rtimes_{\psi} H\ar^{\beta}[d] \\
& H }
\]
where $\alpha$ is a closed embedding with normal image, $\beta$ is a quotient map, and $\ker(\beta)$ is isomorphic to $G$ as a topological group. 

Our aim in the present section is to use the structure of locally compact groups to factorize $\psi$ into a sequence of closed embeddings with normal image and quotient maps, such that all kernels and cokernels appearing in the sequence are small. This will allow group-theoretic properties to be transported between $G$ and $H$.

\subsection{The totally disconnected case}
We first prove a factorization result for normal compression maps $\psi: G \rightarrow H$ of \lcsc groups where $\psi$ restricts to an isomorphism on the identity components. This is, in particular, the case when $G$ and $H$ are totally disconnected.

An \defbold{FC-group}\index{FC-group} is a group in which every conjugacy class is finite; equivalently, an FC-group is such that every element has a finite index centralizer.  The set of finite conjugacy classes of a group $G$ forms a subgroup called the \defbold{FC-center}\index{FC-center} of $G$.  We shall require some basic properties of FC-groups and elements with finite conjugacy classes.

\begin{lem}\label{lem:basic:FCgroup}Let $G$ be a group.
\begin{enumerate}[(1)] 
\item If $x \in G$ is an element of finite order such that $|G:\CC_G(x)|$ is finite, then $x$ is contained in a finite normal subgroup of $G$. \textup{(Dietzmann's Lemma \cite[14.5.7]{R96}.)}
\item If $G$ is an FC-group, then $[G,G]$ and $G/\Z(G)$ are torsion. Hence, every element of $[G,G]$ or $G/\Z(G)$ is contained in a finite normal subgroup. \textup{(\cite[Section 14.5]{R96})}
\end{enumerate}
\end{lem}

Proving our factorization theorem requires two technical lemmas.

\begin{lem}\label{lem:delta_subgroup}
Let $G$ and $H$ be \lcsc groups with $\psi:G\rightarrow H$ a normal compression.  For $U\leq H$ and $K\leq G$ closed subgroups, form $H^{\rtimes}: = G \rtimes_{\psi} U$ where $U$ acts via the $\psi$-equivariant action and set
\[
\Delta_K := \{(k\inv,\psi(k)) \in H^{\rtimes} \mid k \in K\}.
\]
Then,
\begin{enumerate}[(1)]
\item The set $\Delta_K$ is a closed subgroup of $H^{\rtimes}$ with $\Delta_K\leq \CC_{H^{\rtimes}}(G\times \{1\})$.

\item If $\psi$ restricts to a homeomorphism from $K$ to $\psi(K)$ and $\psi(K)$ is a closed normal subgroup of $U$, then $\Delta_K$ is normal in $H^{\rtimes}$, and $q\circ \iota$ is a closed map where $q: H^{\rtimes} \rightarrow H^{\rtimes}/\Delta_K$ is the quotient map and $\iota: G \rightarrow H^{\rtimes}$ is the usual embedding.
\end{enumerate}
\end{lem}

\begin{proof}
For $(1)$, easy calculations verify that $\Delta_K$ is a subgroup and that $\Delta_K$ commutes with $G \times \{1\}$. That $\Delta_K$ is closed follows since $K$ is closed and $\psi$ is continuous.

For $(2)$, since $\psi(K)$ is normal in $U$, the subgroup $\Delta_K$ is normalized by $\triv \times U$, and part $(1)$ implies $\Delta_K$ is centralized by $G \times \{1\}$. The subgroup $\Delta_K$ is thus normal in $H^{\rtimes}$. To show that $q \circ \iota$ is a closed map, it suffices to show $ \iota(R)\Delta_K$ is closed in $H^{\rtimes}$ for every closed subset $R$ of $G$. Fix such an $R$. The set $G \times \psi(K)\subseteq H^{\rtimes}$ is closed, and there is a continuous map $\gamma: G\times \psi(K) \rightarrow G$ defined by $(g,\psi(k))\mapsto gk$. Moreover, $ \iota(R)\Delta_K$ is the preimage $\gamma\inv(R)$. We deduce that $\iota(R)\Delta_K$ is a closed subset of $H^{\rtimes}$, proving that $q \circ \iota$ is a closed map.
\end{proof}

\begin{lem}\label{lem:almost_connected}
Let $G$ and $H$ be \lcsc groups with $\psi:G\rightarrow H$ a normal compression map. If $U \in \U(G)$, then $\N_H(\psi(U))$ is open in $H$.  Furthermore, for every $L\in \U(H)$, there exists $W \le_o U$ such that $\psi(W)$ is normalized by $L$.
\end{lem}

\begin{proof}
Consider $\U(G)$ as a collection of subgroups of $G \times \triv$ in the semidirect product $H^{\rtimes} = G \rtimes_{\psi} H$, which is a Polish group by Proposition~\ref{prop:Polish:semidirect}. Since $U \in \U(G)$ is a closed subgroup of $H^{\rtimes}$, the normalizer $\N_{H^{\rtimes}}(U)$ is closed. On the other hand, the $\psi$-equivariant action is continuous, so $H^{\rtimes}\acts \U(G)$ as permutations. Recalling that  $\U(G)$ is countable, we deduce that $\N_{H^{\rtimes}}(U)$ is a closed subgroup of $H^{\rtimes}$ with countable index. The Baire Category Theorem now implies that $\N_{H^{\rtimes}}(U)$ is open in $H^{\rtimes}$, and it follows that $\N_H(\psi(U))$ is open in $H$.

For the second claim, consider $L$ an almost connected subgroup of $H$ and fix $U \in \U(G)$. For each $l\in L$, the subgroup $\psi\inv(l\psi(U)l^{-1})$ is exactly the image of $U$ by $l$ under the $\psi$-equivariant action of $H$ on $G$, hence $\psi\inv(l\psi(U)l^{-1})$ is again an element of $\U(G)$. On the other hand, every open subgroup of $L$ has finite index. We deduce that $|L:\N_L(\psi(U))|$ is finite, and thus,  
\[
W := \psi\inv\left(\bigcap_{l \in L}l\psi(U)l\inv\right)
\]
is an element of $\U(G)$. Furthermore, $\psi(W)$ is normalized by $L$, verifying the claim.
\end{proof}

A \defbold{covering map}\index{covering map} between Polish groups $G$ and $H$ is a surjective, continuous homomorphism $\psi: G \rightarrow H$ such that there exist identity neighborhoods $U$ in $G$ and $V$ in $H$ for which $\psi\rest_U:U\rightarrow V$ is a homeomorphism. Covering maps between \lcsc groups can be characterized as follows: $\psi: G \rightarrow H$ is a covering map if and only if it is a surjective continuous group homomorphism such that the kernel is discrete.

We now obtain a decomposition of a normal compression $\psi$ of \tdlcsc groups as $\psi = \beta \circ \alpha$, where $\alpha$ is a closed embedding and $\beta$ is a covering map.

\begin{thm}\label{thm:compression_factoring:tdlcsc}
Suppose that $G$ and $H$ are \lcsc groups with $\psi: G \rightarrow H$ a normal compression map and that $\psi$ restricts to a homeomorphism from $G^\circ$ to $H^\circ$.  Then there is an \lcsc group $\tilde{H}$, a closed embedding $\alpha: G \rightarrow \tilde{H}$ with normal image, and a covering map $\beta: \tilde{H} \rightarrow H$ such that
\begin{enumerate}[(1)]
\item $\psi=\beta \circ \alpha$;
\item $\tilde{H}/\alpha(G)$ is compact;
\item $\ker(\beta)$ lies in the FC-center of $\tilde{H}$ and centralizes $\alpha(G)$; and
\item $\tilde{H} = \ol{\alpha(G)\ker(\beta)}$.
\end{enumerate}
If $G$ and $H$ are \tdlcsc groups, then $\tilde{H}$ is also a \tdlcsc group.
\end{thm}

\begin{proof}
Fix $O\in \U(H)$ and form $H^{\rtimes} := G \rtimes O$ where $O$ acts on $G$ via the $\psi$-equivariant action.  By Proposition~\ref{prop:Polish:semidirect}, $H^{\rtimes}$ with the product topology is an \lcsc group. Fix $W\in \U(G)$ such that $\psi(W)\normal O$, as given by Lemma~\ref{lem:almost_connected}, and observe that $V:=W\rtimes O$ is an almost connected open subgroup of $H^{\rtimes}$. 

We now define the group $\tilde{H}$ and the maps $\alpha$ and $\beta$. For $W$ as fixed in the previous paragraph, Lemma~\ref{lem:delta_subgroup} ensures 
\[
\Delta_{W}:=\{(w^{-1},\psi(w))\in H^{\rtimes}\mid w\in W\}
\]
is a closed normal subgroup of $H^{\rtimes}$. We put $\tilde{H}:=H^{\rtimes}/\Delta_W$. For the map $\alpha$, let $\iota: G \rightarrow H^{\rtimes}$ be the obvious inclusion and let $q:H^{\rtimes}\rightarrow \tilde{H}$ be the usual projection. We set $\alpha:=q\circ \iota$. To define $\beta$, consider the quotient map $\pi: H^{\rtimes} \rightarrow H$ as in Theorem~\ref{thm:psi-compression_factor_rel}. The kernel of $\pi$ is exactly $\Delta_G$, and since $\Delta_{G}$ contains $\Delta_W$, the map $\pi$ induces the desired map $\beta:\tilde{H}\rightarrow H$. 

 Let us now check the claimed properties. The maps $\alpha$ and $\beta$ are clearly continuous. It is also clear that $\alpha$ is injective with normal image. That $\alpha$ is a closed map is given by Lemma~\ref{lem:delta_subgroup}. To see $\beta$ is a covering map, we remark that $(g\inv,\psi(g)) \in \Delta_W V$ implies $g \in W$, so $\Delta_G \cap (\Delta_W V)= \Delta_W$. Therefore, $\ker(\beta)\cap q(V)=\{1\}$, so $\ker(\beta)$ is discrete, verifying that $\beta$ is a covering map.

Claims $(1)$ and $(4)$ are immediate from Theorem~\ref{thm:psi-compression_factor_rel}. For claim $(2)$, the subgroup $\iota(G)\Delta_W$ contains the connected component of $H^{\rtimes}$, since $\psi(G^{\circ})=H^{\circ}$. Recalling that $O$ is almost connected, it follows that the quotient 
\[
H^{\rtimes}/\iota(G)\Delta_W\simeq \tilde{H}/\alpha(G)
\]
is compact.

For claim $(3)$, $\ker(\beta)$ equals $\Delta_{G}/\Delta_{W}$. The group $\Delta_{G}$ is centralized by $\iota(G)$, hence $\ker(\beta)$ is centralized by $\alpha(G)$, verifying the second claim of $(3)$. In view of $(2)$, every element of $\ker(\beta)$ has a cocompact centralizer in $\tilde{H}$. On the other hand, $\ker(\beta)$ is a discrete normal subgroup of $\tilde{H}$, so every element indeed has an open centralizer. We thus deduce every $x \in \ker(\beta)$ has a finite index centralizer and so has a finite conjugacy class.  That is to say, $\ker(\beta)$ is in the FC-center of $\tilde{H}$.  
\end{proof}

A locally compact group $G$ is \textbf{compactly presented}\index{compactly presented} if there is a presentation for $G$ as an abstract group $G= \langle S\mid R \rangle$ such that $S$ has compact image in $G$ and $R$ consists of relators with bounded length. For such groups, we obtain a stronger factorization.

\begin{cor}\label{cor:compression_factoring:compactly_presented}
Suppose that $G$ and $H$ are \lcsc groups with $\psi: G \rightarrow H$ a normal compression map and that $\psi$ restricts to a homeomorphism from $G^\circ$ to $H^\circ$. If $G$ is compactly generated and $H$ is compactly presented, then $\tilde{H}$, $\alpha$, and $\beta$ given by Theorem~\ref{thm:compression_factoring:tdlcsc} can be chosen such that $\tilde{H}/\alpha(G)$ is compact abelian and $\ker(\beta)$ is central, finitely generated, and torsion free.
\end{cor}

\begin{proof}
Take $O\in \U(H)$, find $W\in \U(G)$ such that $\psi(W)\normal O$, and form $\tilde{H}=G\rtimes O/\Delta_W$ as given by Theorem~\ref{thm:compression_factoring:tdlcsc}. The group $\tilde{H}$ is compactly generated since it contains $G$ as a cocompact subgroup.

In view of \cite[Proposition 8.A.10]{CdH14}, $\ker(\beta)$ is the normal closure of a compact set. Applying Theorem~\ref{thm:compression_factoring:tdlcsc}, $\ker(\beta)$ is discrete and contained in the $FC$-center of $\tilde{H}$, hence $\ker(\beta)$ is indeed \textit{finitely generated} by some finite subset $X$. Each $x\in X$ is centralized by a finite index open subgroup of $\tilde{H}$, so $C:=\CC_{\tilde{H}}(\ker(\beta))$ is a finite index subgroup of $\tilde{H}$. 

Letting $q:G\rtimes O\rightarrow \tilde{H}$ be the usual projection, we may find $O'\leq O$ open with finite index such that $q(\triv \times O')\leq C$. Since $q\circ\iota(G)=\iota(G)\Delta_W/\Delta_W$ centralizes $\ker(\beta)$, we may assume further that $\psi(W)\leq O'$. We now apply Theorem~\ref{thm:compression_factoring:tdlcsc} to $O'$ and $W$ to form $\tilde{H}'$, $\alpha'$, and $\beta'$. The group $\tilde{H}'$ is an open subgroup of $\tilde{H}$ contained in $C$, and thus, $\ker(\beta')$ is central in $\tilde{H}'$. We thus deduce that $\ker(\beta')$ is a central finitely generated subgroup. 

Observing that $\ker(\beta')=\Delta_G/\Delta_W$, the projection on the second coordinate gives an injective homomorphism $\ker(\beta')\rightarrow O'/\psi(W)$. The group $O'/\psi(W)$ is a profinite group since $\psi(G^{\circ})=H^{\circ}$. We may thus find $O''\leq O'$ a finite index open subgroup such that $\psi(W)\leq O''$ and $O''/\psi(W)$ avoids the torsion elements of the image of $\ker(\beta')$ in $O'/\psi(W)$. Applying Theorem~\ref{thm:compression_factoring:tdlcsc} to $O''$ and $W$, we form $\tilde{H}''$, $\alpha''$, and $\beta''$. The group $\tilde{H}''$ is an open subgroup of $\tilde{H}'$ contained in $C$, and thus, $\ker(\beta'')$ is central in $\tilde{H}''$. Furthermore, the map $\Delta_G/\Delta_W\rightarrow O''/\psi(W)$ must have torsion free image. We thus deduce that $\ker(\beta'')$ is central, finitely generated, and torsion free. Without loss of generality, we now assume $O=O''$.

Appealing to Theorem~\ref{thm:compression_factoring:tdlcsc} again, there is a normal compression map $\theta:\ker(\beta)\rightarrow \tilde{H}/\alpha(G)$. Since $\ker(\beta)$ is abelian, we conclude that $\tilde{H}/\alpha(G)$ is abelian, verifying the corollary.
\end{proof}

\subsection{The general case}
We now relax the requirement that $\psi$ restricts to an isomorphism from $G^\circ$ to $H^\circ$. To deal with the connected components, we require several preliminary results.

For an \lcsc group $G$, the group of continuous automorphisms, denoted by $\Aut(G)$, carries a Polish topology, namely the modified compact-open topology\index{modified compact open topology}; this topology is alternatively called the Braconnier topology or the Arens $g$-topology.  When we consider $\Aut(G)$ as a topological group, we always mean with respect to this to this topology.

\begin{lem}\label{lem:continuous_aut_ext} Suppose that $G$ and $H$ are \lcsc groups with $\psi:G\rightarrow H$ a continuous embedding with dense image and that $A\leq \Aut(G)$ is closed. If every $ \alpha\in A$ has $\epsilon(\alpha)\in \Aut(H)$ such that $\psi\circ \alpha=\epsilon(\alpha)\circ \psi$, then the map $\epsilon:A\rightarrow \Aut(H)$ is continuous.
\end{lem}
\begin{proof}
That $\psi$ has a dense image ensures the map $\epsilon$ is a group homomorphism. To show that $\epsilon$ is continuous, it thus suffices to verify that it is Borel measurable, since $A$ and $\Aut(H)$ are Polish groups; see, for example, \cite[(9.10)]{K95}.

A sub-basis at $1$ for $\Aut(H)$ is given by the sets
\[
\mf{A}_H(K,U):=\{\beta\in \Aut(H)\mid \forall x\in K\; \beta(x)x^{-1}\in U\text{ and }\beta^{-1}(x)x^{-1}\in U\}
\]
where $K\subseteq H$ is compact and $U\subseteq H$ is an open neighborhood of $1$. We may further restrict to compact sets $K$ which are the closure of an open set. To show $\epsilon$ is Borel measurable, it is therefore enough to show the preimage of $\mf{A}_H(K,U)$ is Borel-measurable in $A$ for $U$ open and $K=\ol{O}$ with $O$ open and relatively compact. Fix such a $K$ and $U$.

Fix a countable basis $\mc{B}$ for the topology on $H$. Since $U$ is open and $H$ is a regular space, we may cover $U$ by $B\in \mc{B}$ such that $\ol{B}\subseteq U$. Let $\mc{C}\subseteq \mc{B}$ be a cover of $U$ by such $B$; we may additionally assume $\mc{C}$ is closed under finite unions. The set $\psi^{-1}(K)$ is a non-trivial closed set in $G$, so we may write $\psi^{-1}(K)=\bigcup_{i\in \Nb}J_i$ where the $J_i$ are compact sets. 

For $B\in \mc{C}$ and $i\in \Nb$, we now consider
\[
\Sigma(i,B):=\mf{A}_G\left(J_i,\psi^{-1}(B)\right)\cap A.
\]
We claim 
\[
\epsilon^{-1}(\mf{A}_H(K,U))=\bigcup_{B\in \mc{C}}\bigcap_{i \in \Nb} \Sigma(i,B).
\]
Since the $\Sigma(i,B)$ are open sets in $A$, this will prove $\epsilon$ is Borel measurable.

Take $\beta\in \bigcup_{B\in \mc{C}}\bigcap_{i \in \Nb} \Sigma(i,B)$ and say that $B\in \mc{C}$ is such that $\beta \in \Sigma(i,B)$ for all $i$. Take $x \in \psi(\psi^{-1}(K))$ and suppose that $x = \psi(y)$ for $y \in \psi^{-1}(K)$. Since $\epsilon(\beta)^{\pm 1}\circ\psi =\psi\circ\beta^{\pm 1}$, we have
\[
\epsilon(\beta)^{\pm 1}(x)x^{-1} = \psi(\beta^{\pm 1}(y)y^{-1}) \in B.
\]
The set $K$ is the closure of an open set, so $\psi(\psi^{-1}(K))$ is dense in $K$. We thus deduce that $\epsilon(\beta)^{\pm 1}(x)x^{-1} \in \ol{B}\subseteq U$ for all $x\in K$. That is to say, $\epsilon(\beta)\in \mf{A}_H(K,U)$.

On the other hand, take $\beta\in \epsilon^{-1}(\mf{A}_H(K,U))$. Since $K$ is compact, $\mc{C}$ is closed under finite unions and the maps $K\rightarrow U$ by $x\mapsto \epsilon(\beta)^{\pm 1}(x)x^{-1}$ are continuous, there is some $B\in \mc{C}$ for which $B$ contains the image of both maps. For each $i$ and $y\in J_i$, we thus have $\epsilon(\beta)^{\pm 1}(\psi(y))\psi(y)^{-1}\in B$. Hence, $\beta^{\pm 1}(y)y^{-1}\in \psi^{-1}(B)$. It now follows that $\beta\in \bigcap_{i \in \Nb}\Sigma(i,B)$, and the claim is demonstrated.	
\end{proof}

A locally compact group $G$ is a \defbold{(CA) group}\index{(CA) group} if the natural map $\rho:G\rightarrow \Aut(G)$ has a closed image, where $\Aut(G)$ is endowed with the modified compact-open topology.

\begin{thm}[Zerling, {\cite[Main Theorem]{Zerling_75}}\label{thm:Zerling}]
Suppose that $G$ is a connected locally compact group that is not a (CA) group.  Then there is a (CA) locally compact connected group $N$, a torus $T$ that acts continuously on $N$, and a vector group $V$ that maps densely into $T$, such that:
\begin{enumerate}[(1)]
\item $P := N \rtimes T$ is a (CA) locally compact group.
\item $G \simeq N \rtimes V$.
\item Every automorphism of $G$ as a topological group extends to an automorphism of $P$ as a topological group.
\end{enumerate}
\end{thm}

Given a normal compression map $\psi: G \rightarrow H$ of connected \lcsc groups, we have control over the derived group of $H$.

\begin{cor}\label{cor:Zerling:derived}
If $G$ and $H$ are connected \lcsc groups with $\psi: G \rightarrow H$ a normal compression map, then $\ol{[H,H]} \le \psi(G)\Z(H)$.
\end{cor}

\begin{proof}
Let $P$ and $N$ be as given in Theorem~\ref{thm:Zerling} for $G$ and let $\epsilon: \Aut(G) \rightarrow \Aut(P)$ be the group homomorphism, necessarily unique, taking each automorphism of $G$ to its extension to an automorphism of $P$. Lemma~\ref{lem:continuous_aut_ext} ensures $\epsilon$ is continuous.  Define $\rho: H \rightarrow \Aut(G)$ to be the map induced by the $\psi$-equivariant action of $H$ on $G$.  

The group $\psi(G)$ is dense in $H$, so by continuity, $\epsilon\rho\psi(G)$ is dense in $\epsilon\rho(H)$. Since $P$ is a (CA) group, we see that $\Inn(P)$ is a closed subgroup of $\Aut(P)$ that contains $\epsilon\rho\psi(G)$.  Hence, $\epsilon\rho(H) \le \Inn(P)$.  

We thus obtain a continuous homomorphism $\theta: H \rightarrow P/\Z(P)$.  In particular, $\theta(\ol{[H,H]}) \le \ol{[P,P]}\Z(P)/\Z(P)$, and since $\ol{[P,P]} \le N \le G$, every element of $\ol{[H,H]}$ induces an inner automorphism of $G$. We conclude that $\ol{[H,H]} \le \psi(G)\CC_H(\psi(G))$.  Since $\CC_H(\psi(G)) = \Z(H)$, we indeed have $\ol{[H,H]} \le \psi(G)\Z(H)$, as desired.
\end{proof}

The following special case is immediate.

\begin{cor}\label{cor:Zerling:perfect}
Suppose that $G$ and $H$ are connected \lcsc groups with $\psi: G \rightarrow H$ a normal compression map. If $G$ is topologically perfect and $H$ is centerless, then $\psi$ is an isomorphism of topological groups.
\end{cor}

%

We now prove the desired factorization result for arbitrary normal compression maps between \lcsc groups.

\begin{thm}\label{thm:compression_factoring}
Suppose that $G$ and $H$ are \lcsc groups with $\psi: G \rightarrow H$ a normal compression map.  Then there are continuous homomorphisms and \lcsc groups
\[
G=G_0 \xrightarrow{\alpha_1} K_1 \xrightarrow{\beta_1} G_1  \xrightarrow{\alpha_2} K_2 \xrightarrow{\beta_2} G_2 \xrightarrow{\alpha_3} K_3 \xrightarrow{\beta_3} H 
\]
such that the following properties hold:
\begin{enumerate}[(1)]
\item $\psi = \beta_3 \circ \alpha_3 \circ \beta_2 \circ \alpha_2 \circ \beta_1 \circ \alpha_1$;
\item The $\alpha_i$ are closed embeddings with normal image, the $\beta_i$ are quotient maps, and $\beta_3$ is a covering map;
\item The groups $K_1/\alpha_1(G)$ and $\ker(\beta_1)$ are abelian;
\item The groups $K_2/\alpha_2(G_1)$ and $\ker(\beta_2)$ are solvable of derived length at most $3$;
\item $K_3/\alpha_3(G_2)$ is compact, and $K_3 = \ol{\alpha_3(G_2)\ker(\beta_3)}$; and
\item $\ker(\beta_3)$ lies in the FC-center of $K_3$ and centralizes $\alpha_3(G_2)$.
\end{enumerate}
\end{thm}

\begin{proof}We build the maps $\alpha_i$ and $\beta_i$ and the groups $K_i$ and $G_i$ in stages.
	
Define $R_1 := \ol{\psi(G^\circ)} \le H^\circ$ and $K_1 := G \rtimes_{\psi} \Z(R_1)$ where $\Z(R_1)$ acts via the $\psi$-equivariant action. Set 
\[
\Delta_1 := \{(g\inv,\psi(g)) \in K_1 \mid g \in G\}
\]
and $G_1: = K_1/\Delta_1$.  Let $\alpha_1$ be the natural closed embedding of $G$ into $K_1$ and let $\beta_1$ be the natural quotient map from $K_1$ to $G_1$. The groups $K_1$ and $G_1$ are \lcsc groups by Proposition~\ref{prop:Polish:semidirect} and Lemma~\ref{lem:delta_subgroup}, and the cokernel of $\alpha_1$ and the kernel of $\beta_1$ are abelian. 

Let $\psi_1:G_1\rightarrow H$ be such that $\psi_1 \circ \beta_1 \circ\alpha_1 = \psi$; that is, $\psi_1$ is the map $(g,r)\Delta_1\mapsto \psi(g)r$. The map $\psi_1$ is a normal compression map. Additionally, the map 
\[
\psi_1 \circ \beta_1\rest_{\alpha_1(G^{\circ})}:\alpha_1(G^{\circ})\rightarrow R_1
\]
is a normal compression map from the connected \lcsc group $\alpha_1(G^\circ)$ to the connected \lcsc group $R_1$.  By Corollary~\ref{cor:Zerling:derived}, we have that $\ol{[R_1,R_1]} \le \psi(G^\circ)\Z(R_1)$. Both $\psi(G^\circ)$ and $\Z(R_1)$ are subgroups of $\psi_1(G_1)$, hence $\ol{[R_1,R_1]} \le \psi_1(G_1)$.  The group $T_1:=\psi^{-1}_1(\ol{[R_1,R_1]})$ is thus isomorphic as a topological group to $\ol{[R_1,R_1]}$ via $\psi_1$. 

Define $A_2: = G_1 \rtimes_{\psi_1} H^\circ$, set $\Delta_{2} := \{(t\inv,\psi_1(t)) \in A_2 \mid t \in T_1\}$, and put $K_2: = A_2/\Delta_{2}$.  By Lemma~\ref{lem:delta_subgroup}, $K_2$ is an \lcsc group, and the natural map $\alpha_2: G_1 \rightarrow K_2$ is a closed embedding. Letting $\rho$ be the composition $H^{\circ}\rightarrow A_2\rightarrow K_2$, the image $\alpha_2(G_1)$ contains $\rho(\overline{[R_1,R_1]})$. Lemma~\ref{lem:disconnected_to_connected} ensures $[[H,H^\circ],H^\circ] \le R_1$, and thus, $K_2/\alpha_2(G_1)$ is solvable of derived length at most $3$. 

Define $\Delta_3 := \{(g\inv,\psi_1(g)) \in A_2 \mid g \in G_1\}$ and $G_2: = A_2/\Delta_3$. Define the map $\beta_2$ to be the natural quotient map from $K_2$ to $G_2$. The kernel of $\beta_2$ is $\ker(\beta_2) = \Delta_3/\Delta_2$. The projection onto the second coordinate of $\Delta_3$ induces an injective homomorphism $\Delta_3/\Delta_2\rightarrow H^{\circ}/\ol{[R_1,R_1]}$. Since $[[H,H^\circ],H^\circ] \le R_1$, we conclude that $H^{\circ}/\ol{[R_1,R_1]}$ is solvable of derived length at most $3$. Hence, $\ker \beta_2$ is solvable of derived length at most $3$.
%
%

As with $G_1$, we obtain a normal compression map $\psi_2: G_2 \rightarrow H$ such that 
\[
\psi_2\circ\beta_2 \circ\alpha_2\circ \beta_1\circ \alpha_1 = \psi.
\]
The group $A_2$ already contains a copy of $H^\circ$, so we have $\psi_2(G_2) \ge H^\circ$. The map $\psi_2$ thus restricts to an isomorphism from $(G_2)^\circ$ to $H^\circ$.  We finally obtain the maps $\alpha_3$ and $\beta_3$ by applying Theorem~\ref{thm:compression_factoring:tdlcsc} to $\psi_2:G_2\rightarrow H$, completing the construction. 
\end{proof}

\section{Properties invariant under normal compressions}\label{sec:invariant}
We now explore to what extent various properties are preserved under normal compressions. Several of the properties studied here and the derived preservation theorems will be used later to analyze the structure of chief factors. 

\subsection{Compact generation}\label{sec:generation}
Given a \tdlcsc group $G$ and a normal compression $H$ of $G$ with $H$ compactly generated, it \textit{does not} follow that $G$ is compactly generated.  For example, $\prod_{i \in \Nb}F_i$ is a normal compression of $\bigoplus_{i \in \Nb}F_i$ for any sequence of finite groups $(F_i)_{i\in \Nb}$, and $\Rb$ is a normal compression of $\Qb$. Our next theorem shows that if $H$ is compactly generated, then profinite quotients and connected central factors of $H$ are the only potential obstacles to the compact generation of $G$.

\begin{thm}\label{thm:compression:generation}
	Suppose that $G$ and $H$ are \lcsc groups with $\psi: G \rightarrow H$ a normal compression map and set $R:=  \psi\inv(H^\circ)$.
	\begin{enumerate}[(1)]
		\item If $G$ is compactly generated, then $H$ is compactly generated.
		\item If $H$ is compactly generated and $R/\ol{[R,G]}$ is compactly generated, then there exists an increasing exhaustion $(N_i)_{i \in \Nb}$ of $G$ by compactly generated open normal subgroups such that $\ol{\psi(N_i)}$ is cocompact and normal in $H$ for each $i \in \Nb$. 
	\end{enumerate}
\end{thm}

\begin{proof}
	Claim $(1)$ is immediate.
	
	For $(2)$, we see that $R \ge G^\circ$, and Lemma~\ref{lem:disconnected_to_connected} implies $[R,G] \le G^\circ$. The quotient $R/\ol{[R,G]}$ is compactly generated, so $R/G^\circ$ is also compactly generated. It now follows that $G$ is compactly generated if and only if $G/R$ is compactly generated. We thus reduce to $G/R$, and since $\psi$ induces a normal compression map $\tilde{\psi}:G/R\rightarrow H/H^{\circ}$, we may suppose that both $G$ and $H$ are totally disconnected. 
	
	Using Lemma~\ref{lem:almost_connected}, we fix $W \in \U(G)$ and $U \in \U(H)$ such that $\psi(W)$ is a normal subgroup of $U$. Let $\{g_i\}_{i\in\Nb}$ be a countable dense subset of $G$ with $1=g_0$ and set $S_i: = \{\psi(g_0),\dots,\psi(g_i)\}$.  By Proposition~\ref{prop:factor}, there is a finite symmetric set $A_i\subseteq H$ containing $ S_i$ such that $UA_iU=A_iU$ and $A_iU$ generates $H$. Since $\psi(G)$ is dense, we can choose $A_i$ to be a subset of $\psi(G)$; we can also choose $(A_i)_{i \in \Nb}$ to be an increasing sequence.
	
	The group $U$ acts continuously on $G$ via the $\psi$-equivariant action. Under this action, the orbit $U.\psi^{-1}(A_i)$ is a compact subset of $G$, so we obtain a compactly generated open subgroup $N_i:=\grp{U.\psi^{-1}(A_i),W}$ that contains $\{g_0,\dots,g_i\}$. The sequence $(N_i)_{i \in \Nb}$ is thus an increasing exhaustion of $G$ by compactly generated open subgroups. 
	
	Fix $i$, set $N := N_i$, and set $A:= A_i$.  The group $\psi(N)$ is generated by $A^U\cup \psi(W)$. That $\psi(W)\normal U$ ensures that $U$ normalizes $\psi(N)$, and plainly, $A\subseteq \psi(N)$.  Since $H=\grp{A,U}$, we deduce that $\psi(N)\normal H$ and hence $N\normal G$. The group $U$ maps surjectively onto $H/\ol{\psi(N)}$, so $\ol{\psi(N)}$ is also cocompact in $H$. The proof is now complete.
\end{proof}

\begin{cor}
	Suppose that $G$ and $H$ are \tdlcsc groups with $H$ a normal compression of $G$.  Then $G$ is compactly generated if and only if $H$ is compactly generated and every discrete quotient of $G$ is finitely generated.
\end{cor}
\begin{proof}
	If $G$ has an infinitely generated discrete quotient, then certainly $G$ is not compactly generated, and if $H$ is not compactly generated, then $G$ is not compactly generated by Theorem~\ref{thm:compression:generation}.  Conversely, if $H$ is compactly generated and every discrete quotient of $G$ is finitely generated, then in particular $G/N_i$ is finitely generated for all $i$ where $(N_i)_{i \in \Nb}$ is the exhaustion given by Theorem~\ref{thm:compression:generation}. It now follows that $G = N_j$ for $j$ large enough, and thus, $G$ is compactly generated.
\end{proof}

\begin{cor}\label{cor:compactlygenerated}
	Suppose that $G$ and $H$ are non-abelian non-compact topologically characteristically simple \lcsc groups with $H$ a normal compression of $G$. Then $G$ is compactly generated if and only if $H$ is compactly generated.
\end{cor}
\begin{proof}
	Since both $G$ and $H$ are topologically characteristically simple groups, Corollary~\ref{cor:con_association} ensures either both are totally disconnected or both are connected. Connected groups are compactly generated, so the corollary is immediate in the latter case. We thus assume that $G$ and $H$ are \tdlcsc groups. 
	
	The forward implication is given by Theorem~\ref{thm:compression:generation}. Conversely, suppose that $H$ is compactly generated and let $\psi:G\rightarrow H$ be a normal compression map. Appealing to Theorem~\ref{thm:compression:generation}, there is $N\normal G$ open and compactly generated such that $\ol{\psi(N)}$ is cocompact in $H$. The quotient $H/\ol{\psi(N)}$ is therefore a profinite group.
	
	If this quotient is non-trivial, then $\Res{}(H)$ is a proper characteristic closed subgroup of $H$, hence $\Res{}(H)=\{1\}$. We thus deduce that $H$ is residually discrete. Applying \cite[Corollary 4.1]{CM11}, $H$ admits a compact open normal subgroup $L$. The locally elliptic radical of $H$ is then a non-trivial characteristic closed subgroup, hence $H$ is a compactly generated locally elliptic group. We conclude that $H$ is compact, which contradicts the hypotheses. 
	
	The group $\ol{\psi(N)}$ thus equals $H$, and Proposition~\ref{prop:normal_compression} implies $[G,G]\leq N$. Since $G$ is non-abelian and topologically characteristically simple, we conclude that $N=G$. In particular, $G$ is compactly generated.	
\end{proof}

\subsection{Decomposition rank}\label{sec:elementary2}

The property of being elementary and the associated decomposition rank are outright invariant under normal compressions between \tdlcsc groups.

\begin{prop}\label{prop:compression:elementary_rank}
	Suppose that $G$ and $H$ are \tdlcsc groups with $H$ a normal compression of $G$.  Then $G$ is elementary with decomposition rank $\alpha$ if and only if $H$ is elementary with decomposition rank $\alpha$. 
\end{prop}

\begin{proof}
	Let $\psi: G\rightarrow H$ be a normal compression map. Let $\tilde{H}$, $\alpha:G\rightarrow H$, and $\beta:\tilde{H}\rightarrow H$ be as given by Theorem~\ref{thm:compression_factoring:tdlcsc}.
	
	Suppose first that $G$ is an elementary group. Since $G$ embeds as a cocompact normal subgroup of $\tilde{H}$, Lemma~\ref{lem:compact_ext} implies that $\xi(\tilde{H}) = \xi(G)$. The group $H$ is a quotient of $\tilde{H}$, so in view of Proposition~\ref{prop:xi_monotone}, we deduce that $H$ is elementary with $\xi(H) \le \xi(G)$. On the other hand, $\psi:G\rightarrow H$ is a continuous injective homomorphism, and Proposition~\ref{prop:xi_monotone} thus shows that $\xi(G) \le \xi(H)$. We conclude that $\xi(H)=\xi(G)$.
	
	Conversely, if $H$ is elementary, then Proposition~\ref{prop:xi_monotone} ensures that $G$ is elementary. The forward direction then implies that $\xi(G)=\xi(H)$.
\end{proof}

\subsection{Amenability}\label{sec:amenability}
A locally compact group $G$ is \defbold{amenable}\index{amenable group} if every continuous affine action of $G$ on a non-empty compact convex subset of a locally convex topological vector space has a fixed point; we direct the reader to \cite[Appendix G]{BHV08} for an account of the basic theory of amenable groups. Amenability is a second property outright preserved under normal compressions.

Let us recall several basic facts.

\begin{lem}[See {\cite[Appendix G]{BHV08}}]\label{lem:prelim:amenability}
	Let $G$ be an \lcsc group.
	\begin{enumerate}[(1)]
		\item Every compact or virtually solvable \lcsc group is amenable.
		\item If $K$ is a closed normal subgroup of $G$, then $G$ is amenable if and only if both $K$ and $G/K$ are amenable.
		\item If $G$ is amenable, then every closed subgroup of $G$ is amenable.
		\item If $H$ is an amenable \lcsc group and $\psi: H \rightarrow G$ is a continuous homomorphism, then $\ol{\psi(H)}$ is amenable.
		\item If $G = \ol{\bigcup_{\alpha \in I}A_\alpha}$ with $(A_\alpha)_{\alpha \in I}$ a directed system of closed amenable subgroups of $G$, then $G$ is amenable.
	\end{enumerate}
\end{lem}

Observe that the FC-groups are also amenable, with respect to the discrete topology; this is an easy exercise using Lemma~\ref{lem:prelim:amenability}. 

\begin{prop}\label{prop:amenability:invariance}
	Suppose that $G$ and $H$ are \lcsc groups with $H$ a normal compression of $G$.  Then $G$ is amenable if and only if $H$ is amenable.
\end{prop}

\begin{proof}
	If $G$ is amenable, then $H$ is amenable by Lemma~\ref{lem:prelim:amenability}.
	
	Conversely, suppose that $H$ is amenable. Applying Theorem~\ref{thm:compression_factoring}, we obtain a factorization of $\psi: G \rightarrow H$ as $\psi = \beta_3\circ\alpha_3\circ\beta_2\circ\alpha_2\circ\beta_1\circ\alpha_1$ where
	\[
	G=G_0 \xrightarrow{\alpha_1} K_1 \xrightarrow{\beta_1} G_1  \xrightarrow{\alpha_2} K_2 \xrightarrow{\beta_2} G_2 \xrightarrow{\alpha_3} K_3 \xrightarrow{\beta_3} H 
	\]
	The kernel $\ker(\beta_3)$ is an FC-group, so it is amenable. Since $H$ is also amenable, Lemma~\ref{lem:prelim:amenability} ensures that $K_3$ is amenable. Each $\gamma \in \{\alpha_1,\alpha_2,\alpha_3,\beta_1,\beta_2\}$ is either a closed embedding or a quotient map with solvable kernel. In either case, if the range of $\gamma$ is amenable, then so is the domain.  Since $K_3$ is amenable, we deduce that $G$ is amenable.
\end{proof}

\subsection{Quasi-discreteness}
Just as with compact generation, discreteness is \textit{not} a normal compression invariant, even if we restrict to compressions between non-abelian topologically characteristically simple groups; cf. $\bigoplus_{i\in \Zb}A_5$ and $\prod_{i\in \Zb}A_5$. There is a more general property that is close to discreteness and admits better preservation properties: A \lcsc group $G$ is said to be \defbold{quasi-discrete}\index{quasi-discrete} if $G$ has a dense quasi-center $\QZ(G):=\{g\in G\mid \CC_G(g)\text{ is open}\}$. 

\begin{rmk}
For an \lcsc group $G$, the quasi-center $\QZ(G)$ is the union of all countable conjugacy classes in $G$, so $\QZ(G)$ is the union of all countable normal subgroups of $G$. This gives the following characterization: \textit{For $G$ an \lcsc group, $G$ is quasi-discrete if and only if $G$ contains a countable dense normal subgroup.} The quasi-discrete \lcsc groups are thus exactly the \lcsc groups that arise as normal compressions of countable discrete groups.
\end{rmk}

Our proof requires a classical commutator identity; recall that $[x,y,z]:=[[x,y],z]$, $[a,b] = ab a\inv b\inv$, and $a^b = bab\inv$.

\begin{fact}[Hall--Witt identity]
	For $G$ a group and any $a,b,c\in G$, the following holds:
	\[
	[a\inv,c,b]^a[b\inv,a,c]^b[c\inv,b,a]^c=1.
	\]
\end{fact}

\begin{thm}\label{thm:compression:quasidiscrete}
	Suppose that $G$ and $H$ are \lcsc groups with $H$ a normal compression of $G$. 
	\begin{enumerate}[(1)]
		\item If $H$ is quasi-discrete, then $G/\ol{\QZ(G)}$ is metabelian.
		\item If $G$ is quasi-discrete, then $H/\ol{\QZ(H)}$ is metabelian.
	\end{enumerate}
\end{thm}

\begin{proof} Let $\psi:G\rightarrow H$ be a normal compression map.
	
	For $(1)$, set $Q := \QZ(H)$. For $g \in G$ such that $\psi(g) \in Q$, the injectivity of $\psi$ ensures that $\psi\inv(\CC_H(\psi(g))) = \CC_G(g)$, so $\CC_G(g)$ is open. We deduce that $\psi\inv(Q) \le \QZ(G)$.
	
	Setting $R := \ol{\psi(G) \cap Q}$, the subgroups $\psi(G)$ and $Q$ commute modulo $R$, and since $\psi(G)$ and $Q$ are both dense normal subgroups of $H$, the group $H/R$ is abelian. The group $G/K$ with $K:=\psi^{-1}(R)$ is therefore abelian. The  map $\psi:K\rightarrow R$ is a normal compression map, and $K\cap \ol{\QZ(G)}$ is a closed normal subgroup of $K$ with dense image in $R$. Proposition~\ref{prop:normal_compression} ensures that  $\ol{[K,K]}\leq K\cap \ol{\QZ(G)}$. It now follows that $G/\ol{\QZ(G)}$ is metabelian, verifying $(1)$.
	
	For $(2)$, each $g\in \QZ(G)$ centralizes $G^{\circ}$, so $G^{\circ} \le \Z(G)$, which implies that $Z := \ol{\psi(G^\circ)}$ is central in $H$. We thus obtain a normal compression map $G/G^\circ\rightarrow H/Z$ with $G/G^\circ$ quasi-discrete.  Let $K:=\pi^{-1}(\QZ(H/Z))$ with $\pi:H\rightarrow H/Z$ the usual projection.  For each $m,n\in K$, there is an open subgroup $O\leq H$ such that $[O,m]\cup [O,n]\subseteq Z$.  Therefore, $[o\inv,m,n] = [n\inv,o,m] = 1$, and the Hall-Witt identity with $a:=o$, $c:=m$, and $b:=n$ implies $[m\inv,n,o]=1$ for all $o \in O$. We deduce that $[m\inv,n]\in \QZ(H)$ for all $m,n\in K$, and thus, $\ol{[K,K]}\leq \ol{\QZ(H)}$.
	
	To show that $H/\ol{\QZ(H)}$ is metabelian, it suffices to verify that $H/\ol{K}$ is abelian. We observe that $H/\ol{K}\simeq (H/Z)/(\ol{K}/Z)$, so it is enough to show that $(H/Z)/\ol{\QZ(H/Z)}$ is abelian. By passing to $H/Z$, we assume that $Z=\{1\}$ and argue that $H/\ol{\QZ(H)}$ is abelian. Given this reduction, $G$ is a \tdlc group.
	
	Fix $n,m\in \QZ(G)$ and $W\in \U(G)$ such that $W$ centralizes $n$ and $m$. Appealing to Lemma~\ref{lem:almost_connected}, we may find $V\in \U(H)$ such that $\psi(W)\normal V$. We now consider $R:=\grp{\psi(n),\psi(m),\psi(W),V}\leq H$. The subgroup $L:=\psi(W)$ is a compact normal subgroup of $R$ that centralizes $\psi(m)$ and $\psi(n)$. Furthermore, $(\psi(G)\cap R)L/L$ is countable and normal in $R/L$, so it is quasi-central in $R/L$. We may thus find an open subgroup $V'$ of $V$ such that $[V',\psi(n)]\cup[V',\psi(m)]\subseteq L$. Fixing $v\in V'$, we see that $[v\inv,\psi(m),\psi(n)]=1$ and $[\psi(n)\inv,v,\psi(m)]=1$. The Hall--Witt identity with $a:=v$, $b:=\psi(n)$, and $c:=\psi(m)$ thus implies that $[\psi(m)\inv,\psi(n),v] = 1$. It follows that $[\psi(m)\inv,\psi(n)]=\psi([m\inv,n])$ has an open centralizer in $H$. 
	
	We have now established that $\psi([\QZ(G),\QZ(G)])\leq \QZ(H)$. On the other hand, $\QZ(G)$ is dense in $G$, so by continuity, $\psi(\QZ(G))$ is a dense subgroup of $H$.  We deduce that
	\[
	\ol{[H,H]}\leq \ol{\psi([\QZ(G),\QZ(G)])}.
	\]
	The group $H/\ol{\QZ(H)}$ is thus abelian, completing the proof of (2).
\end{proof}

\begin{cor}\label{cor:compression:quasidiscrete}
	Suppose that $G$ and $H$ are topologically perfect \lcsc groups with $H$ a normal compression of $G$. Then $H$ is quasi-discrete if and only if $G$ is quasi-discrete.
\end{cor}


\section{Generalized direct products}

We now relate locally compact generalized central products to local direct products, showing in particular that any \tdlc quasi-product is a normal compression of a local direct product.  We then establish a structure theorem for locally compact groups of semisimple type.  Our discussion concludes by considering generalized central products of elementary groups.

\subsection{Preliminaries}

For a topological group $G$ and $J$ a set of closed subgroups, we define $G_J :=  \cgrp{N \mid N \in J}$. 

\begin{defn} 
Suppose that $G$ is a topological group and $\mc{S}\subseteq \mc{N}(G)$. The set $\mc{S}$ is a \defbold{generalized central factorization}\index{generalized central factorization} of $G$, if $G_{\mc{S}}=G$ and $[N,M] = \triv$ for any distinct $N$ and $M$ in $\mc{S}$. In such a case, $(G,\mc{S})$ is said to be a \defbold{generalized central product}\index{generalized central product}; we will also say that $G$ is a generalized central product, when the factorization is implicit.

We say $\mc{S}$ is a \defbold{quasi-direct factorization}\index{quasi-direct factorization} of $G$ if $G_{\mc{S}}=G$ and $\mc{S}$ has the following topological independence property:
\[
\forall X\subseteq \mathcal{P}(\mc{S}):\; \bigcap X=\emptyset \Rightarrow \bigcap_{A\in X}G_A=\{1\}. 
\]
In such a case, $(G,\mc{S})$ (or $G$) is said to be a \textbf{quasi-product}\index{quasi-product}. A subgroup $H$ of $G$ is a \defbold{quasi-factor}\index{quasi-factor} of $G$ if it is an element of some quasi-direct factorization of $G$.

A factorization $\mc{S}$ is \textbf{non-trivial} if $|\mc{S}|\geq 2$.
\end{defn}

We shall make use of a general property of generalized central products.
\begin{thm}[{\cite[Theorem 1.7]{RW_P_15}}]\label{thm:directproduct:quasiproduct}
	Let $G$ be a topological group with $\mc{S}$ a non-trivial generalized central factorization of $G$. Then the diagonal map 
		\[
		d: G \rightarrow \prod_{N\in \mc{S}}G/G_{\mc{S}\setminus\{N\}}
		\]
		is a continuous homomorphism such that $d(G) \cap G/G_{\mc{S}\setminus \{N\}}$ is dense in $G/G_{\mc{S}\setminus \{N\}}$ for every $N\in \mc{S}$ and $\ker(d)$ is central in $G$. Furthermore, $\mc{S}$ is a quasi-direct factorization of $G$ if and only if $d$ is injective.
\end{thm}

Theorem~\ref{thm:directproduct:quasiproduct} has a useful, immediate corollary.
\begin{cor}\label{cor:quasi-product_char}
	 Let $G$ be a topological group with $\mc{S}$ a non-trivial generalized central factorization of $G$. Then $\mc{S}$ is a quasi-direct factorization if and only if $\bigcap_{N\in \mc{S}}G_{\mc{S}\setminus\{N\}}=\{1\}$.
\end{cor}

At various points, we will take quotients of generalized central products. There is a natural quasi-product one can extract from such a quotient:

\begin{prop}[{\cite[Proposition 4.7]{RW_P_15}}]\label{prop:quasiproduct:quotient}
Let $G$ be a topological group with $\mc{S}$ a non-trivial generalized central factorization of $G$ and let $N$ be a closed normal subgroup of $G$.  Set
\[
M := \bigcap_{S \in \mc{S}}\ol{G_{\mc{S} \setminus \{S\}}N}.
\]
Then $M/N$ is central in $G/N$, and
\[
\mc{S}/M := \{\ol{SM}/M \mid S\in\mc{S}\text{ and } S \not\le M\}
\]
is a quasi-direct factorization of $G/M$.
\end{prop}

The canonical examples of quasi-products are the local direct products:

\begin{defn}
	Suppose that $(G_i)_{i\in I}$ is a set of topological groups and that there is a distinguished open subgroup $O_i\leq G_i$ for each $i\in I$.  The \textbf{local direct product}\index{local direct product} of $(G_i)_{i\in I}$ over $(O_i)_{i\in I}$ is denoted by $\bigoplus_{i\in I}\left(G_i,O_i\right)$ and defined to be
	\[
	\left\{ f:I\rightarrow \bigsqcup_{i\in I} G_i\mid f(i)\in G_i \text{ and }f(i)\in O_i\text{ for all but finitely many }i\in I\right\}
	\]
	with the group topology such that the natural embedding of $\prod_{i\in I}O_i$  with the product topology is continuous and open.
	\end{defn}

In the setting of locally compact groups, a notion of convergence becomes useful to study generalized products.

\begin{defn}
A set $\mc{S}$ of closed subgroups of a topological group $G$ is \defbold{compactly convergent} if for every compact set $C$ and every identity neighborhood $U$ in $G$, there exists a cofinite subset $\mc{S}'\subseteq \mc{S}$ such that $G_{\mc{S}'}\cap C \subseteq U$.
\end{defn}

This notion of convergence arises in quasi-products.

\begin{lem}\label{lem:quasiproduct:local_convergence}
If $(G,\mc{S})$ is a quasi-product, then $\mc{S}$ is compactly convergent in $G$.
\end{lem}

\begin{proof}
Via Theorem~\ref{thm:directproduct:quasiproduct}, the diagonal map $d: G \rightarrow \prod_{N\in\mc{S}}G/G_{\mc{S}\setminus \{N\}}$ is injective and continuous. For any subset $J\subseteq \mc{S}$, we have $d(G_J) \le \prod_{N \in J}G/G_{\mc{S}\setminus\{N\}}$.  For any identity neighborhood $V$ in $\prod_{N\in\mc{S}}G/G_{\mc{S}\setminus \{N\}}$, there thus exists a cofinite subset $\mc{S}'$ of $\mc{S}$ such that $d(G_{\mc{S}'}) \subseteq V$.

Fix $C$ a compact set and $U$ an identity neighborhood in $G$; without loss of generality, we may assume that $U$ is open and that $1 \in C$.  Since $d$ is injective and continuous, it restricts to a homeomorphism $C\rightarrow d(C)$.  There is then an open subset $V$ of $\prod_{N\in\mc{S}}G/G_{\mc{S}\setminus \{N\}}$ such that $d(C) \cap V = d(C \cap U)$, and since $1 = d(1) \in d(C \cap U)$, the set $V$ is an identity neighborhood.  We deduce there is a cofinite subset $\mc{S}'$ of $\mc{S}$ such that $d(G_{\mc{S}'}) \subseteq V$, and hence
\[
d(G_{\mc{S}'} \cap C) = d(G_{\mc{S}'}) \cap d(C) \subseteq V \cap d(C) = d(C \cap U) \subseteq d(U).
\]
Therefore, $G_{\mc{S}'}\cap C\subseteq U$, so $\mc{S}$ is compactly convergent.
\end{proof}

With topologically perfect centerless quasi-products $G$, there is a connection between the connectedness of $G$ and the connectedness of its quasi-factors.

\begin{lem}\label{lem:centerless_quotient}
	Let $G$ be a topological group. If $A$ is a centerless closed normal subgroup of $G$, then $G/\CC_G(A)$ is centerless.
\end{lem}

\begin{proof}
	Let $\pi: A \rightarrow G/\CC_G(A)$ be the restriction of the usual quotient map.  For $h\CC_G(A) \in \Z(G/\CC_G(A))$, we see that $[a,h] \in A \cap \CC_G(A)$ for all $a\in A$. Since $\Z(A) = \triv$, it must be the case that $[a,h] = 1$. We conclude that $h \in \CC_G(A)$, and thus, $h\CC_G(A)$ is the trivial coset.
\end{proof}

\begin{lem}\label{lem:tdlc_quasi-factors}
	Let $G$ be a centerless Polish group with $\mc{S}$ a quasi-direct factorization of $G$.  If every $N \in \mc{S}$ is topologically perfect and totally disconnected, then so is $G$.
\end{lem}

\begin{proof}
	It is immediate that $G$ is topologically perfect and that every $N \in \mc{M}$ is centerless.  For $N \in \mc{M}$, the quotient $G/\CC_G(N)$ is a normal compression of $N$. The quotient is then topologically perfect, and it is centerless by Lemma~\ref{lem:centerless_quotient}. Applying Corollary~\ref{cor:con_association}, we conclude that $G/\CC_G(N)$ is totally disconnected, and thus, $G^{\circ} \le \CC_G(N)$.  It now follows that $G^\circ \le \Z(G)=\{1\}$, so $G$ is totally disconnected.
\end{proof}

Our primary interest in generalized direct products is their connection with groups of semisimple type.
\begin{defn}\label{defn:component}
	Let $G$ be a topological group.  A \defbold{component}\index{component} of $G$ is a closed subgroup $M$ of $G$ such that the following conditions hold:
	\begin{enumerate}[(a)]
		\item $M$ is normal in $ \ngrp{M}$.
		\item $M/\Z(M)$ is non-abelian.
		\item Whenever $K$ is a closed normal subgroup of $\ngrp{M}$ such that $K < M$, then $K$ is central in $\ngrp{M}$. 
	\end{enumerate}
	The \defbold{layer}\index{layer} $E(G)$ of $G$ is the closed subgroup generated by the components of $G$. 
\end{defn}

\begin{defn}\label{defn:semisimple}
	A topological group $G$ is of \defbold{semisimple type}\index{semisimple type} if $G = E(G)$. We say $G$ is of \defbold{strict semisimple type}\index{strict semisimple type} if in addition $\Z(G) = \triv$.
\end{defn}

Each component of a semisimple type group is topologically perfect via \cite[Observation 5.3]{RW_P_15}, so groups of semisimple type are topologically perfect. Furthermore, they are very close to products of topologically simple groups.

\begin{thm}[See {\cite[Theorem 1.11]{RW_P_15}}]\label{thm:min_normal:simple}
	Let $G$ be a non-trivial topological group and let $\mc{M}$ be the collection of components of $G$.
	\begin{enumerate}[(1)]
		\item The group $G$ is of semisimple type if and only if $(G,\mc{M})$ is a generalized central product and each $M\in \mc{M}$ is non-abelian and central-by-topologically simple.
		\item The group $G$ is of strict semisimple type if and only if $(G,\mc{M})$ is a quasi-product and each $M\in \mc{M}$ is non-abelian and topologically simple. 
	\end{enumerate}
\end{thm}

\subsection{Local direct models}\label{sec:localdirect}

\begin{defn}
Let $G$ be a locally compact group with $\mc{S}$ a generalized central factorization of $G$.  A \defbold{local direct model}\index{local direct model} for $(G,\mc{S})$ is a locally compact local direct product of the form $\bigoplus_{\mc{S}}(G) := \bigoplus_{N\in \mc S}(N, N\cap O)$ where $O$ is an open subgroup of $G$ together with a continuous homomorphism $\phi: \bigoplus_{\mc{S}}(G) \rightarrow G$ such that $\phi\rest_N=\mathrm{id}_N$ for all $N \in \mc{S}$. When there is no ambiguity in doing so, we will refer to either the group $\bigoplus_{\mc{S}}(G)$ or the homomorphism $\phi$ as the local direct model.
\end{defn}

Let us make a few basic observations about local direct models.

\begin{lem}\label{lem:LDM_normal}Let $G$ be a locally compact group with $\mc{S}$ a generalized central factorization of $G$. If $\phi: \bigoplus_{\mc{S}}(G) \rightarrow G$ is a local direct model for $G$, then $\img(\phi)$ is a dense normal subgroup of $G$.
\end{lem}
\begin{proof}
Since $\bigoplus_{\mc{S}}(G)$ is locally compact, there is a cofinite $\mc{S}'\subseteq \mc{S}$ such that $N\cap O$ is compact for all $N\in \mc{S}'$. It now follows that
	\[
	\img(\psi)=\left\langle\mc{S}\cup \cgrp{N\cap O\mid N\in \mc{S}'}\right\rangle.
	\]
Therefore, $\img(\psi)$ is dense in $G$. Furthermore, since $O$ normalizes $N \cap O$ for all $N \in \mc{S}$, we see that $O \le \N_G(\img(\psi))$. The normalizer $N_G(\img(\psi))$ is then a dense open subgroup and thus equals $G$. 
\end{proof}

\begin{prop}\label{prop:LDQ_injective}
Let $G$ be a locally compact group with $\mc{S}$ a quasi-direct factorization of $G$. If $\phi: \bigoplus_{\mc{S}}(G) \rightarrow G$ is a local direct model for $G$, then $\phi$ is injective. Hence, $G$ is a normal compression of $\bigoplus_{\mc{S}}(G)$. 
\end{prop}
\begin{proof}
Suppose for contradiction that $f\in \ker(\phi)$ is non-trivial. Since $\phi\rest_S$ is injective for all $S\in \mc{S}$, there is some $N\in \mc{S}$ such that $f(N)\neq 1$ and $f\rest_{\mc{S}\setminus\{N\}}\neq 1$. Consider the functions $g, h \in \bigoplus_{S\in \mc{S}}(N,N\cap O)$ defined as follows:
\[
g(M) := \begin{cases}
 f(M) &\mbox{if } M=N \\ 
 1 & \mbox{else} 
\end{cases};\text{ and } h(M) := \begin{cases}
 f(M) &\mbox{if } M\neq N \\ 
 1 & \mbox{else.} 
\end{cases}
\]

By construction, $f=gh$, and since $1=\phi(f)=\phi(g)\phi(h)$, we deduce that $\phi(h)\in N\cap G_{\mc{S}\setminus \{N\}}$. This contradicts the independence property enjoyed by $\mc{S}$. 
\end{proof}

We now establish the universal property for the local direct model.

\begin{thm}\label{thm:LDM_universal_prop}
Let $G$ be a locally compact group with $\mc{S}$ a compactly convergent generalized central factorization and suppose that $(G,\mc{S})$ admits a local direct model. If $H=\bigoplus_{N\in \mc{S}}(N,W_N)$ is a locally compact local direct product with $\psi:H\rightarrow G$ a continuous map such that $\psi\rest_N=\mathrm{id}_N$ for all $N\in\mc{S}$, then $H\leq \bigoplus_{\mc{S}}(G)$, and the inclusion $\iota:H\rightarrow \bigoplus_{\mc{S}}(G)$ is a continuous map making the following diagram commute:
\[
\xymatrixcolsep{3pc}\xymatrix{
\bigoplus_{\mc{S}}(G) \ar^{\phi}[r] & G\\
H \ar_{\psi}[ru] \ar^{\iota}[u]& }
\]
\end{thm}
\begin{proof}
If $\mc{S}$ is finite, the universal property is immediate. Let us suppose that $\mc{S}$ is infinite and let $O\leq G$ be open such that $\bigoplus_{N\in \mc{S}}(N,N\cap O)=\bigoplus_{\mc{S}}(G)$. Since $H$ is locally compact, there is a cofinite $\mc{T}\subseteq \mc{S}$ such that $\prod_{N\in \mc{T}}W_N$ is compact, and hence $W_H:=\psi(\prod_{N\in \mc{T}}W_N)$ is compact. That $\mc{S}$ is compactly convergent ensures there is a cofinite $\mc{S}'\subseteq \mc{T}$ such that $G_{\mc{S}'}\cap W_H\leq O$. For $f\in H$, it now follows that $f(N)\in G_{\mc{S}'}\cap W_H\subseteq O$ for cofinitely many $N\in \mc{S}$, hence $f\in \bigoplus_{\mc{S}}(G)$. We thus deduce that $H\leq \bigoplus_{\mc{S}}(G)$. One easily checks that $\iota:H\rightarrow \bigoplus_{\mc{S}}(G)$ is continuous, verifying the theorem.
\end{proof}
It follows local direct models are unique, provided they exist.
\begin{cor}\label{cor:LDM_unique}
	Let $G$ be a locally compact group with $\mc{S}$ a compactly convergent generalized central factorization. If there is a local direct model for $(G,\mc{S})$, then it is unique up to isomorphism.
\end{cor}

It remains to identify which factorizations of locally compact groups admit a local direct model.

\begin{lem}\label{lem:quasiproduct:compact}
Let $G$ be a compact group with $\mc{S}$ a generalized central factorization of $G$ and let $I$ be the directed set of finite subsets of $\mc{S}$. For a sequence $(h_N)_{N\in \mc{S}}\in \prod_{N\in \mc{S}}N$ and $J\in I$, set $h_J:=\prod_{N\in J}h_N$ where the product is taken in $G$. If $\mc{S}$ is compactly convergent, then for any sequence $(h_N)_{N\in \mc{S}}\in \prod_{N\in \mc{S}}N$ the associated net $\hat{h}:=(h_J)_{J\in I}$ in $G$ converges. 
\end{lem}

\begin{proof}
As $G$ is compact, there exists a limit point $k\in G$ of a subnet of $\hat{h}$.  Given an identity neighborhood $U$ in $G$, there exists $J \in I$ such that $G_{\mc{S} \setminus J} \subseteq U$, since $\mc{S}$ is compactly convergent and $G$ is compact. It is then the case that $h_JU = h_{J'}U$ for all $J' \in I$ such that $J \subseteq J'$, and considering $J'$ from our convergent subnet, we deduce that $h_JU=kU$. Hence, $\hat{h}$ converges.
\end{proof}

Say that a locally compact group is \defbold{almost \tdlc}\index{almost \tdlc} if the connected component is compact. 

\begin{thm}\label{thm:LDF} 
Let $G$ be an almost \tdlc group. If $\mc{S}$ is a compactly convergent generalized central factorization of $G$, then $(G,\mc{S})$ admits a local direct model.
\end{thm}
\begin{proof} 
Fix $U\in \U(G)$ and consider $K:=\cgrp{N\cap U\mid N\in \mc{S}}$. The subgroup $K$ is compact and admits a generalized central factorization $\{N\cap U\mid N\in \mc{S}\}$. This factorization is also compactly convergent since $\mc{S}$ is compactly convergent. 

Form the local direct product $\bigoplus_{N\in \mc{S}}(N,N\cap U)$ and let $I$ be the directed set consisting of the finite subsets of $\mc{S}$ ordered by inclusion. An element $h := (h_N)_{N \in \mc{S}}$ of $\bigoplus_{N\in \mc{S}}(N,N\cap U)$ gives a net $\hat{h}: = (h_J)_{J \in I}$ in $G$ where $h_J:=\prod_{N\in J}h_N$. Since $h$ has only finitely many coordinates that do not lie in $U$, there is a finite $J\subseteq \mc{S}$ such that $h_J^{-1}h_{J'}\in K\leq U$ for all $J'\supseteq J$. The net $(h_J)^{-1}h_{J'}$ for $J'$ extending $J$ converges by Lemma~\ref{lem:quasiproduct:compact}, hence the net $\hat{h}$ converges. 

We may now define $\psi:\bigoplus_{N\in \mc{S}}(N,N\cap U)\rightarrow G$ via $\psi(h):=\lim \hat{h}$.  It is easily verified that $\psi$ is a continuous homomorphism, that $\psi$ restricts to the identity map on each factor $N$, and that $\psi$ has a dense image in $G$. 
\end{proof}

In view of Lemma~\ref{lem:quasiproduct:local_convergence}, we deduce the following corollary:
\begin{cor}\label{cor:tdlc_LDM}
Let $G$ be an almost \tdlc group. If $\mc{S}$ is a quasi-direct factorization of $G$, then $(G,\mc{S})$ admits a local direct model.
\end{cor}

There are two classes of groups for which any quasi-product is actually a direct product, so in particular, the quasi-product is its own local direct model. The first follows immediately from Theorem~\ref{thm:LDF}.

\begin{cor}\label{cor:quasiproduct:compact}
If $G$ is a compact group with $\mc{S}$ a quasi-direct factorization, then $G \simeq \prod_{N\in \mc{S}}N$.
\end{cor}

\begin{prop}\label{prop:quasi_connected}
Let $G$ be a connected \lcsc group that is topologically perfect and centerless. If $\mc{S}$ is a quasi-direct factorization of $G$, then $G \simeq  \prod_{S\in \mc{S}}S$, and all but finitely many members of $\mc{S}$ are compact.
\end{prop}

\begin{proof}
Fix $N\in \mc{S}$ and set $\ol{[G,N]}=\ol{[N,N]}=:K$. The usual projection $\pi:N\rightarrow G/G_{\mc{S}\setminus \{N\}}=:H$ is a normal compression map. Since $H$ is topologically perfect, $\pi(K)$ must be dense in $H$, and thus, $\pi(\ol{[K,K]})$ is also dense in $H$. In view of Proposition~\ref{prop:normal_compression}, we deduce that $K\leq \ol{[K,K]}$ and therefore that $K$ is topologically perfect. 

We now have a normal compression map $\pi:K\rightarrow H$ with $K$ and $H$ topologically perfect and $H$ connected. Theorem~\ref{thmintro:connected} ensures the topologically perfect group $K$ is connected, and applying Corollary~\ref{cor:Zerling:derived}, we infer that $\pi(K)\Z(H)=H$. On the other hand, $\pi(N)\leq H=\pi(K)\Z(H)$, so $\pi(N)=\pi(K)(\Z(H)\cap \pi(N))$.  It is clear from the definition of a quasi-product that $\Z(N) \le \Z(G)$, so $N$ is centerless. The group $\pi(N)$ is then also centerless, and thus, $\Z(H)\cap \pi(N)=\triv$. We deduce that $\pi(N)=\pi(K)$ and therefore that $N$ is topologically perfect.  Applying again Theorem~\ref{thmintro:connected}, $N$ is additionally connected. Every $N \in \mc{S}$ is therefore topologically perfect, connected, and centerless.  

In view of Theorem~\ref{thm:yamabe_radical}, there is a compact normal subgroup $R$ of $G$ such that $G/R$ is a Lie group. The quotient map $\pi: G \rightarrow G/R$ is a closed map, so each member of $\{NR/R \mid N\in \mc{S}\}$ is closed and normal in $G/R$. By Lemma~\ref{lem:Lie:directed_family}, there is a finite subset $J$ of $\mc{S}$ such that $NG_JR/G_JR$ is discrete for all $N \in \mc{S}$. The group $G/G_JR$ is connected, so each $NG_JR/G_JR$ is indeed abelian. The group $G/G_JR$ is thus abelian, and since $G$ is topologically perfect, we deduce that $G_JR=G$.

Each $M\in \mc{S}\setminus J$ is topologically perfect, so $L:=G_{\mc{S}\setminus J}$ is topologically perfect. We now see that 
\[
L=\ol{[L,L]}=\ol{[G,L]}=\ol{[G_JR,L]}=\ol{[R,L]} \le R.
\]
The group $L$ is then compact.  In particular, every element of $\mc{S} \setminus J$ is compact.  Applying Corollary~\ref{cor:quasiproduct:compact}, we conclude that $G_{\mc{S}\setminus J}\simeq \prod_{S\in \mc{S}\setminus J}S$, and it follows further that $G\simeq G_J\times \prod_{S\in \mc{S}\setminus J}S$.

We now turn our attention to $G_J$. Fix $N\in J$ and form $G_J/\CC_{G_J}(N)$.  This quotient is centerless by Lemma~\ref{lem:centerless_quotient}.  The restriction of the natural projection $\pi:N\rightarrow G_J/\CC_{G_J}(N)$ is a normal compression map. Since $N$ and $G_J$ are connected, topologically perfect groups, Corollary~\ref{cor:Zerling:perfect} ensures $\pi$ is indeed an isomorphism. We conclude $G_J=N\CC_{G_J}(N)$, and it follows that $G_{J}\simeq \prod_{N\in J}N$, finishing the proof.
\end{proof}

Since we are primarily interested in non-abelian chief factors, which are topologically characteristically simple, the above results will suffice for our purposes. However, one naturally wonders if a local direct model exists for any locally compact quasi-product. The following example shows that not all locally compact quasi-products admit a local direct model, and it also illustrates the need to exclude abelian quasi-products from Proposition~\ref{prop:quasi_connected}. 

\begin{example}\label{ex:torus_quasiproduct}
Set $G=\Rb\times \Tb^{\Nb}$ and let us consider $G$ as the collection of functions $f:\{-1\}\cup\Nb\rightarrow \Rb\sqcup \Tb$ such that $f(-1)\in \Rb$ and $f(i)\in \Tb$ for $i\in \Nb$. The group $G$ is a connected locally compact group under the product topology. For each $i\in \Nb$ let $f_i\in G$ be defined as follows:
\[
f_i(n):=
\begin{cases}
1& n=-1\\
\sqrt{2} & n=i\\
0 & \text{ else.}
\end{cases}
\]
We now define subgroups of $G$:
\[
G_i:=
\begin{cases}
\Rb & i=-1\\
\grp{f_i} & \text{ else.}
\end{cases}
\]
It is easy to verify that each $G_i$ is closed and the collection $\{G_i\}_{i\in \{-1\}\cup \Nb}$ is a quasi-direct factorization of $G$.

Since $G$ is connected, it has no proper open subgroups, so the only possible local direct model for $(G,\mc{S})$ is $ \prod_{i \in \Nb}G_i$. However, the direct product fails to be a local direct model since it is not locally compact.
\end{example}

\subsection{Locally compact groups of semisimple type}

We here derive a structure result for locally compact groups of strict semisimple type.

\begin{thm}\label{thm:strict_semisimple}
For $G$ an \lcsc group of strict semisimple type, 
\[
G\simeq D\times \prod_{i\in I}C_i\times\prod_{j=0}^nS_j
\]
with $D$ a \tdlcsc group of strict semisimple type, each $C_i$ a topologically simple compact connected Lie group, and each $S_j$ a non-compact topologically simple connected Lie group.
\end{thm}
\begin{proof}
Let $\mc{M}$ list the components of $G$. Appealing to Theorem~\ref{thm:min_normal:simple}, each $N\in \mc{M}$ is topologically simple, and $\mc{M}$ is a quasi-direct factorization of $G$. Let $\mc{C}$ be the connected elements of $\mc{M}$.\

The group $C:=\cgrp{\mc{C}}$ is a connected, topologically perfect locally compact group with quasi-factorization $\mc{C}$, and it is centerless, since $G$ is centerless. Applying Proposition~\ref{prop:quasi_connected}, $C\simeq \prod_{N\in \mc{C}}N$ with only finitely many $N\in \mc{C}$ non-compact. It now follows from Theorem~\ref{thm:yamabe_radical} that $C\simeq  \prod_{i\in I}C_i\times\prod_{j=0}^nS_j$ with each $S_j$ a non-compact topologically simple connected Lie group and each $C_i$ a topologically simple compact connected Lie group.

Setting $D:=\CC_G(C)$, the group $G/D$ is centerless by Lemma~\ref{lem:centerless_quotient}, and $G/D$ is a normal compression of $C$.  Appealing to Corollary~\ref{cor:Zerling:perfect}, the natural projection $\pi:G\rightarrow G/D$ restricts to an isomorphism from $C$ to $G/D$, and thus, $G\simeq C\times D$.

Finally, since $C$ and $D$ are closed direct factors of $G$ with $D \ge G_{\mc{M}\setminus\mc{C}}$, the product $CG_{\mc{M} \setminus \mc{C}}$ is closed in $G$. This product is also dense, so it is indeed the case that $G = CG_{\mc{M} \setminus \mc{C}}$. We thus deduce that $D = G_{\mc{M}\setminus\mc{C}}$.  The group $D$ is then a quasi-product of the set $\mc{M}\setminus \mc{C}$ of non-abelian topologically simple \tdlcsc groups.  Theorem~\ref{thm:min_normal:simple} implies $D$ is of strict semisimple type, and $D$ is totally disconnected by Lemma~\ref{lem:tdlc_quasi-factors}. 
\end{proof}

\begin{cor}\label{cor:strict_semisimple:local_model}
	Let $G$ be an \lcsc group of strict semisimple type with $\mc{S}$ the set of non-trivial topologically simple closed normal subgroups of $G$.  Then $(G,\mc{S})$ is a quasi-product that admits a local direct model, and the local direct model is an \lcsc group of strict semisimple type.
\end{cor}

\begin{proof}
Theorem~\ref{thm:min_normal:simple} ensures that $(G,\mc{S})$ is a quasi-product. Let $\mc{C}$ be the collection of connected non-compact elements of $\mc{S}$ and define $D:=\cgrp{\mc{S}\setminus \mc{C}}$ and $L:=\cgrp{\mc{C}}$. By Theorem~\ref{thm:strict_semisimple}, we can write $G=D\times L$. Additionally, $D$ is almost \tdlc, and $L$ is a finite direct product of topologically simple connected Lie groups. It is now enough to show that $D$ and $L$ both admit local direct models. Since $L$ is a finite product, it is its own local direct model. On the other hand, $D$ is a quasi-product of $\mc{S}\setminus \mc{C}$, so Corollary~\ref{cor:tdlc_LDM} ensures that $D$ also admits a local direct model. 

For the second claim, the local direct model is itself a quasi-product of non-abelian topologically simple groups, so it is of strict semisimple type.
\end{proof}

\subsection{Products of elementary groups}
We conclude this section with a generalization of Lemma~\ref{lem:product_elementary} to generalized central products.

\begin{thm}\label{quasiproduct:elementary}
	Let $G$ be a \tdlcsc group with $\mc{S}$ a countable generalized central factorization of $G$.  If each $N \in \mc{S}$ is elementary, then $G$ is elementary with
	\[
	\xi(G) = {\sup_{N \in \mc{S}}}^+\xi(N).
	\]
\end{thm}

\begin{proof}
Let $\{N_i\}_{i\in \Nb}$ enumerate $\mc{S}$ and put $G_j:=\cgrp{N_i\mid i\leq j}$.  Each $G_j$ is a generalized central product of finitely many elementary groups.  For fixed $j$, there is a local direct model $\phi: P \rightarrow G_j$ where $P = \prod_{i \leq j}N_i$. In view of Lemma~\ref{lem:product_elementary}, $\xi(P)=\max_{ i\leq j}\xi(N_i)$, and applying Proposition~\ref{prop:compression:elementary_rank} to the normal compression $P/\ker(\phi) \rightarrow G_j$, we deduce that 
\[
\xi(G_j)\leq \xi(P/\ker(\phi))\leq \xi(P)=\max_{ i\leq j}\xi(N_i)\leq \xi(G_j),
\]
where the middle inequality follows from Proposition~\ref{prop:xi_monotone}.  Equality thus holds throughout; in particular $\xi(G_j) = \max_{i \le j}\xi(N_i)$. Proposition~\ref{prop:d-rank:ascending} now completes the proof, showing that $G$ is elementary with
\[
	\xi(G) = {\sup_{j \in \Nb}}^+\xi(G_j) = {\sup_{i \in \Nb}}^+\xi(N_i). \qedhere
\]
\end{proof}

\section{Chief factors and blocks}\label{sec:quasidiscrete}
 Any Polish group $G$ comes with a collection of chief blocks $\mf{B}_G$, as defined in \cite{RW_P_15}. Each $\mf{a}\in \mf{B}_G$ is an equivalence class of non-abelian chief factors under the relation of association. The chief blocks appear to be basic building blocks of compactly generated \lcsc groups, as discovered in \cite{RW_EC_15}. It is, however, not straightforward to interpret the group-theoretic significance of chief blocks. Associated chief factors are not necessarily isomorphic as abstract groups, and isomorphic chief factors are not necessarily associated. In this section, we show that the association relation does in fact preserve many group-theoretic properties, so group-theoretic data can be recovered from the chief blocks. 

\subsection{Preliminaries}
Given a topological group $G$, we say the closed normal factors $K_1/L_1$ and $K_2/L_2$ are \defbold{associated}\index{association relation} to one another if the following equations hold:
	\[
	\ol{K_1L_2} = \ol{K_2L_1}; \; K_1 \cap \ol{L_1L_2} = L_1; \; K_2 \cap \ol{L_1L_2} = L_2.
	\]
A \textbf{chief factor}\index{chief factor} of a topological group $G$ is a normal factor $K/L$ with $L < K$ closed such that there is no closed $M\normal G$ with $L<M<K$. By \cite[Corollary~6.9]{RW_P_15}, restricting the association relation to non-abelian chief factors of a Polish group $G$ produces an equivalence relation. 

\begin{defn} A \defbold{chief block}\index{chief block} of a Polish group $G$ is an association class of non-abelian chief factors. The set of chief blocks of $G$ is denoted by $\mf{B}_G$\index{$\mf{B}_G$}.
\end{defn}
\noindent A chief factor is called \textbf{negligible}\index{chief factor!negligible} if it is associated to a compact or discrete factor. A chief block is \textbf{negligible}\index{chief block!negligible} if it contains a negligible chief factor; the collection of non-negligible chief blocks is denoted by $\mf{B}_G^*$. 

An important normal subgroup associated to a chief block is the centralizer: For a chief block $\mf{a}\in \mf{B}_G$ in a topological group $G$, the \defbold{centralizer}\index{chief block!centralizer} of $\mf{a}$ is
\[
\CC_G(\mf{a}):=\{g\in G\mid [g,L]\subseteq K\}
\]
for some (equivalently, any) representative $K/L$ of $\mf{a}$; see \cite[Proposition 6.8]{RW_P_15}. Inclusion of centralizers induces a partial order on $\mf{B}_G$ denoted by $\leq$. Centralizers also give rise to canonical representatives.

\begin{prop}[{\cite[Proposition 7.4]{RW_P_15}}]\label{prop:upper-rep}
Let $G$ be a Polish group and $\mf{a}\in \mf{B}_G$. Then,
	\begin{enumerate}[(1)]
		\item $G/\CC_G(\mf{a})$ is monolithic, and the socle $M/\CC_G(\mf{a})$ of $G/\CC_G(\mf{a})$ is a representative of $\mf{a}$; and
		\item  If $R/S \in \mf{a}$, then $M/\CC_G(\mf{a})$ is a $G$-equivariant normal compression of $R/S$.
	\end{enumerate}
\end{prop}
The representative isolated in Proposition~\ref{prop:upper-rep} is called the \textbf{uppermost representative}\index{chief block!uppermost representative}.

For Polish groups, a normal series may always be refined to include a representative for a given chief block.

\begin{thm}[{\cite[Theorem~1.14]{RW_P_15}}]\label{thm:Schreier_refinement}
	Let $G$ be a Polish group, $\mf{a}\in \mf{B}_G$, and
	\[
	\triv = G_0 \leq G_1 \leq \dots \leq G_n = G
	\]
	be a series of closed normal subgroups in $G$.  Then there is exactly one $i \in \{0,\dots,n-1\}$ such that there exist closed normal subgroups $G_i \le B < A \le G_{i+1}$ of $G$ for which $A/B\in \mf{a}$.
\end{thm}

In the setting of locally compact groups, we have available the essentially chief series, which restricts the number of chief blocks.
\begin{defn}
	An \defbold{essentially chief series}\index{esentially chief series} for a locally compact group $G$ is a finite series
	\[
	\triv = G_0 \leq G_1 \leq \dots \leq G_n = G
	\]
	of closed normal subgroups such that each normal factor $G_{i+1}/G_i$ is either compact, discrete, or a chief factor of $G$.
\end{defn}

\begin{thm}[{\cite[Theorem~4.4]{RW_EC_15}}]\label{thm:chief_series}
	Let $G$ be a compactly generated locally compact group and $(G_i)_{i=0}^{m}$ be a finite ascending sequence of closed normal subgroups of $G$.  Then there exists an essentially chief series
	\[
	\triv = K_0 \le K_1 \le \dots \le K_l = G
	\]
	such that $\{G_0,\dots,G_{m}\}$ is a subset of $\{K_0,\dots,K_l\}$.
\end{thm}

The previous two theorems together imply that the set of \emph{non-negligible} blocks of a compactly generated \lcsc group is finite.

\begin{cor}\label{cor:nonneg}
Let $G$ be an \lcsc group and $(A_i)_{i=0}^m$ be an essentially chief series for $G$. Setting
\[
\mc{N} := \{ A_i/A_{i-1} \mid  A_i/A_{i-1} \text{ is a non-negligible chief factor of $G$, and }1 \le i \le m\},
\]
there is a bijection $\theta: \mf{B}_G^* \rightarrow \mc{N}$ such that $\theta(\mf{a})$ is the unique representative of $\mf{a}$ occurring as a factor of $(A_i)_{i=0}^m$.  In particular, $\mf{B}_G^*$ is finite.
\end{cor}

\subsection{Chief block properties}

\begin{defn}
A property $P$ of groups is a \textbf{(non-negligible) block property}\index{block property} if for every \lcsc group $G$ and $(\mf{a}\in \mf{B}^*_G)$ $\mf{a}\in \mf{B}_G$ there exists $K/L\in \mf{a}$ with $P$ if and only if every $K/L\in \mf{a}$ has $P$.  For a (non-negligible) block property $P$, we say a block ($\mf{a}\in \mf{B}_G^*$) $\mf{a}\in \mf{B}_G$ has property $P$ if some, equivalently all, $K/L\in \mf{a}$ has property $P$.
\end{defn}

\begin{prop}\label{prop:cg_block_property}
	Compact generation is a non-negligible block property.
\end{prop}
\begin{proof}
	Suppose that $G$ is an \lcsc group and $\mf{a}\in \mf{B}^*_G$ has a compactly generated representative $K/L$. Take $A/B$ another representative and form the uppermost representative $M/C\in\mf{a}$ as given by Proposition~\ref{prop:upper-rep}. The factor $M/C$ is a normal compression of $K/L$, so it is compactly generated. Furthermore, $M/C$ is a normal compression of $A/B$ and is non-compact since $\mf{a}$ is non-negligible. As both $M/C$ and $A/B$ are also non-abelian topologically characteristically simple groups, $A/B$ is compactly generated by Corollary~\ref{cor:compactlygenerated}.  All representatives of $\afr$ are thus compactly generated.	
\end{proof}

\begin{prop}\label{prop:block_properties}
	The following are block properties in \lcsc groups: Being isomorphic to a given connected group, being elementary with rank $\alpha$, being amenable, and being quasi-discrete.
\end{prop}
\begin{proof}
Let $G$ be an \lcsc group, $\mf{a} \in \mf{B}_G$, and $M/C$ be the uppermost representative of $\mf{a}$, as given by Proposition~\ref{prop:upper-rep}.

For the first property, suppose that $K/L\in \mf{a}$ is isomorphic to a connected group $H$. The group $M/C$ is connected via Corollary~\ref{cor:con_association}, so Corollary~\ref{cor:Zerling:perfect} ensures that $K/L$ is isomorphic to $M/C$. Consider another representative $A/B$ of $\mf{a}$. There is a normal compression map $A/B\rightarrow M/C$, and again $A/B \simeq M/C$ by Corollaries~\ref{cor:con_association} and~\ref{cor:Zerling:perfect}. All representatives of $\mf{a}$ are therefore isomorphic to the connected group $H$.

Let us now suppose that $\mf{a}\in \mf{B}_G$ has a quasi-discrete representative $K/L$ and take $A/B \in \mf{a}$. The factor $M/C$ is a normal compression of $K/L$, and these chief factors are topologically perfect. Corollary~\ref{cor:compression:quasidiscrete} thus ensures that $M/C$ is quasi-discrete. On the other hand, $M/C$ is also a normal compression of $A/B$. Applying again Corollary~\ref{cor:compression:quasidiscrete}, we deduce that $A/B$ is quasi-discrete, so all representatives are quasi-discrete.

The proof that being elementary with rank $\alpha$ is a block property is similar, using Proposition~\ref{prop:compression:elementary_rank} instead of Corollary~\ref{cor:compression:quasidiscrete}.  Likewise, amenability is a block property using Proposition~\ref{prop:amenability:invariance}.
\end{proof}

Proposition~\ref{prop:block_properties} allows us to define the decomposition rank of a chief block.

\begin{defn} Let $G$ be an \lcsc group and $\mf{a}\in \mf{B}_G$.  For an elementary block $\mf{a}\in \mf{B}_G$, we define the \textbf{decomposition rank} of $\mf{a}$ as $\xi(\mf{a}):=\xi(K/L)$ for some (equivalently, any) $K/L\in \mf{a}$.
\end{defn}

If $\mf{a}$ is negligible and totally disconnected, then $\xi(\mf{a})=2$, since both profinite and discrete groups have decomposition rank $2$. The converse is false in general; as shown in Example~\ref{ex:rank 2}, there exist decomposition rank two blocks which are non-negligible.  Elementary blocks with transfinite rank also exist. For instance, Example~\ref{ex:weak} admits a block of rank $\omega+1$; in later work, we use the techniques outlined in Example~\ref{ex:stacking} to give many more examples.  We recall that by Corollary~\ref{cor:rank_char_simple}, only certain ranks are possible:

\begin{cor}\label{cor:rank_char_simple:blocks}
Let $G$ be an \lcsc group and let $\mf{a}$ be a totally disconnected chief block of $G$.  Then exactly one of the following holds:
	\begin{enumerate}[(1)]
	\item $\xi(\mf{a})=2$;
	\item $\xi(\mf{a})=\alpha+1$ for $\alpha$ some countable transfinite limit ordinal;
	\item $\mf{a}$ is non-elementary.
	\end{enumerate}
\end{cor}

Let us conclude by giving a method to compute the rank of elementary chief factors of semisimple type.
\begin{prop}\label{prop:rank_semisimple}
	Let $G$ be a topologically characteristically simple \tdlcsc group of semisimple type and let $L\leq K$ be closed normal subgroups of $G$ such that $K/L$ is non-abelian.  Then $G$ is elementary if and only if $K/L$ is elementary.  If $G$ is elementary, then $\xi(G) = \xi(K/L)$.
\end{prop}

\begin{proof} 
	The forward implication of the first claim is immediate from Proposition~\ref{prop:xi_monotone}. For the converse implication, consider the quotient $G/L$ and let $\mc{M}$ list the components of $G$. The group $G/L$ contains a dense subgroup generated by the images of components. Since $K/L$ is non-abelian, there is some $C\in \mc{M}$ that intersects $K/L$ non-trivially.
	
	Let $\pi:G\rightarrow G/L$ be the usual projection and find $C\in \mc{M}$ such that $\pi(C)\cap K/L$ is non-trivial. Since $G$ is non-abelian, it has trivial center and therefore is of strict semisimple type. Every component $N\in\mc{M}$ is thus topologically simple via Theorem~\ref{thm:min_normal:simple}. The restriction $\pi:C\rightarrow G/L$ is injective, and $\pi(C)$ is contained in $K/L$. The group $\ol{CL/L}\leq K/L$ is then a normal compression of $C$. As $K/L$ is elementary, $C$ is elementary with $\xi(C)= \xi(K/L)$ via Proposition~\ref{prop:compression:elementary_rank}.
	
	All components of $G$ are isomorphic to $C$, since $G$ is topologically characteristically simple. In view of Theorem~\ref{quasiproduct:elementary}, we now deduce that $G$ is elementary with
	\[
	\xi(G) = {\sup_{N \in \mc{M}}}^+\xi(N)=\xi(C)= \xi(K/L).\qedhere
	\]		
\end{proof}

\subsection{Negligible chief blocks}
By Corollary~\ref{cor:nonneg}, an essentially chief series for a compactly generated \lcsc group $G$ includes representatives for every non-negligible chief block of $G$. Of course, an essentially chief series can also include negligible chief factors. The utility of essentially chief series as a decomposition into well-behaved ``basic'' groups, compact groups, and discrete groups thus depends on how much structure beyond the compact or discrete cases can be hidden in negligible chief factors.  We here analyze negligible chief factors and blocks. This analysis will demonstrate that negligible chief factors are either compact or very close to discrete.

Let us first identify three types of chief blocks of an \lcsc group $G$. Proposition~\ref{prop:block_properties} ensures the following definition is sensible:
\begin{defn} 
	Let $G$ be an \lcsc group and $\mf{a}\in \mf{B}_G$. We say that the block $\mf{a}$ is \defbold{connected compact}\index{chief block!connected compact} if all representatives are isomorphic to a single connected compact group. We say that $\mf{a}$ is \textbf{quasi-discrete}\index{chief block!quasi-discrete} if all representatives are quasi-discrete. The block $\mf{a}$ is \defbold{robust}\index{chief block!robust} if none of the representatives of $\mf{a}$ are compact or quasi-discrete. Denote by $\mf{B}_G^r$ the collection of robust blocks.
\end{defn}

We now argue that these three types partition the collection of chief blocks. To this end, recall the special structure of topologically characteristically simple profinite groups.

\begin{lem}[See for instance {\cite[Lemma~8.2.3]{RZ00}}]\label{lem:profinite_charsimple}
	A profinite group $P$ is topologically characteristically simple if and only if
	\[
	P \simeq F^I
	\]
	for some finite simple group $F$ and set $I$.  In particular, every topologically characteristically simple profinite group is quasi-discrete, since it contains the dense quasi-central subgroup $F^{<I}$.
\end{lem}

\begin{thm}\label{thm:negligible_types}
Suppose that $G$ is an \lcsc group. If $\mf{a}$ is a chief block of $G$, then $\mf{a}$ is exactly one of the following: connected compact, quasi-discrete, or robust.	
\end{thm}
\begin{proof}
It is obvious the three cases are exclusive, so we need only to show $\mf{a}$ is one of the three types.  By Proposition~\ref{prop:block_properties}, if $\mf{a}$ is not quasi-discrete, then none of its representatives are quasi-discrete. Suppose that $\mf{a}$ is neither robust nor quasi-discrete; in other words, there is a representative $A/B \in \mf{a}$ that is compact, but not quasi-discrete. The factor $A/B$ is either connected or totally disconnected. In the latter case, $A/B$ is profinite, so it is quasi-discrete by Lemma~\ref{lem:profinite_charsimple}. We deduce that $A/B$ is connected and compact, and hence by Proposition~\ref{prop:block_properties}, every representative of $\mf{a}$ is isomorphic to $A/B$ as a topological group.  Thus, $\mf{a}$ is connected compact.
\end{proof}

A structure theorem for compactly generated groups without robust blocks is now in hand.
\begin{cor}\label{cor:negligible_decomp} 
Suppose that $G$ is a compactly generated \lcsc group. If $\mf{B}_G^r=\emptyset$, then every essentially chief series $(G_{i})_{i=0}^n$ for $G$ is such that each factor $G_i/G_{i-1}$ is either compact or quasi-discrete.
\end{cor}
\begin{proof}
Let $\triv=G_0<\dots <G_n=G$ be an essentially chief series for $G$, so each factor $G_i/G_{i-1}$ is either compact, discrete, or chief.  Discrete factors and abelian chief factors are obviously quasi-discrete.  Theorem~\ref{thm:negligible_types} implies that every non-abelian chief factor in the series is either compact-connected or quasi-discrete.
\end{proof}

\subsection{Application: Hopfian property}\label{sec:hopfian}
A group is called \textbf{Hopfian}\index{Hopfian} if every continuous self-surjection is also an injection. Using the machinery of chief blocks, we show many compactly generated \lcsc groups are Hopfian. Our analysis of the Hopfian property is via the notion of block morphisms developed in \cite[Section 8]{RW_P_15}; we direct the reader to loc.\ cit.\ for further details on this construction. 

For $G$ a Polish group, let $X_G$ denote the collection of non-abelian chief factors. The partial order on centralizers induces a preorder on $X_G$, which we denote by $(X_G,\leq)$. We remark that the natural quotient of $(X_G,\leq)$ giving a partial ordering recovers the partially ordered set of chief blocks $\mf{B}_G$. 

\begin{defn}\label{defn:block_morphism}
	For $G$ and $H$ Polish groups, a \textbf{strong block morphism} $\psi:(X_H,\leq)\rightarrow (X_G,\leq)$ is an homomorphism of preorders such that $\psi(A/B)\simeq A/B$ as topological groups for each $A/B\in X_H$. If the induced map $\wt{\psi}:\mf{B}_H\rightarrow \mf{B}_G$ is injective, we say that $\psi$ is a \textbf{strong block monomorphism}.
\end{defn}

By \cite[Proposition~8.4]{RW_P_15}, a continuous surjective homomorphism $\phi:G\rightarrow H$ of Polish groups induces a strong block monomorphism $\psi:(X_H,\leq)\rightarrow (X_G,\leq)$ defined by $A/B\mapsto\phi^{-1}(A)/\phi^{-1}(B)$. In the setting of locally compact groups, this morphism also respects robust blocks.

\begin{lem}\label{lem:robust_injection}
Suppose that $G$ and $H$ are \lcsc groups and $\phi:G\rightarrow H$ is a continuous surjective homomorphism. Then the strong block monomorphism 
\[
\psi:(X_H,\leq)\rightarrow (X_G,\leq)
\]
induced by $\phi$ is such that $\tilde{\psi}(\mf{B}_H^r) =  \mf{B}_G^r \cap \tilde{\psi}(\mf{B}_H)$, where $\tilde{\psi}$ is the map induced from $\psi$.
\end{lem}

\begin{proof}
Take $\mf{a}\in \mf{B}_H$ and suppose that $\tilde{\psi}(\mf{a})\notin\mf{B}_G^r$. Taking $A/B\in \mf{a}$, that $\psi$ is a strong block monomorphism implies that $A/B\simeq \psi(A/B)$ and that $\psi(A/B)\in \tilde{\psi}(\mf{a})$.  The block $\tilde{\psi}(\mf{a})$ is not robust, so Theorem~\ref{thm:negligible_types} implies that $\tilde{\psi}(\mf{a})$ is either connected-compact or quasi-discrete. The factor $\psi(A/B)$ is thus either a connected compact group or a quasi-discrete group, hence $A/B$ is  either a connected compact group or a quasi-discrete group. We conclude that $\mf{a}$ is either connected-compact or quasi-discrete. That is to say, $\mf{a}$ is not robust.  This demonstrates that $\tilde{\psi}(\mf{B}_H^r) \subseteq  \mf{B}_G^r \cap \tilde{\psi}(\mf{B}_H)$.

Conversely, suppose that $\mf{a} \not\in \mf{B}_H^r$.  We may find $A/B \in \mf{a}$ which is either connected and compact or quasi-discrete. The image $\psi(A/B)$ is then either connected and compact or quasi-discrete, since $\psi$ is a strong block monomorphism. The block $\tilde{\psi}(\mf{a})$ is thus not robust.  This demonstrates that $\tilde{\psi}(\mf{B}_H \setminus \mf{B}_H^r) \subseteq  \mf{B}_G \setminus \mf{B}_G^r$, and hence $\tilde{\psi}(\mf{B}_H^r) = \mf{B}_G^r \cap \tilde{\psi}(\mf{B}_H)$.
\end{proof}

We stress that it is important to consider \emph{robust blocks} for the previous lemma. Negligible blocks can become non-negligible in a quotient, so in general we do not have $\tilde{\psi}(\mf{B}_H^*) \subseteq \mf{B}_G^*$, where $\mf{B}_G^*$ is the set of non-negligible blocks.

\begin{thm}\label{thm:Hopfian} 
Suppose that $G$ is a compactly generated \lcsc group. If $G$ has no non-trivial compact or quasi-discrete closed normal subgroups, then $G$ is Hopfian.
\end{thm}
\begin{proof}
We assume that $G$ is non-trivial. Suppose for contradiction that $\phi:G\rightarrow G$ is a continuous surjective homomorphism with a non-trivial kernel. We apply Theorem~\ref{thm:chief_series} to find a chief series $(G_{i})_{i=0}^n$ which refines the normal series $\triv <\ker(\phi)<G$. The first factor $G_1/G_0=G_1$ must be neither compact nor quasi-discrete since $G$ has no such non-trivial closed normal subgroups. Hence, $G_1/G_0$ is a non-abelian chief factor, and the chief block $[G_1/G_0]$ is robust via Theorem~\ref{thm:negligible_types}. 
	
On the other hand, Lemma~\ref{lem:robust_injection} supplies a canonical injective map $\tilde{\psi}:\mf{B}_{G}^r\rightarrow \mf{B}_G^r$ arising from the surjection $\phi:G\rightarrow G$. The set $\mf{B}_G^r$ is finite by Corollary~\ref{cor:nonneg}, so $\tilde{\psi}$ is indeed a bijection. We may thus find a chief factor $A/B$ of $G$ such that 
\[
K/L=\psi(A/B)=\phi^{-1}(A)/\phi^{-1}(B)
\]
is an element of $[G_1/G_0]$. In the series $G_0<G_1<\ker(\phi)\leq L<K\leq G$, the non-abelian chief factor $G_1/G_0$ is associated to the factor $K/L$. This is absurd since associated chief factors have the same centralizer and $L\leq \CC_G(K/L)$.
\end{proof}

It is easy to build non-discrete compactly generated locally compact groups which are non-Hopfian simply by taking the direct product of an infinite profinite group with a finitely generated non-Hopfian group.

\section{Extensions of chief blocks}\label{sec:limit_blocks}
Corollary~\ref{cor:negligible_decomp} shows a \emph{compactly generated} \lcsc group $G$ without robust chief factors has limited complexity as a topological group. In this section, we show that except for groups with limited topological structure, \textit{every} \lcsc group admits robust chief blocks that are minimally covered. These results are then applied to investigate the structure of chief factors of compactly generated \lcsc groups.

\subsection{Preliminaries}

\begin{defn}
	Let $\mf{a}$ be a chief block and let $N/M$ be a closed normal factor of a topological group $G$.  We say $N/M$ \defbold{covers}\index{covers} $\mf{a}$ if $\CC_G(\mf{a})$ contains $M$ but not $N$.
	If $N/M$ does not cover $\mf{a}$, we say that $N/M$ \defbold{avoids}\index{avoids} $\mf{a}$. A block can be avoided in one of two ways: The factor $N/M$ is \defbold{below} $\mf{a}$ if $N \le \CC_G(\mf{a})$; equivalently, $G/N$ covers $\mf{a}$. The factor $N/M$ is \defbold{above} $\mf{a}$ if $M \not\le \CC_G(\mf{a})$; equivalently, $M/\triv$ covers $\mf{a}$. For a closed normal subgroup $N\normal G$, we will say that $N$ covers or avoids a chief block by seeing $N$ as the closed normal factor $N/\triv$.
\end{defn}

It follows easily from Theorem~\ref{thm:Schreier_refinement} that if a closed normal factor $N/M$ of a Polish group $G$ covers $\mf{a} \in \mf{B}_G$, then there is a representative $K/L\in \mf{a}$ such that $M\leq L<K\leq N$.

A block $\afr\in \mf{B}_G$ is \textbf{minimally covered}\index{chief block!minimally covered} if there is a least closed normal subgroup covering the block. This least normal subgroup is denoted by $G_{\afr}$. The \textbf{lowermost representative} is $G_{\afr}/\CC_{G_{\afr}}(\afr)$; this normal factor is a representative of $\afr$ via \cite[Proposition 7.13]{RW_P_15}. The collection of minimally covered chief blocks is denoted by $\mf{B}_G^{\min}$.

\begin{defn}
Let $G$ be a Polish group with $H$ a closed subgroup.  Given $\mf{a} \in \mf{B}_H$ and $\mf{b} \in \mf{B}_G$, we say that $\mf{b}$ is an \defbold{extension}\index{chief block!extension} of $\mf{a}$ in $G$ if for every closed normal subgroup $K$ of $G$, the group $K$ covers $\mf{b}$ if and only if $K\cap H$ covers $\mf{a}$.
\end{defn}

In general, it is not clear whether extensions exist; however, when an extension exists, it is unique, {\cite[Lemma 9.2]{RW_P_15}}.  If the extension exists, we write $\mf{a}^G$ for the extension of $\mf{a}$ to $G$.  The following observations are made in \cite[Section 9.1]{RW_P_15}; to gain familiarity with the definitions, the reader is encouraged to prove these observations.

\begin{obs}\label{obs:extensions_transitive}
Let $A \le B \le G$ be closed subgroups of a topological group $G$ and $\mf{a} \in \mf{B}_A$ be extendable to $B$.  Then $\mf{a}$ is extendable to $G$ if and only if $\mf{a}^B$ is extendable to $G$. If $\mf{a}$ is extendable to $G$, then $\mf{a}^G = (\mf{a}^B)^G$.
\end{obs}

\begin{obs}\label{obs:centralizers}
Let $G$ be a topological group with $H$ a closed subgroup. If $K/L$ is a closed normal factor of $G$ such that $K\cap H/L\cap H$ covers $\mf{a}\in \mf{B}_H$, then $\CC_H(K/L)\leq \CC_H(\mf{a})$.
\end{obs}

\begin{obs}\label{obs:covering}
Let $H$ be a closed subgroup of a Polish group $G$ and $\mf{a} \in \mf{B}_H$ be extendable to $G$.  Then a closed normal factor $K/L$ of $G$ covers $\mf{a}^G$ if and only if $K\cap H/L\cap H$ covers $\mf{a}$.
\end{obs}

In the case of minimally covered blocks, extensions are easier to identify:

\begin{lem}[{\cite[Lemma~9.7]{RW_P_15}}]\label{lem:extension:min_covered}
	Let $G$ be a topological group with $H$ a closed subgroup of $G$ and $\mf{a} \in \mf{B}^{\min}_H$.  Then $\mf{a}$ is extendable to $G$ if and only if there exists $\mf{b} \in \mf{B}^{\min}_G$ such that $(G_{\mf{b}} \cap H)/\CC_{(G_{\mf{b}} \cap H)}(\mf{b})$ covers $\mf{a}$.  If $\mf{a}$ is extendable to $G$, then $\mf{a}^G = \mf{b}$.
\end{lem}
We can say more when the subgroup $H$ is also normal.
\begin{prop}[{\cite[Proposition~9.8]{RW_P_15}}]\label{prop:induced_block}
	Let $G$ be a topological group with $H$ a closed normal subgroup of $G$ and $\mf{a} \in \mf{B}^{\min}_H$.  Then $\mf{a}$ is extendable to $G$, and the extension $\mf{a}^G$ has least representative $M/N$ where
	\[
	M := \ngrp{H_{\mf{a}}}_G\quad \text{ and } \quad N := \bigcap_{g \in G}g\CC_G(\mf{a})g\inv \cap M.
	\]
\end{prop}

The previous results are particularly useful in the setting of compactly generated groups, where minimally covered blocks are known to exist.
\begin{prop}[{\cite[Proposition 4.10, Corollary 4.9]{RW_EC_15}}]\label{prop:min_covered}
If $G$ is a compactly generated \lcsc group, then all non-negligible chief blocks are minimally covered, and there are only finitely many non-negligible chief blocks.
\end{prop}

\subsection{Extending robust blocks}
We now show that the robust blocks of any compactly generated open subgroup of a \lcsc group $G$ extend to $G$. Let us first note that robustness is preserved by extension, provided the extension exists.

\begin{lem}\label{lem:quasi-center_chief}
	Suppose that $G$ is an \lcsc group with $K/L$ a quasi-discrete closed normal factor of $G$. If $L\leq B<A\leq K$ are such that $A/B$ is a chief factor of $G$, then $A/B$ is quasi-discrete.
\end{lem}

\begin{proof}
	The factor $K/B$ is quasi-discrete, and $A/B$ is a normal subgroup of $K/B$. If $A/B$ intersects $\QZ(K/B)$ trivially, then $A/B$ is central in $K/B$ and a fortiori is quasi-discrete. If the group $A/B$ has non-trivial intersection with $\QZ(K/B)$, then $A/B$ has a non-trivial quasi-center. Since the quasi-center is characteristic and $A/B$ is a chief factor of $G$, we conclude that $A/B$ is quasi-discrete.
\end{proof}

\begin{prop}\label{prop:robust_extension}
Suppose that $G$ is an \lcsc group with $O$ an open subgroup and let $\mf{a}$ be a robust chief block of $O$.
\begin{enumerate}[(1)]
\item If $K/L$ is a closed normal factor of $G$ such that $K\cap O/L\cap O$ covers $\afr$, then $K/L$ is neither compact nor quasi-discrete.
\item If $\mf{a}$ extends to $G$, then $\mf{a}^G$ is robust.
\end{enumerate}
\end{prop}

\begin{proof}
For (1), the subgroup $(K\cap O)L/L$ is an open subgroup of $K/L$, and $(K \cap O)L/L \simeq (K \cap O)/(L\cap O)$.  If $K/L$ is compact, then $(K \cap O)/(L \cap O)$ is compact, and since a representative of $\mf{a}$ occurs as a closed normal factor of $(K \cap O)/(L \cap O)$, it follows that $\mf{a}$ is negligible, which is absurd.  If instead $K/L$ is quasi-discrete, then $(K \cap O)L/L$ is also quasi-discrete since it is an open subgroup of $K/L$, so $(K \cap O)/(L \cap O)$ is quasi-discrete.  In this case, Lemma~\ref{lem:quasi-center_chief} ensures that $\mf{a}$ has a quasi-discrete representative, which contradicts the hypothesis that $\mf{a}$ is robust. 

Claim (2) now follows immediately from (1) and Observation~\ref{obs:covering}.
\end{proof}

We now isolate a set of blocks that arise from robust blocks of compactly generated open subgroups.

\begin{defn}
Let $G$ be an \lcsc group and let $\mf{a} \in \mf{B}_G$. We say that $\mf{a}$ is a \defbold{regionally robust block}\index{chief block!regionally robust} if there exists a compactly generated open subgroup $O$ of $G$ and a robust block $\mf{b}$ of $O$ such that $\mf{a}$ is the extension of $\mf{b}$ in $G$.  The collection of regionally robust blocks is denoted by $\mf{B}_G^{rr}$.
\end{defn}

Proposition~\ref{prop:min_covered} and Lemma~\ref{lem:extension:min_covered} ensure that every regionally robust block is minimally covered, and Proposition~\ref{prop:robust_extension} ensures that every regionally robust block is robust.  If $G$ is compactly generated, we now see that $\mf{B}_G^{rr} \subseteq \mf{B}_G^r$, so indeed $\mf{B}_G^{rr} = \mf{B}_G^r$.

\begin{lem}\label{lem:limit_blocks_number}
	An \lcsc group has at most countably many regionally robust blocks. 
\end{lem}

\begin{proof}
Let $G$ be an \lcsc group. A regionally robust block $\mf{b}$ of $G$ is determined by a pair $(O,\mf{a})$ where $O$ is a compactly generated open subgroup of $G$ and $\mf{a} \in \mf{B}^r_O$ is such that $\mf{a}^G = \mf{b}$. An easy counting argument shows there are only countably many compactly generated open subgroups, and via Proposition~\ref{prop:min_covered}, a compactly generated open subgroup has finitely many robust blocks. We conclude there are at most countably many such pairs $(O,\afr)$, and hence an \lcsc group has at most countably many regionally robust blocks.
\end{proof}

We now set about showing that regionally robust blocks of open subgroups always extend. We begin by considering compactly generated open subgroups of compactly generated \lcsc groups; the key here is to exploit the essentially chief series. In the following proof, it will be important that certain subgroups and homomorphisms are invariant under some group action. The action is always the obvious action by conjugation.

\begin{lem}\label{lem:unique_extension} 
Let $G$ be a compactly generated \lcsc group. If $O$ is a compactly generated open subgroup of $G$, then all robust blocks of $O$ extend to $G$.
\end{lem}

\begin{proof}
Fix $\mf{a}\in \mf{B}^r_O$ and let $(H_i)_{i=0}^n$ be an essentially chief series for $G$. The restriction $(H_i\cap O)_{i=0}^n$ is a normal series for $O$. Applying Theorem~\ref{thm:Schreier_refinement}, there is some $0\leq i<n$ such that $H_{i+1}\cap O/H_i\cap O$ covers $\mf{a}$.  The factor $H_{i+1}/H_i$ is neither compact nor quasi-discrete by Proposition~\ref{prop:robust_extension}, and thus, $H_{i+1}/H_i$ is a non-abelian chief factor. Theorem~\ref{thm:negligible_types} implies that $\mf{b}:=[H_{i+1}/H_i]$ is indeed robust. In view of Proposition~\ref{prop:min_covered}, we deduce that $\mf{b} \in \mf{B}^{\min}_G$.
	
We argue that $\mf{b}$ is the extension of $\mf{a}$ to $G$. Let us first consider the uppermost representative $N/C$ of $\mf{b}$.  Since $N \cap O \ge H_{i+1} \cap O$, the subgroup $N\cap O$ covers $\mf{a}$. On the other hand, Observation~\ref{obs:centralizers} implies $\CC_O(H_{i+1}/H_i) \le \CC_O(\mf{a})$, and the construction of the uppermost representative ensures $C = \CC_G(H_{i+1}/H_i)$. We conclude that $C \cap O \le \CC_O(\mf{a})$ and thus that $C\cap O$ is below $\mf{a}$. Hence, $N\cap O/C\cap O$ covers $\mf{a}$.
	
We now consider the lowermost representative $I/J$ of $\mf{b}$. Proposition~\ref{prop:upper-rep} gives a $G$-equivariant normal compression $\psi:I/J\rightarrow N/C$. We also have an $O$-equivariant isomorphism
\[
\iota:N\cap O/C\cap O\rightarrow (N\cap O)C/C:=W.
\]
Since $N\cap O/C\cap O$ covers $\mf{a}$, there are $A$ and $B$ closed $O$-invariant subgroups of $N\cap O/C\cap O$ such that $A/B$ is a representative of $\mf{a}$. The subgroup $W$ is open in $N/C$, so $\psi(I/J)\cap W$ is dense and normal in $W$ and hence accounts for a dense normal subgroup of $W/\ol{\iota(B)}$.  Since $A/B$ is topologically perfect, setting $D:= (\psi(I/J)\cap \iota(A))\iota(B)$, we have
\[
\ol{D} \ge \ol{[\psi(I/J)\cap W,\iota(A)]\iota(B)} \ge \ol{[W,\iota(A)]\iota(B)}=\ol{\iota(A)}.
\]
In other words, $D$ is a dense subgroup of $\iota(A)$.

The map $\iota$ is an $O$-equivariant isomorphism of topological groups, so $\iota\inv(D)$ is a dense $O$-invariant subgroup of $A$.  Since $A/B$ is not quasi-discrete, it follows that $|\iota\inv(D):B|$ is uncountable. On the other hand, the group $V:=(I\cap O)J/J$ is open in $I/J$, so it is of countable index. It follows that $(\psi(V)\cap \iota(A))\iota(B)$ is of countable index in $D$, and thus, $|(\psi(V) \cap \iota(A))\iota(B):\iota(B)|$ is also uncountable; in particular, $\psi(V) \cap \iota(A) \nleq \iota(B)$. The group $I\cap O$ therefore cannot centralize $\afr$. That is to say, $I\cap O$ covers $\mf{a}$. Plainly, $J \cap O$ is below $\mf{a}$, since $J \le C$, so $I\cap O/J\cap O$ covers $\mf{a}$. Lemma \ref{lem:extension:min_covered} now ensures that $\mf{b}$ is the extension of $\mf{a}$ to $G$.
\end{proof}

We now upgrade from compactly generated supergroups to arbitrary \lcsc supergroups.  The strategy is to apply Lemma~\ref{lem:unique_extension} to an increasing exhaustion by compactly generated open subgroups. The existence of the extensions $\mf{a}_i$ in the following statement is ensured by Lemma~\ref{lem:unique_extension}.

\begin{prop}\label{prop:block_sequence}
Let $G$ be an \lcsc group with $(O_i)_{i\in \Nb}$ an increasing exhaustion of $G$ by compactly generated open subgroups.  Fix $\mf{a}_0 \in \mf{B}^r_{O_0}$ and set $\mf{a}_i := \mf{a}^{O_i}$ for each $i\in \Nb$.  Then there is a robust chief block $\bfr$ of $G$ that extends each of the $\mf{a}_i$ to $G$. Additionally, 
	\[
	 \CC_G(\mf{b})=\bigcup_{n\geq 0}\bigcap_{i\geq n}\CC_{O_n}(\mf{a}_i).
    \]
\end{prop}

\begin{proof}
Set
\[
	D:=\bigcup_{n\geq 0}\bigcap_{i\geq n}\CC_{O_n}(\mf{a}_i)
\]
For fixed $n$ and all $i\geq n$, Observation~\ref{obs:centralizers} ensures that $\CC_{O_n}(\mf{a}_{i+1})\leq \CC_{O_n}(\mf{a}_i)$, hence
\[
	D \cap O_n = \bigcap_{i\geq n}\CC_{O_n}(\mf{a}_i) \le \CC_{O_n}(\mf{a}_n).
\]
The group $D$ is thus closed and such that $D\cap O_0$ avoids $\mf{a}_0$.  

For $N \in \mc{N}(G)$ such that $N \nleq D$, there exists $n$ such that $N \cap O_n \nleq D \cap O_n$. Therefore, $N \cap O_n \nleq \CC_{O_n}(\mf{a}_i)$ for some $i \ge n$, so $N\cap O_i$ covers $\mf{a}_i$. From the definition of an extension, we infer that $N\cap O_0$ covers $\afr_0$. The subgroup $D$ is thus the unique maximal closed normal subgroup of $G$ such that $D\cap O_0$ avoids $\afr_0$. 

Set $M:=\bigcap\{N\in \mc{N}(G)\mid  N\cap O_0\text{ covers }\afr_0\text{, and }D\leq N\}$. Since the block $\afr_0$ is minimally covered, it follows that $M\cap O_0$ covers $\afr_0$. On the other hand, $D$ is maximal among closed normal subgroups of $G$ that avoid $\mf{a}_0$, hence $M/D$ is a chief factor of $G$.	Set $\mf{b}:=[M/D]$ and let $C = \CC_G(\mf{b})$.  Since $M\cap O_0/D\cap O_0$ covers $\mf{a}_0$, we have $C \cap O_0 \le \CC_{O_0}(\mf{a}_0)$ by Observation~\ref{obs:centralizers}. The subgroup $C$ therefore avoids $\afr_0$, so $C \le D$ by the previous paragraph. As $D$ clearly centralizes $M/D$, we deduce that $C = D$.
	
Finally, $N\in \mc{N}(G)$ is such that $N\cap O_0$ covers $\mf{a}_0$ if and only if $N \nleq D$.  Since $D = \CC_G(\mf{b})$, it is also the case that $N\in \mc{N}(G)$ covers $\mf{b}$ if and only if $N \nleq D$.  The block $\mf{b}$ is thus the extension of $\mf{a}_0$, and it is robust via Proposition~\ref{prop:robust_extension}. That $\mf{b}$ extends each $\afr_i$ follows from Observation~\ref{obs:extensions_transitive}.
\end{proof}

We now prove our general extension theorem.
\begin{thm}\label{thm:limit_extension}
	Suppose that $G$ is an \lcsc group and $H$ is a closed subgroup of $G$.  If there is a sequence
	\[
	H = H_0 \le H_1 \le \dots \le H_n = G
	\]
	of closed subgroups of $G$ such that $H_i$ either an open subgroup of $H_{i+1}$ or a normal subgroup of $H_{i+1}$ for each $0 \le i < n$, then every regionally robust block of $H$ extends to a regionally robust block of $G$.
\end{thm}
\begin{proof}
By Observation~\ref{obs:extensions_transitive} and induction on $n$, it is enough to consider the cases where $H$ is either open in $G$ or normal in $G$.

Suppose first that $H$ is open in $G$ and let $\mf{a}$ be a regionally robust block of $H$.  The block $\mf{a}$ is the extension of some robust chief block $\mf{b}_0$ of $O_0$, where $O_0$ is a compactly generated open subgroup of $H$.  Starting with $O_0$, we can find an increasing exhaustion $(O_i)_{i\in \Nb}$ of $G$ by compactly generated open subgroups. Proposition~\ref{prop:block_sequence} then provides the extension $\mf{b}$ of $\mf{b}_0$ to $G$, and appealing to Observation~\ref{obs:extensions_transitive}, $\mf{b}$ is the extension of $\mf{a}$ to $G$.  We conclude that $\mf{a}$ is extendable to $G$, and the extension is clearly also a regionally robust block of $G$. 

Suppose that $H$ is normal in $G$ and let $\mf{a}$ be a regionally robust block of $H$. Say that $\mf{a}$ is the extension of $\mf{b} \in \mf{B}^r_O$, where $O$ is a compactly generated open subgroup of $H$.  There are only countably many compactly generated open subgroups of $H$, so $N_G(O)$ is open in $G$. We can then take a compactly generated open subgroup $W$ of $N_G(O)$ such that $O\normal W$.  By Propositions~\ref{prop:induced_block} and \ref{prop:robust_extension}, $\mf{b}$ extends to a robust chief block $\mf{b}^W$ of $W$. The previous case now ensures that $\mf{b}^W$ extends to a regionally robust block $\mf{c} := (\mf{b}^W)^G$ of $G$, and in view of Observation~\ref{obs:extensions_transitive}, we deduce that
\[
\mf{c} = (\mf{b}^W)^G = \mf{b}^G = (\mf{b}^H)^G = \mf{a}^G.
\]
The block $\mf{a}$ is thus extendable to $G$. Since $\mf{b}^W$ is a robust block of a compactly generated open subgroup of $G$, we see that $\mf{a}^G$ is also a regionally robust block.
\end{proof}

\subsection{Further observations on regionally robust blocks}
The persistence of regionally robust blocks from open subgroups shows that regionally robust blocks are persistent in a more subtle manner: if a chief factor $K/L$ \emph{has} a regionally robust block, then $[K/L]$ \emph{is} a regionally robust block.

\begin{prop}\label{prop:limit_blocks_of_blocks}
Let $G$ be an \lcsc group with $K/L$ a closed normal factor of $G$.  If $\mf{B}^{rr}_{K/L} \neq \emptyset$, then $K/L$ covers a regionally robust block of $K$ ,and hence, it covers a regionally robust block of $G$.  If in addition $K/L$ is a chief factor of $G$, then $[K/L]$ is a regionally robust block of $G$.
\end{prop}

\begin{proof}
Fix a regionally robust block $\mf{a}$ of $K/L$ and say that $\mf{a}$ is the extension of a robust block $\mf{b}$ of $O$ where $O$ is an open compactly generated subgroup of $K/L$. We may fix $W\leq K$ compactly generated and open such that $WL/L=O$. Given a continuous surjection of \lcsc groups, there is a canonical way to lift blocks from the range to the domain; see the discussion in Section~\ref{sec:hopfian}. Lemma~\ref{lem:robust_injection} also ensures that robust blocks lift to robust blocks.  Considering the surjection $W\rightarrow WL/L=O$, the lift $\tilde{\mf{b}}$ of $\mf{b}$ to $W$ is robust. Theorem~\ref{thm:limit_extension} then implies that $\tilde{\mf{b}}$ is extendable to $K$, and the extension is plainly regionally robust. In view of the definition of the lift, $L\cap W$ avoids $\wt{\bfr}$, so the normal factor $K/L$ covers $\wt{\bfr}^K=:\cfr$.  By Theorem~\ref{thm:limit_extension}, $\cfr$ is also extendable to a regionally robust block $\mf{d}$ of $G$; since $K/L$ is a normal factor of $G$, it also covers $\mf{d}$, via Proposition~\ref{prop:induced_block}. 

The factor $K/L$ covers the regionally robust block $\mf{d}$ of $G$, so we may insert closed normal subgroups $L\leq B<A\leq K$ of $G$ such that $A/B\in \mf{d}$. If $K/L$ is already a chief factor, we must have $A = K$ and $B = L$.  We deduce that $K/L\in \mf{d}$, hence $[K/L]$ is a regionally robust block of $G$.
\end{proof}

We now prove an alternative for \lcsc groups.

%

\begin{prop}\label{prop:min_covered_exist} For $G$ an \lcsc group, either
	\begin{enumerate}
		\item $\mf{B}_{G}^{rr}\neq \emptyset$, or
		\item $G^{\circ}$ is compact-by-solvable-by-compact, and $G/G^{\circ}$ is elementary with rank at most $\omega+1$. If $G$ is compactly generated, then $G/G^{\circ}$ is elementary with finite rank.
	\end{enumerate}
\end{prop}

\begin{proof}
Suppose that $\mf{B}_G^{rr}=\emptyset$; we argue that (2) holds. 

Let us first suppose that $G$ is compactly generated. Set $R:=\RadLE(G^{\circ})$. The group $R$ is a compact normal subgroup of $G$, and $G^{\circ}/R$ is a Lie group by Theorem~\ref{thm:yamabe_radical}. Letting $\pi:G^{\circ}\rightarrow G^{\circ}/R$ be the natural projection, put $S:=\pi^{-1}(\Sol(G^{\circ}/R))$ where $\Sol(G^{\circ}/R)$ is the solvable radical and take $Z$ to be such that $Z/S:=\Z(G^{\circ}/S)$. This produces a series of closed normal subgroups of $G$:
\[
\triv\leq R\leq S\leq Z\leq G^{\circ}\leq G.
\]
Applying Theorem~\ref{thm:chief_series}, we refine this series into an essentially chief series $\triv=G_0<\dots <G_n=G$.  Any factor given by terms $G_k<G_{k+1}$ with $Z\leq G_k<G_{k+1}\leq G^{\circ}$ is a product of non-abelian connected simple groups, since $G^{\circ}/Z$ is a centerless connected semisimple group. Since $\mf{B}_G^r=\mf{B}_G^{rr}=\emptyset$, these simple groups must be compact.  The connected component $G^{\circ}$ is therefore compact-by-solvable-by-compact.

For the totally disconnected part, Corollary~\ref{cor:negligible_decomp} ensures that any chief factors $G_{k+1}/G_k$ with $G^\circ \le G_k$ are quasi-discrete, hence elementary of rank $2$. The compact totally disconnected factors and discrete factors are also elementary of rank $2$. The rank is subadditive by Lemma~\ref{lem:d-rank_extensions}, hence
\[
\xi(G/G^{\circ})\leq 2n < \omega.
\] 
This completes the proof in the compactly generated case.

Suppose now that $G$ is an arbitrary \lcsc group. Taking $O\leq G$ open and compactly generated, we see that $\mf{B}_O^{rr}=\mf{B}_O^{r}=\emptyset$, since otherwise $G$ admits a regionally robust block via Theorem~\ref{thm:limit_extension}. Appealing to the compactly generated case of the proposition, we deduce that $O^{\circ}=G^{\circ}$ is compact-by-solvable-by-compact and that $O/O^{\circ}$ is elementary with finite rank. The quotient $G/G^{\circ}$ is thus an increasing union of finite rank elementary groups, hence $\xi(G/G^{\circ})\leq \omega+1$, via Proposition~\ref{prop:d-rank:ascending}. 
\end{proof}

As a consequence, sufficiently large chief factors are necessarily regionally robust.

\begin{cor}\label{cor:nonlimits_are_small}
	Let $G$ be an \lcsc group with $K/L$ a chief factor of $G$.  Suppose that one of the following holds:
\begin{enumerate}[(1)]
\item $K/L$ is connected, non-abelian and non-compact;
\item $K/L$ is elementary with $\xi(K/L) > \omega+1$;
\item $K/L$ is totally disconnected and non-elementary.
\end{enumerate}
Then $[K/L]$ is a regionally robust block of $G$, and $\mf{B}^{rr}_{K/L} \neq \emptyset$.
\end{cor}

\begin{proof}
By Proposition~\ref{prop:limit_blocks_of_blocks}, it is sufficient to show that $K/L$ has a regionally robust block.  Such a regionally robust block exists in all cases by Proposition~\ref{prop:min_covered_exist}.
\end{proof}

\subsection{A-simple groups}
For a topological group $G$ and $A\leq \Aut(G)$, we say that $G$ is \textbf{$A$-simple}\index{$A$-simple} if the only $A$-invariant closed normal subgroups of $G$ are $\{1\}$ and $G$.  A trichotomy exists for $A$-simple topological groups.

\begin{defn}\label{def:A-simple_types} 
Let $G$ be a topologically characteristically simple topological group. 
\begin{enumerate}
\item The group $G$ is of \defbold{weak type}\index{weak type} if $\mf{B}_G^{\min}=\emptyset$. 
\item The group $G$ is of \defbold{stacking type}\index{stacking type} if $\mf{B}_G^{\min}\neq \emptyset$ and for all $\mf{a}, \mf{b} \in \mf{B}_G^{\min}$, there exists $\psi \in \Aut(G)$ such that $\psi.\mf{a} < \mf{b}$.
\end{enumerate}
\end{defn}

\begin{thm}[{\cite[Theorem 9.16]{RW_P_15}}]\label{thm:trichotomy}
Suppose that $G$ is an $A$-simple topological group for some $A \le \Aut(G)$. Then exactly one of the following holds:
\begin{enumerate}[(1)]
\item The group $G$ is of weak type.
\item The group $G$ is of semisimple type and $A$ acts transitively on $\mf{B}_G$.
\item The group $G$ is of stacking type and for all $\mf{a}, \mf{b} \in \mf{B}_G^{\min}$, there exists $\psi \in A$ such that $\psi.\mf{a} < \mf{b}$. 
\end{enumerate}
\end{thm}

Using our results for regionally robust blocks, we obtain a stronger version of Theorem \ref{thm:trichotomy} for locally compact groups, by describing the groups that are neither semisimple nor stacking type.

\begin{thm}\label{thm:chief:block_structure}
	Suppose that $G$ is a topologically characteristically simple \lcsc group. Then $G$ is exactly one of the following:
\begin{enumerate}[(1)]
		\item an abelian group;
		\item a non-abelian quasi-discrete \tdlcsc group;
		\item a regionally SIN \tdlcsc group with trivial quasi-center;
		\item an elementary \tdlcsc group of decomposition rank $\omega+1$;
		\item a direct product of copies of a simple Lie group;
		\item a quasi-product of copies of a topologically simple \tdlcsc group of decomposition rank greater than $\omega+1$;
		\item a \tdlcsc group of stacking type and decomposition rank greater than $\omega+1$.
\end{enumerate}
\end{thm}

\begin{proof}
 We may assume that $G$ is non-abelian.  Since $G$ is topologically characteristically simple, the group $G$ is then topologically perfect and centerless. 

Plainly, $G$ is either connected or totally disconnected. Let us suppose first that $G$ is connected. By \cite[Theorem~10.25]{HMProLie}, $G$ has a unique largest pro-solvable connected normal subgroup, and since $G$ is topologically perfect, $G$ itself cannot be pro-solvable.  That $G$ is characteristically simple now ensures that $G$ is semisimple in the sense of \cite[Theorem~10.29]{HMProLie}.  Applying \cite[Theorem~10.29]{HMProLie}, we deduce that $G$ is a direct product of copies of a simple Lie group. Therefore, case (5) holds.  

We now suppose that $G$ is totally disconnected. We consider first the case that $\xi(G) > \omega+1$.  From Proposition~\ref{prop:min_covered_exist}, it follows that $G$ has a regionally robust block, so it is either of semisimple type or of stacking type via Theorem~\ref{thm:trichotomy}.   If $G$ is of semisimple type, then it is a quasi-product of copies of some topologically simple \tdlcsc group $H$ by Theorem~\ref{thm:min_normal:simple}.  If $H$ is elementary, Theorem~\ref{quasiproduct:elementary} ensures that $\xi(G)=\xi(H)$, so $\xi(H) > \omega+1$. If $H$ is non-elementary, then plainly $\xi(H)>\omega+1$. We thus deduce that exactly one of (6) and (7) holds. 

Let us finally consider the case that $\xi(G)\leq \omega+1$. By Corollary~\ref{cor:rank_char_simple}, $\xi(G)$ is either equal to $\omega+1$ or $2$.  If $\xi(G) = \omega+1$, then (4) holds.  If $\xi(G)=2$, then Lemma~\ref{lem:regionally_SIN} implies $G$ is a regionally SIN group.   If $G$ additionally has a trivial quasi-center, then (3) holds; otherwise, since $G$ is topologically characteristically simple, the quasi-center is dense and (2) holds.

We thus deduce that for any topologically characteristically simple \lcsc group one of (1)-(7) holds. That the conditions are exclusive we leave as an exercise.
\end{proof}

All of the cases (1)-(7) can occur. In particular, Example~\ref{ex:rank 2} gives an example of (3), and Example~\ref{ex:stacking} gives an example of (7). The techniques sketched in Section~\ref{sec:examples} can be used to produce \textit{elementary} chief factors which satisfy (7); these examples will appear in a future work. Examples satisfying (6) are easy to produce, but we do not know if (6) can occur for elementary groups.

\begin{cor}	
	Chief factors of \lcsc groups have one of the forms given in Theorem~\ref{thm:chief:block_structure}.  Moreover, if $K_1/L_1$ and $K_2/L_2$ are associated chief factors of an \lcsc group $G$, then $K_1/L_1$ and  $K_2/L_2$ have the same form.
\end{cor}
\begin{proof}
Since $G_1: = K_1/L_1$ and $G_2 := K_2/L_2$ are chief factors, they are topologically characteristically simple. They are thus each of one of the forms described in Theorem~\ref{thm:chief:block_structure}. 

Applying \cite[Lemma 6.6]{RW_P_15}, there is a third normal factor $K/L$ of $G$ such that $K/L$ is a normal compression of both $G_1$ and $G_2$. Hence, $G_1$ is abelian if and only if $G_2$ is abelian. Having dispensed with case (1), we may assume that both $G_1$ and $G_2$ are non-abelian. Cases (2), (3), (4), and (5) are now easy applications of Proposition~\ref{prop:block_properties}. For the cases (6) and (7), Proposition~\ref{prop:block_properties} ensures that $\xi(G_1) = \xi(G_2)$, and the trichotomy of Theorem~\ref{thm:trichotomy} is invariant under association via \cite[Proposition~9.19]{RW_P_15}. Therefore, $G_1$ satisfies (6) if and only if  $G_2$ satisfies (6), and $G_1$ satisfies (7) if and only if $G_2$ satisfies (7).
\end{proof}

The division into cases of Theorem~\ref{thm:chief:block_structure} is not exactly a refinement of Theorem~\ref{thm:trichotomy}, since elementary \tdlcsc groups of rank $\le \omega+1$ can be of semisimple or stacking type.  Here is an alternative version that emphasizes the role of weak type; after dividing into the three cases given by Theorem~\ref{thm:trichotomy}, the proof is essentially the same as that of Theorem~\ref{thm:chief:block_structure}.

\begin{thm}\label{thm:types}
Suppose that $G$ is an $A$-simple \lcsc group for some $A \le \Aut(G)$.  Then exactly one of the following holds:
\begin{enumerate}[(1)]
		\item The group $G$ is abelian.
		\item The group $G$ is of weak type and is a non-abelian \tdlcsc group that is quasi-discrete.
		\item The group $G$ is of weak type and is a regionally SIN \tdlcsc group with trivial quasi-center.
		\item The group $G$ is of weak type and is an elementary \tdlcsc group of decomposition rank $\omega+1$.
		\item The group $G$ is of semisimple type and $A$ acts transitively on $\mf{B}_G$.
		\item The group $G$ is a \tdlcsc group of stacking type and for all $\mf{a}, \mf{b} \in \mf{B}_G^{\min}$, there exists $\psi \in A$ such that $\psi.\mf{a} < \mf{b}$. 
\end{enumerate}
\end{thm}

There are weak type chief factors of rank exactly $\omega+1$; see Example~\ref{ex:weak}. Case (4) therefore cannot be omitted.

\begin{rmk}
Let us pause to consider the seven cases of Theorem~\ref{thm:chief:block_structure} with respect to the goal of decomposing chief factors into compact, discrete, or topologically simple groups.  This goal is clearly achieved in cases (5) and (6), and it is straightforward to show in case (1) that $G$ is discrete-by-compact in the connected case and compact-by-discrete in the totally disconnected case.  In cases (2) and (3), $G$ is a directed union of SIN \tdlcsc groups, so in particular a directed union of compact-by-discrete groups.  In case (4), $G$ is a directed union of compactly generated open subgroups, and each such subgroup admits a canonical finite series in which all factors are regionally SIN, via  Theorem~\ref{thm:lower_finite_rank}. Hence, $G$ is constructed from compact and discrete groups.


This leaves case (7). Stacking type ensures a rich subnormal subgroup structure, but it is unclear exactly how rich this structure can be. The techniques of Section~\ref{sec:examples} show that there is a great deal of freedom in a stacking type group. The most natural decomposition property to hope for in general thus seems to be the following well-foundedness condition:
\end{rmk}
\begin{quest}\label{qu:well-foundedness} 
	Suppose that $G=:G_0$ is a topologically characteristically simple \lcsc group. If $G_0$ is abelian, elementary with rank at most $\omega+1$, or of semisimple type, we stop. Else, we find $G_1:=K/L$ a non-abelian chief factor of $G_0$ which is robust. Continuing in this fashion produces $G_0,G_1,\dots$ a sequence of \lcsc groups. Is it the case that any such sequence halts in finitely many steps? What about in the case that the group $G$ is also elementary?
\end{quest}

\section{Examples}\label{sec:examples}

\subsection{Preliminaries}
We here describe a technique which produces compactly generated \tdlcsc groups with interesting chief factors. We only sketch the construction and recall the relevant facts; the full details will appear in a later article. The construction is inspired by ideas from \cite{B05,LB16,S14}, and the reader familiar with \cite{LB16} can likely prove all claimed facts below.

Let $T$ be the countably regular tree and fix $\delta$ an end of $T$. We orient the edges of $T$ such that all edges point toward the end $\delta$. The resulting directed graph is denoted $\Td$; we call $\delta$ the \textbf{distinguished end} of $\Td$. Given a countable set $X$, a \textbf{coloring} of $\Td$ is a function $c:E\Td\rightarrow X$ such that for each $v\in V\Td$,
\[
c_v:=c\rest_{\inn(v)}:\inn(v)\rightarrow X
\] 
is a bijection. The set $\inn(v)$ is the collection of directed edges with terminal vertex $v$. The coloring allows us to define the \textbf{local action} of $g\in \Aut(\Td)$ at $v\in V\Td$: 
\[
\sigma(g,v):=c_{g(v)}\circ g\circ c_v^{-1}\in \sym(X).
\]
We call the coloring \textbf{ended} if there is a monochromatic directed ray which is a representative of the distinguished end $\delta$; \textit{we shall always assume our colorings are ended.} 

The tree $\Td$ is also equipped with horoballs and horospheres centered at $\delta$, defined as follows.  There is a function $\eta$ from $V\Td$ to $\Zb$ such that, for each of the oriented edges of $e$ of $\Td$, we have $\eta(t(e)) = \eta(o(e))+1$. One can see that the function $\eta$ is unique up to an additive constant: it is determined by choosing a base point $v_0$ such that $\eta(v_0) = 0$.  The set $\{\eta\inv(i) \mid i \in \Zb\}$ of fibers is therefore uniquely determined.  The fibers are the \defbold{horospheres} centered at $\delta$, and the sets $X_i := \{v \in V\Td \mid \eta(v) \ge i\}$ for $i \in \Zb$ are the \defbold{horoballs} centered at $\delta$; see Figure~$1$.

\begin{figure}[h]\label{fig:X_n}
	\centering
	\includegraphics[width=0.8\textwidth]{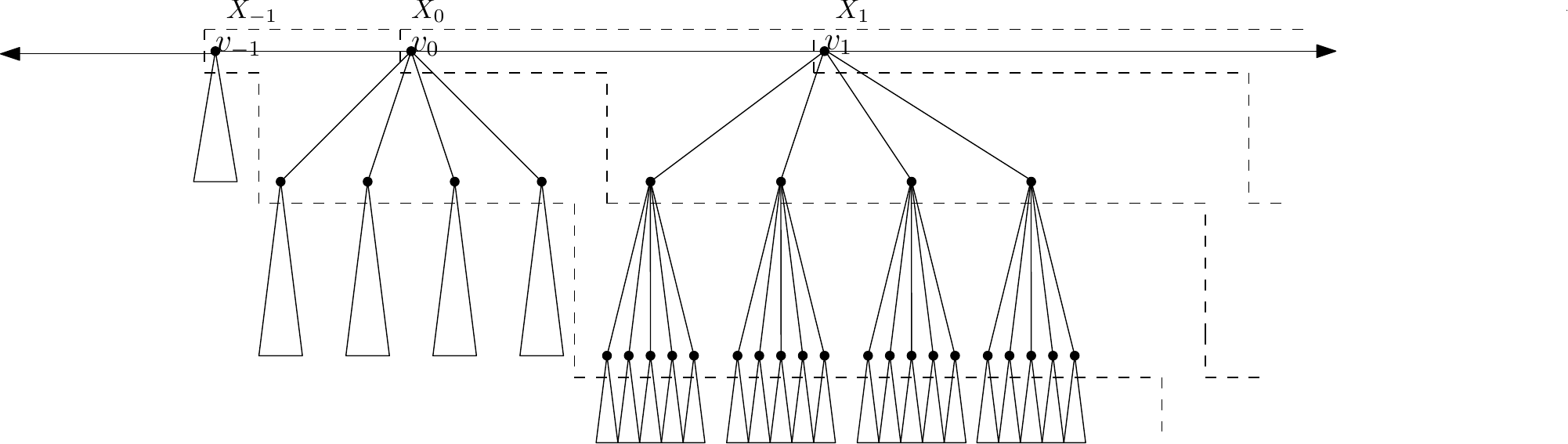}
	\caption{The horoballs $X_i$}
\end{figure}

To define the groups, it will be convenient to make a definition: A \textbf{\tdlcsc permutation group} is a pair $(G,X)$ where $G$ is a \tdlcsc group and $X$ is a countably infinite set on which $G$ acts faithfully with compact open point stabilizers. 

\begin{defn}
	Suppose that $(G,X)$ is a \tdlcsc permutation group with $U\in \U(G)$ and color the tree $\Td$ by $X$. We define the group $E_X(G,U)\leq \Aut(\Td)$ as follows: $E_X(G,U)$ is the set of $g\in \Aut(\Td)$ such that $\sigma(g,v)\in G$ for all $v\in V\Td$ and that $\sigma(g,v)\in U$ for all but finitely many $v\in V\Td$. 
\end{defn}

It is easy to verify that $E_X(G,U)$ is an abstract group. With more care, one can identify a natural \tdlcsc topology on $E_X(G,U)$. For a vertex $v\in V\Td$, it is an easy exercise to see that the point stabilizer $E_X(U,U)_{(v)}$ is compact in the topology inherited from $\Aut(\Td)$. The topology on $\Aut(\Td)$ is that of pointwise convergence; this topology is Polish but not locally compact.

\begin{prop}
	For $(G,X)$ a \tdlcsc permutation group and $U\in \U(G)$, there is a \tdlcsc group topology on $E_X(G,U)$ such that the inclusion $E_X(U,U)_{(v)}\injects E_X(G,U)$ is continuous and open for any $v\in V\Td$.
\end{prop}

The \tdlcsc permutation group $(G,X)$ controls the \tdlcsc group $E_X(G,U)$ in several important ways.

\begin{prop}\label{prop:compact_generation} Suppose that $(G,X)$ is a \tdlcsc permutation group. Then,
	\begin{enumerate}[(1)]
		\item If $(G,X)$ is transitive and $G$ is compactly generated, then $E_X(G,U)$ is compactly generated for any $U\in \U(G)$.
		\item If $G$ is elementary, then $E_X(G,U)$ is elementary for any $U\in \U(G)$.
	\end{enumerate}
\end{prop}

The group $E_X(G,U)$ admits a canonical open normal subgroup. Let $P_X(G,U)$ denote the set of $g\in E_X(G,U)$ such that $g$ fixes a vertex; equivalently, $P_X(G,U)$ is the set of $g \in E_X(G,U)$ such that $\eta(gv) = \eta(v)$ for some (equivalently, any) $v \in V\Td$. It is easy to verify that $P_X(G,U)$ is an open, normal subgroup of $E_X(G,U)$ and that $E_X(G,U)/P_X(G,U)\simeq \Zb$. The subgroup $P_X(G,U)$ also gives a chief factor. Recall that a topological group is \textbf{monolithic} if the intersection of all non-trivial closed normal subgroups is non-trivial.

\begin{prop}\label{prop:elliptic_sgrp}
	Suppose that $(G,X)$ is a transitive \tdlcsc permutation group. Then, $E_X(G,U)$ is monolithic, and the monolith is
	\[
	M:=\ol{[P_X(G,U),P_X(G,U)]}.
	\]
	If $G$ is also topologically perfect, then $P_X(G,U)$ is the monolith.
\end{prop}

%

The pointwise stabilizer $H_i:=\Stab_L(X_i)$ of a horoball $X_i$, called a \textbf{horoball stabilizer}, is a closed normal subgroup of $L$, and $H_i\leq H_j$ for $i\leq j$. We thus obtain canonical normal factors $H_{i+1}/H_i$ of $P_X(G,U)$. We call these factors the \textbf{horoball factors} of $P_X(G,U)$. 

\begin{prop}\label{prop:horoball_stab} 
For $(G,X)$ a \tdlcsc permutation group, each horoball factor $K/L$ admits a $P_X(G,U)$-equivariant isomorphism $\psi:K/L\rightarrow \bigoplus_{\partial X}(G,U)$ where $X$ is the horoball corresponding to $K$.
\end{prop}

\subsection{Stacking type chief factors}\label{ex:stacking}
We now exhibit an example of a non-elementary stacking type chief factor in a compactly generated \tdlcsc group. Again technical results are stated without proof.

Fix $G$ a non-discrete compactly generated simple \tdlcsc group and $U\in \U(G)$; for concreteness, one can take $PSL_2(\Qp)=G$ and $PSL_2(\Zb_p)=U$. Set $X:=G/U$ and let $G\acts X$ by left multiplication. The resulting permutation group $(G,X)$ is a transitive \tdlcsc permutation group, so via Proposition~\ref{prop:compact_generation}, the group $E_X(G,U)$ is a compactly generated \tdlcsc group. In view of Proposition~\ref{prop:elliptic_sgrp}, $P_X(G,U)$ is the unique chief factor of $E_X(G,U)$. In this setting, $P_X(G,U)$ is of stacking type via the following:

\begin{prop}\label{prop:stacking}
	Suppose that $(G,X)$ is a transitive permutation \tdlcsc group and $U\in \U(G)$ is non-trivial. If $G$ is topologically simple and compactly generated, then $P_X(G,U)$ is a stacking type chief factor of $E_X(G,U)$, and every chief factor of $P_X(G,U)$ is a horoball factor.
\end{prop}

We can also describe the chief factors of $P_X(G,U)$ explicitly. Via Propositions~\ref{prop:horoball_stab} and \ref{prop:stacking}, every chief factor of $P_X(G,U)$ is isomorphic to $\bigoplus_{\Nb}(G,U)$.

The group $P_X(G,U)$ is thus a non-elementary stacking type chief factor, giving an example of Theorem~\ref{thm:chief:block_structure}(7).

\subsection{Chief factors of weak type}\label{ex:weak}
We apply our construction a second time to exhibit rank $\omega+1$ weak type chief factors of a compactly generated \tdlcsc group. Our construction requires an observation on the rank of the groups $E_X(G,U)$.

\begin{prop}\label{prop:rank}
	Suppose that $(G,X)$ is a compactly generated transitive \tdlcsc permutation group. If $G$ is elementary with finite rank, then $E_X(G,U)$ is elementary with 
	\[
	\xi(E_X(G,U))=\omega+2
	\]
	for any non-trivial $U\in \U(G)$.
\end{prop}

Let $D$ be the infinite dihedral group and $C$ be the order two subgroup which reverses orientation. Setting $X:=D/C$, we obtain a \tdlcsc permutation group $(D,X)$. We show that $G:=E_X(D,C)$ has a chief factor of weak type with rank $\omega+1$. 

By Proposition~\ref{prop:rank}, $G$ has rank $\omega+2$. Applying Proposition~\ref{prop:elliptic_sgrp}, $G$ is monolithic, and the monolith $M$ is such that $G/M$ is metabelian. Since solvable groups are finite rank, Lemma~\ref{lem:d-rank_extensions} and Corollary~\ref{cor:rank_char_simple} ensure that $\xi(M)=\omega+1$. We now argue that $M$ is of weak type.

\begin{prop} 
	$\mf{B}_M^{\min}=\emptyset$.
\end{prop}
\begin{proof}
Let $A/B$ be a chief factor of $M$ and suppose for contradiction that $[A/B]\in \mf{B}_M^{\min}$. Fixing a base vertex $v=:v_0$, we form the horoballs $X_i$ and set $H_i$ to be the horoball stabilizer in $P_X(G,U)$. 
	
 The union $\bigcup_{i\in \Zb}H_i$ is a normal subgroup of $E_X(G,U)$, and it is dense in $P_X(G,U)$. Since $M$ is non-abelian and chief, it follows that $\bigcup_{i\in \Zb}H_i\cap M$ is dense and normal in $M$. If $A\cap H_i\leq B$ for all $i$, then $A/B$ commutes with $\bigcup_{i\in \Zb}H_i\cap M$ in $M/B$. It follows that $A/B$ is abelian, which contradicts our hypothesis. We thus deduce that $A\cap H_i\nleq B$ for some $i$, so $A\cap H_i$ covers the chief block $[A/B]$. The intersection of all $H_i$ is trivial, hence since $[A/B]$ is minimally covered, we may take $A\cap H_i$ covering $[A/B]$ such that $A/\cap H_{i-1}$ does not cover $[A/B]$. The normal factor $A\cap H_i/A\cap H_{i-1}$ thus covers the chief block $[A/B]$. 
	
The normal factor $A\cap H_i/A\cap H_{i-1}$ admits an injection into $H_i/H_{i-1}$. Proposition~\ref{prop:horoball_stab} implies that $H_i/H_{i-1}\simeq \bigoplus_{\Nb}(D,C)$, and therefore, $H_i/H_{i-1}$ is metabelian. The factor $A\cap H_i/A\cap H_{i-1}$ is then also metabelian. We may now find $K/L\in [A/B]$ such that $A\cap H_{i-1}\leq L<K\leq A\cap H_i$, and since $K/L$ is chief, it follows that $K/L$ is abelian. This is absurd as we assume $K/L$ is a non-abelian chief factor.
\end{proof}

The group $M$ is thus a rank $\omega+1$ chief factor of weak type, giving an example of Theorem~\ref{thm:types}(4). 

\subsection{Rank two chief factors}\label{ex:rank 2}
Let $U^+(A_6)$ be the Burger--Mozes universal simple group for the alternating group on $6$ elements with the usual permutation representation; see \cite{BM00,CdM11} for a discussion of these groups. The group $U^+(A_6)$ acts on the $6$-regular rooted tree $T_6$. Fixing $\delta$ an end of $T_6$, we consider the end stabilizer $G:=U^+(A_6)_{(\delta)}$.

Fix a base vertex $v\in T_6$ and let $v=v_0,\dots,v_n,\dots$ be a geodesic ray representing the end $\delta$. Fix $t\in G$ a translation down the geodesic ray and let $E\normal G$ be the collection of elliptic isometries --- i.e.\ the isometries which fix a vertex. It is an exercise to see that $E$ is a locally elliptic closed normal subgroup and $G/E\simeq \Zb$.

Take the edge $e=(v_0,v_1)$, where $v_0$ and $v_1$ are the first two elements of our distinguished ray, and define $T_{v_1}$ to be the half-tree determined by $e$ which contains $v_1$. It now follows that 
\[
E=\ol{\bigcup_{n\geq 1}t^n\Stab_{G}(T_{v_1})t^{-n}}.
\]
The stabilizer $\Stab_{G}(T_{v_1})$ is the profinite group given by taking the inverse limit of the iterated wreath products of $(A_5,[5])$. We thus deduce that $\Stab_{G}(T_{v_1})$ is a topologically perfect profinite group, and it has trivial quasi-center. It follows further that $E$ has trivial quasi-center.

\begin{lem}\label{lem:rank 2} 
	The subgroup $E$ is a non-negligible chief factor of $G$ with rank two.
\end{lem}
\begin{proof} Suppose that $N\normal G$ is non-trivial, closed, and contained in $E$. Fixing $n\in N$ non-trivial, there is some nearest vertex $w\in T_6$ to the ray $v_0,v_1\dots $ that is moved by $n$. By conjugating $n$ with $t^{\pm 1}$ and elements of $E$, we may assume that $w=v_0$ and $n$ fixes $v_1$.
	
Let $e$ be the edge from $v_0$ to $v_1$ and let $T_{v_1}$ be the half-tree determined by $e$ that contains $v_1$. We now apply a familiar trick: for any $g,h\in \Stab_G(T_{v_1})$, the commutator $[[n^{-1},g^{-1}],h]$ equals $[g,h]\in \Stab_G(T_{v_1})$. The subgroup $N$ thus contains all commutators of $\Stab_G(T_{v_1})$. The stabilizer $\Stab_G(T_{v_1})$ is topologically perfect, so $N$ contains $\Stab_G(T_{v_1})$. Since $N$ is normal, $N$ indeed contains $t^n\Stab_G(T_{v_1})t^{-n}$ for all $n$. We thus deduce that $N=E$, so $E$ is a chief factor of $G$.
	
The group $E$ is locally elliptic, hence $\xi(E)=2$. Applying Theorem~\ref{thm:negligible_types}, the chief block $[E]$ is either connected compact, quasi-discrete, or robust. Since $G$ is totally disconnected, $[E]$ is not connected compact. The representative $E$ is not quasi-discrete, so $[E]$ is also not quasi-discrete. We thus deduce that $[E]$ is a robust chief block. A fortiori, $E$ is a non-negligible chief factor of $G$. 
\end{proof}

The group $E$ is thus a rank $2$ chief factor with trivial quasi-center, giving an example of Theorem~\ref{thm:chief:block_structure}(3). 

\printindex
\bibliographystyle{amsplain}
\bibliography{biblio}{}

\end{document}